\documentclass[12pt,a4paper,intlimits,sumlimits]{amsart}

\usepackage{amsmath, amsthm, mathtools}
\usepackage{amsfonts}
\usepackage[alphabetic,nobysame,initials]{amsrefs}
\usepackage{graphicx}
\usepackage{framed}
\usepackage{amssymb}
\usepackage{esvect}
\usepackage{xcolor}
\usepackage{eucal}
\usepackage{mathrsfs}
\usepackage{hyperref}
\usepackage[english]{babel}
\usepackage{mathrsfs}
\usepackage{esint}
\usepackage{accents}
\usepackage{tabularx}
\usepackage{booktabs}

\def \R {\mathbb{R}}

\theoremstyle{definition}
\newtheorem{definition}{Definition}[section]
\newtheorem{example}[definition]{Example}
\newtheorem{remark}[definition]{Remark}

\theoremstyle{plain}
\newtheorem{theorem}[definition]{Theorem}
\newtheorem{proposition}[definition]{Proposition}
\newtheorem{lemma}[definition]{Lemma}

\numberwithin{equation}{section}

\usepackage{geometry}
\renewcommand{\epsilon}{\varepsilon}

\renewcommand{\leq}{\leqslant}
\renewcommand{\le}{\leqslant}
\renewcommand{\geq}{\geqslant}
\renewcommand{\ge}{\geqslant}

\allowdisplaybreaks

\title[Free boundary analysis for tumor invasion]{A free boundary analysis \\ of tumor invasion driven by angiogenesis}
 
\author[S. Dipierro, E. Proietti Lippi, C. Sportelli and E. Valdinoci]{Serena Dipierro, Edoardo Proietti Lippi, \\Caterina Sportelli and Enrico Valdinoci}

\address{Serena Dipierro: Department of Mathematics and Statistics, The University of Western Australia, 35 Stirling Highway, Crawley, Perth, WA 6009, Australia}
\email{serena.dipierro@uwa.edu.au}

\address{Edoardo Proietti Lippi: Department of Mathematics and Statistics, The University of Western Australia, 35 Stirling Highway, Crawley, Perth, WA 6009, Australia}
\email{edoardo.proiettilippi@uwa.edu.au}

\address{Caterina Sportelli: Departamento de Análisis Matemático, Universidad de Granada, 18071 Granada, Spain  \& Department of Mathematics and Statistics, The University of Western Australia, 35 Stirling Highway, Crawley, Perth, WA 6009, Australia}
\email{caterina.sportelli@uwa.edu.au, caterina.sp@ugr.es}

\address{Enrico Valdinoci: Department of Mathematics and Statistics, The University of Western Australia, 35 Stirling Highway, Crawley, Perth, WA 6009, Australia}
\email{enrico.valdinoci@uwa.edu.au}

 \date{}

\begin{document}

\begin{abstract}
We discuss a free boundary model for tumor invasion that describes a cloud of cells that diffuse and, at the same time, are drifted along the vector field of the chemotactic direction.  The model captures the evolution of a solid tumor, including the process of angiogenesis, which consists in the formation of new blood vessels that supply the tumor with oxygen and other nutrients, thereby promoting its spread and growth.

We prove that, once formed,  the tumor survives through time, maintaining strictly positive thickness. An explicit expression in terms of the initial data is derived. 

Moreover, we distinguish two regimes depending on the ratio $\kappa$ between the spreading of tumor cells and the growth of the tumor mass. If $\kappa$ is sufficiently large, then the tumor grows exponentially in time and invades the entire host tissue. In contrast, if $\kappa$ is small enough, then either the tumor remains bounded in size over time or may experience a fast contraction.
 \end{abstract}
 \maketitle

\tableofcontents
  
\section{Introduction} \label{Intro}
\subsection{A general model}\label{INT1.1}
In this paper, we investigate the evolution of a tumor invasion.
Many models have been proposed for this scope;
the classical literature in the field of tumor dynamics heavily relies on reaction-diffusion processes (see e.g.~\cite{MR1684873, MR1920634, MR3401608}) and
various settings are typically described by different equations
(see e.g.~\cite[equations~(3.1a) and~(5.1)]{ARM06},
\cite[equations~(1) and~(2)]{MR3022677.13}, \cite[equation~(1.1)]{MR3260704.14},
\cite[equations~(1.1) and~(1.2)]{2020nonlocalcelladhesionmodels},
\cite[equation~(6.3)]{AC22}, \cite[equation~(37.2)]{MRELE}, and~\cite[equations~(2) and~(3)]{abderrahman2025optimalcontrolmedicaldrug25}).

In this paper, we  consider a model of the form
\begin{equation}\label{dvtECNhajkcldkm}
u_t(x, t) = \Delta u(x, t) + \chi(c) \nabla c(x, t)\cdot\nabla u(x, t), 
\end{equation}
in which a radially directed advection (chemotaxis) is superposed to a pure diffusion term~$\Delta u$. 
We now describe in detail the quantities involved
in equation~\eqref{dvtECNhajkcldkm}.

Chemotaxis is the phenomenon by which the movement of cells is directed in response to an extracellular chemical gradient. In the specific context of cancer, chemotaxis pathways can be reprogrammed in favour of tumor cell dissemination.  For example, to successfully metastasize, a carcinoma cell must invade, intravasate, extravasate and grow at a distant site. Chemotaxis is thought to be involved in each of these crucial steps of tumor cell dissemination (see e.g.~\cite{PMC4030706}).

In the framework of the tumor models, chemoattractants are signaling molecules with the role to direct the motion of cancer (or immune) cells toward areas where their concentration is higher.
In the analysis of tumor models,  one of the mostly employed chemoattractants is the Vascular Endothelial Growth Factor (usually shortened as VEGF). It is a signaling protein that plays a key role in the process of angiogenesis and behaves as a chemoattractant for endothelial cells, guiding them toward the tumor with the aim to form new blood vessels. 
In vivo studies of chemotaxis have shown that tumor cells can respond to shallow gradients generated from various devices designed to deliver known amounts of VEGF and other chemoattractants (see~\cite[page 12]{PMC4030706} and the references therein).

Throughout the paper, we will denote by $c(x, t)$ the concentration of a chemoattractant (e.g. VEGF), that is the spatial distribution and amount of chemoattractant present in the tumor environment at the time $t$. As is well known, tumor cells and/or endothelial cells are capable of detecting and moving toward higher concentrations of these molecules, with the migration guided by the concentration gradient. Also, the concentration of chemoattractants in the tumor microenvironment often depends on time and it changes at different stages of the tumor growth (see e.g.~\cite{CHAP98, MR3793187} and the references therein).

In addition, we denote by $\chi=\chi(c)$ the chemotactic sensitivity function,  which plays the role of tracking how strongly cells respond to the VEGF gradient.  A high chemotactic sensitivity suggests that the tumor cells are highly responsive to the chemical gradients. This leads to a more aggressive migration and to a possibly faster tumor growth. Instead,  a lower sensitivity indicates a slower directional migration.

We point out that various functional forms have been proposed for $\chi(c)$. For a positive constant $\chi_0$, these include a constant law
\[
\chi(c):=\chi_0,
\]
a receptor kinetic law
\[
\chi(c):=\frac{\chi_0\, k}{(k+c)^2},
\] for a given~$k\in(0,+\infty)$,
and a logarithmic\footnote{The name ``logarithmic'' that the literature attached to~\eqref{chilog}
is perhaps confusing, since it could be more appropriate to call it
an ``inverse law'' or ``hyperbolic law''. The reason it is often called logarithmic is likely due to the fact that
when one computes total quantities, one obtains a logarithmic expression via integration.
We kept here the standard nomenclature on the topic, as in~\cite[page~155]{ChapSt93}.} law
\begin{equation}\label{chilog}
\chi(c):=\frac{\chi_0}{c},
\end{equation}
see e.g.~\cite{ChapSt93}. 

Having a constant chemotactic sensitivity function means that all the cells respond equally strongly to the gradient, regardless of the concentration of the chemoattractant. This choice is typical when introducing simplest models.

Instead,  a receptor kinetic law is used for the chemotactic sensitivity function to reflect the assumption that chemotactic sensitivity decreases
with increased VEGF concentration (see e.g.~\cite{CHAP98}).

On the other hand, the logarithmic expression in~\eqref{chilog} contributes to the implementation of a realistic scenario,
since many cells and microorganisms respond approximately to relative changes in chemical concentration rather than absolute changes: indeed, \eqref{chilog} indicates that the cell chemotactic response to VEGF follows the Weber--Fechner's law, which has several applications in biological modelings (see e.g.~\cite{KSBACTERIA} and the references therein).
Accordingly,  in our model, we will choose $\chi(c)$ as in~\eqref{chilog}.

As described in \cite{greenspan}, when a tumor is small, nutrients can reach every cell, allowing for rapid growth. As the tumor grows, nutrient levels drop toward the center. Eventually,  a vital nutrient falls below the threshold needed for cell survival, forming a central area called in jargon necrotic core. This leads to a noticeable slowdown in the tumor’s overall growth rate, as it becomes more difficult for cells to take in nutrients and remove waste by diffusion alone.  The typical steady-state structure at this point is a tiny sphere made up of three concentric zones. Namely, cells are still actively growing and dividing in the thin outer layer. In the intermediate zone, cells remain alive but show little mitotic activity.  Finally, the innermost core is composed of necrotic material at various stages of breakdown.

For the purposes of this model,  we assume that the tumor consists of two distinct regions, an outer crust in which all cells are proliferating and an inner region in which no cells are proliferating (there is no
need in this derivation to distinguish between the mantle and the necrotic core,
both of which are nonproliferative).  For simplicity, we assume that the non-proliferative region (referred to here as the necrotic core) consists
of a ball of given radius $\varepsilon>0$,
and that the cell density on the tumor boundary is constant,
say equal to some~$\overline{u}\in(0,+\infty)$.
Moreover, we assume that the necrotic core radius is fixed (this is a fairly standard assumption in the literature, see e.g. \cite{LU2023}), with the radius of the spheroid, whose evolution in time represents the growth of the tumor, governing the overall expansion.
 
We let $\Omega(t)$ denote the region occupied by the tumor at time $t$.
This region will include the necrotic core~$B_\varepsilon$
and our aim is to describe the evolution of the cell concentration
in the ``active tumor region''~$\Omega(t)\setminus B_\varepsilon$.
This corresponds to a mathematical model given by
\begin{equation}\label{maingen1}
\begin{cases}
u_t (x,t)= \Delta u(x, t) +\chi(c) \nabla c(x, t)\cdot\nabla u(x, t) &\quad\mbox{ in } \Omega(t)\setminus B_\varepsilon,\\
u(x, t)=\overline{u} &\quad\mbox{ on } \partial B_\varepsilon,\\
u(x, t)=\underline{u} &\quad\mbox{ on } \partial\Omega(t)\setminus\partial B_\varepsilon,\\
\end{cases}
\end{equation} for some~$\overline{u}$, $\underline{u} \in(0,+\infty)$.

We will suppose that~$\overline{u}>\underline{u}$ and, in the medical application,
the density~$\underline{u}$ represents the density threshold
for the survival of tumor cells.
Indeed, low cell density poses severe survival challenges for cells, whereas cooperative chemical signaling and growth factors promote cell survival and proliferation at higher densities. See~\cite{qwjudfbn23o4t52022, MR4648489, COOPETU25} for experimental evidence of this phenomenon in tumors.

In tumor models, cell density is often related to the density of the nutrients.
For example, in~\cite{MR2237681}
the density of the cells  is first supposed to depend on the concentration of nutrients, and thus, assuming that this dependence is linear, the two densities are identified. In this spirit, from the medical perspective, the above model posits that the tumor mass occupies the spatial region
in which the nutrient is sufficiently abundant for the tumor cells to be active and possibly proliferate.

There is, however, an important caveat to keep in mind when identifying cell and nutrient densities. While this identification is reasonable at the macroscopic scale or when describing steady states, greater care is required when considering the corresponding evolution equations. Tumor cells are active agents that can migrate toward higher concentrations of chemical signals through chemotaxis. In contrast, nutrients such as oxygen and glucose are passive chemical species that spread by diffusion and lack the biological machinery needed to sense or actively move along chemical gradients. Consequently, chemotactic terms are typically included in the evolution equations for tumor cells rather than for nutrient concentrations
(and, for this reason, in this paper we will focus on tumor cell density, rather than nutrient density).

It is also true, however, that nutrients and chemoattractants should not be viewed as completely independent chemical species within the tissue. In practice, both are transported not only by diffusion but also by advection due to the interstitial fluid flow, namely the movement of fluid through the extracellular matrix between blood and lymphatic vessels. This flow provides an important transport mechanism for dissolved molecules, including large proteins that diffuse only slowly.
Moreover, interstitial fluid flow plays an important role in shaping the tumor immune microenvironment and has been shown to promote cancer cell invasion in several tumor types. Such invasion constitutes one mechanism by which tumors can evade therapy and subsequently recur (see, e.g., \cite{PMC4144982}).

Consequently, when interstitial flow is taken into account, it advects both nutrients and chemoattractants along the underlying velocity field. By redistributing these chemical cues, the flow can either enhance or impede the migratory behavior of cancer cells during tumor invasion, depending on the direction and magnitude of the flow relative to the chemotactic gradients.
Therefore,
although nutrients are not chemotactic, their evolution is indirectly affected by chemotaxis through the redistribution of the tumor cell population, which modifies nutrient consumption rates.

{F}rom the mathematical standpoint, problem~\eqref{maingen1} is structured as a particular version of the Kolmogorov equation (see e.g.~\cite{MR2597943}), namely as the non‑divergence form of a diffusion process with velocity field given by~$ \chi(c) \nabla c$.
In this context, the term~$\chi(c) \nabla c\cdot\nabla$ describes how cell density exchanges occur in the direction where the concentration of the chemoattractant is higher, namely the tumor cells show a tendency to move toward regions with more of the chemical signal. 
At each point $x$, the vector $\nabla c(x,t)$ points ``uphill'' in chemoattractant concentration, namely it points toward locations where the chemoattractant is more concentrated (toward the tumor core).

We can distinguish different scenarios. If the gradient~$\nabla u$ share the same orientation as~$\nabla c$, then the concentration of the tumor cells increases in the same direction as the chemoattractant gradient.  Thus, the growth of~$u$ in that region tends to increase
because the advection term is positive. Hence, tumor cells spread over the area to migrate following the VEGF gradient.
If instead the gradient~$\nabla u$ points away from~$\nabla c$, then we have that the concentration of cells is lower in the direction of increasing VEGF.
Since the advection term turns to be negative, it helps accumulate the cells away from that specific region. Finally, when the gradient~$\nabla u$ is orthogonal to~$\nabla c$,  the change in the concentration of cells occurs in a direction which is not affected by the chemical gradient. This suggests that the cell exchanges between different regions are governed only by a diffusion process and the migration term has no effect in that specific direction.

We point out that problem~\eqref{maingen1} is of independent interest and can be suitably tailored to describe the evolution of different tumor models (e.g. immune system interaction models, metastasis models, crowding and competition models, hypoxia-driven models, metabolic models).

\subsection{Derivation of a model for tumor with angiogenesis}
We now focus on our main model of interest, namely we address the case in which problem~\eqref{maingen1} is ``completed" to establish a model that outlines the evolution process of a solid tumor with angiogenesis. In the rest of the paper we will focus on the physically meaningful dimensional case~$N=3$.

Since cancer cells proliferate abnormally fast, they require more oxygen and
other nutrients than the normal capillary system can provide.  For this reason, cancer cells initiate
a process, called angiogenesis, that induces blood vessels to sprout capillary tips which migrate toward and penetrate into the tumor, thus providing it with circulating blood supply.
However, angiogenesis not only facilitates tumor growth by supplying nutrients and oxygen to the tumor, but is also a prerequisite for tumor cell dissemination and metastasis.  Indeed, blood vessel density is correlated with a higher incidence of metastasis and a more rapid recurrence of disease (see~\cite{PMC4030706}).

For the sake of simplicity, we assume that the tumor is spherically symmetric, namely, for every~$t\ge0$,
\begin{equation}\label{Omegapalla}
\Omega(t)=B_{s(t)}
\end{equation}
for some~$s(t)>\varepsilon$.
Correspondingly,
the boundary of the tumor is given by the sphere~$\partial B_{s(t)}$.

Also, here we consider a steady state approximation of the concentration of chemoattractants, meaning that $c(x, t)$ is assumed to have reached a balance between production, diffusion, and consumption.  Hence, the time dependence in~$c$ is dropped.

Furthermore, we assume that the concentration of the chemoattractant~$c$ is radially symmetric. When we consider radial functions,
from now on, with a slight abuse of notation, we will use~$c(r)$
to denote the spherically symmetric
function~$c(x)$ when~$|x|=r$ (and a similar notation
will be utilized for all radial functions later on).
As a result, we write
\begin{equation}\label{c-radiale}
c=c(r).
\end{equation}

Under these assumptions, the main free boundary problem of interest for us takes the form
\begin{equation}\label{main1}
\begin{cases}
\kappa u_t(x,t) = \Delta u(x, t) +\chi(c) \nabla c(x)\cdot\nabla u(x, t) + \Gamma(u_B - u(x, t)) -\lambda_0 u(x, t) &\;\mbox{ in } B_{s(t)}\setminus\overline{B_{\varepsilon}},\\
u(x, t) = \underline{u} &\;\mbox{ on } \partial B_{\varepsilon}, \\
u(x, t) = \overline{u} &\;\mbox{ on } \partial B_{s(t)}, \\
u(x, 0) = u_0(|x|) &\;\mbox{ in } B_{s(0)}\setminus\overline{B_{\varepsilon}},
\end{cases}
\end{equation}
where $\kappa\in(0,+\infty)$, $s(0)\in(\varepsilon,+\infty)$ and~$u_0
\in C^1([\varepsilon,s(0)])$. 

The positive constant~$\lambda_0$ accounts for tumor cell loss due to processes such as
necrosis and
lack of resources and thus can be considered a death rate.

The positive constant~$\Gamma$ is a relaxation rate constant
that determines how quickly the tumor cell density is restored toward the reference value~$u_B$.
In this regard,
the positive constant~$u_B$ represents a ``preferred bulk tumor cell density'', namely a
baseline tumor cell density maintained by local tissue conditions, accounting for a balance between cell proliferation and loss mechanisms not explicitly modeled.
Note that, according to~\eqref{main1},
the tumor cell population has a tendency to remain close to the equilibrium density~$u_B$
(a larger~$\Gamma$ means a faster adjustment toward~$u_B$,
whereas a smaller~$\Gamma$ corresponds to a slower replenishment process).

In this paper,
we assume that, at each time $t$, there are no volumetric sources or sinks for the chemoattracant in the region $\{\varepsilon < r <s(t)\}$. Hence, the concentration of the chemoattractant satisfies the equation
\begin{equation}
\Delta c =0 \quad\mbox{ in $B_{s(t)}\setminus\overline{B_\varepsilon}$,  for all~$ t>0$.}
\end{equation}
Thus, in light of~\eqref{c-radiale}, 
\begin{equation}
c(r) = \frac{A}{r} + B \quad\mbox{ in $B_{s(t)}\setminus\overline{B_\varepsilon}$, for all~$ t>0$,}
\end{equation}
for some~$A$, $B\in\R$.  

We will also assume that
the chemoattractant density decays away from the region of interest considered in this paper,
namely that
\[
\lim_{r\to +\infty} c(r)=0,
\] and this gives~$B=0$.

Moreover, since the function $c=c(r)$ is supposed to be continuous at~$r=\varepsilon$, we find that~$A=c(\epsilon)\epsilon$.
These considerations give that
\begin{equation}\label{cfin}
c(r) = \frac{c(\varepsilon)\varepsilon}{r} \quad\mbox{ in $B_{s(t)}\setminus\overline{B_\varepsilon}$, for all~$ t>0$.}
\end{equation}

As previously mentioned, throughout this paper we assume that the chemotactic sensitivity function is of the form~\eqref{chilog}.
Thus, from~\eqref{chilog} and~\eqref{cfin}, we deduce that
\begin{equation}\label{evcpbioòdnc}
\chi(c) \nabla c(x)\cdot\nabla u(x, t) = -\frac{\chi_0}{r^2}\, x\cdot\nabla u(x, t) \quad\mbox{ in $B_{s(t)}\setminus\overline{B_\varepsilon}$, for all~$ t>0$.}
\end{equation}

We also suppose that~$u_0(r)$ is bounded below by $\frac{\Gamma u_B}{\lambda}$ and above by~$\overline{u}$ in~$\Omega(0)\setminus B_\varepsilon$, namely
\begin{equation}\label{u_0positiva}
0<\frac{\Gamma u_B}{\lambda}\le u_0(r)\le\overline{u}, \qquad \mbox{if }\,\,\varepsilon\leq r\leq s(0).
\end{equation}
We point out that, in general, \eqref{u_0positiva} can be relaxed to
\[
0< u_0(r)\le\overline{u}, \qquad \mbox{if }\,\,\varepsilon\leq r\leq s(0).
\]
However, the lower bound $\frac{\Gamma u_B}{\lambda}$ is useful, in our setting, to establish the 
forthcoming Lemma~\ref{lemma0<u<utilde}.
This lower bound corresponds, from a biological point of view, to a regime in which blood–tissue cell transfer is small compared with the death rate in a dense tumor cell population. 

Moreover, we set
$$ \lambda:=\Gamma+\lambda_0.$$
Accordingly, in this framework we deduce that the free boundary problem~\eqref{main1} reduces to
\begin{equation}\label{mainfinal1}
\begin{cases}
\kappa u_t (x,t)= \Delta u(x, t) - \displaystyle\frac{\chi_0}{r^2}\, x\cdot\nabla u(x, t) -\lambda u(x, t)+ \Gamma u_B 
&\quad\mbox{ in }B_{s(t)}\setminus\overline{B_\varepsilon},\\
u(x, t) = \underline{u}, &\quad\mbox{ on } \partial B_\varepsilon \\
u(x, t) = \overline{u} &\quad\mbox{ on } \partial B_{s(t)}, \\
u(x, 0) = u_0(|x|) &\quad\mbox{ in } B_{s(0)}\setminus B_\varepsilon.
\end{cases}
\end{equation}
In this setting, in
the region~$\Omega(t)$
the concentration~$u$ of the tumor cells remains above the surviving
threshold~$\frac{\Gamma u_B}{\lambda}$
(see Lemma \ref{lemma0<u<utilde}
below which indeed confirms that the density~$u$,
solving~\eqref{mainfinal1}
remains above the threshold~$\frac{\Gamma u_B}{\lambda}$
in the region~$\Omega(t)$).

We point out that the solution~$u$ of~\eqref{mainfinal1} is rotationally symmetric (this can be checked by considering a rotation~${\mathcal{R}}$ and the function~$u_{{\mathcal{R}}}(x):=u({\mathcal{R}}x)$, which is also a solution due to the spherical symmetry of~$\Omega(t)$ and~$c$, then, by uniqueness of solutions, we have that~$u=u_{{\mathcal{R}}}$, see~\cite[Corollary~5.2]{MR1465184}).

Similar equations (though in the absence of the necrotic core) have been considered in~\cite{MR2176317}, where existence, uniqueness, and regularity results have been presented (see also~\cite{MR3950722} for the case of Robin boundary conditions). Also motivated by these results, we will always implicitly assume here that solutions of~\eqref{mainfinal1} are classical (not only distributional).

{F}rom now on, we will focus on the specific case in which
\begin{equation}\label{mainfinal1:DUE}
\chi_0:=2.
\end{equation}
This choice is due to computational convenience, since it allows one to construct some explicit solutions which
can be used as barriers for comparison arguments (see the forthcoming Proposition~\ref{propsoluzionestazionaria}).
We think that the case of more general values of~$\chi_0$ is interesting and deserves further investigation. 

It is also
worth noting that the problem is quite different from the ones already present in the literature. Indeed, 
outside the necrotic core, the standard diffusion is combined with a drift term, differently from the pure diffusion process described in~\cite{MR1684873}.

\subsection{The growth rate of the tumor}

In this subsection, following the approach in~\cite{greenspan}, we employ the principle of conservation of mass to derive the formulation for the growth rate of the tumor. To do this, we require some assumptions on the tumor. To start with, we suppose that every cell inside the tumor region is identical
in volume and density, as well as incompressible. We also assume that cell division by mitosis is instantaneous relatively to the growth time of the tumor, and every newborn cell has the same volume of any other living cell. Moreover, we suppose that the cell proliferation rate within the tumor is described by a function of the cell concentration, denoted by $S(u)$.
In this spirit, we assume that the number of cells produced at time~$t$ in the tumor region equals
\begin{equation}\label{NEWBORN}
{\mathcal{P}}(t):=\int_{B_{s(t)}\setminus B_\varepsilon} S(u(x,t))\,dx=
4\pi\int^{\max\{s(t),\varepsilon\}}_\varepsilon S(u(r,t))r^2\,dr.
\end{equation}

We also disregard the death and disintegration of cells into simpler chemical compounds:
in practice,
the production of
cellular debris in different stages of disintegration
and the degeneration of dead cells inside the necrotic core may also cause the tumor to lose volume,
but we suppose that the time scale at which a significant amount of
cell degeneration takes place
is longer than the one considered here in relation to tumor growth.
All these assumptions have to be regarded as convenient simplifications:
more complex models can be taken into account, but we present here
the conceptually simplest scenario in order to extract general quantitative information.

Under the above assumptions, the total number~${\mathcal{N}}(t)$ of cells
(of any kind) in the tumor region~$B_{s(t)}\setminus B_\varepsilon$
is proportional to the volume of this region, which is~$\frac{4\pi}3\big(s^3(t)-\varepsilon^3\big)$. The proportionality constant
is given by the cell density, which we suppose here to be normalized to~$1$. Then, we write
\begin{equation}\label{NUMEROC}
{\mathcal{N}}(t)=\frac{4\pi}3\big(s^3(t)-\varepsilon^3\big).\end{equation}

The constitutive assumption that the tissue maintains a constant cell density is common in many mathematical models for tumors, see e.g.~\cite{WARD}. In these cases, the models assume that the sum of the densities of the tumor and healthy cells is equal to a constant, which is the maximum capacity tissue density. In these contexts, the 
nutrient enables cell proliferation, which does not change the local density but generates pressure: since the cells cannot increase the density, they move outward, and this collective movement pushes the tumor boundary outward.

In these models, a tumor mass is treated as a multiphase medium composed of tumor cells, which may also become necrotic, and healthy cells (with the presence of an interstitial fluid ensuring nutrients diffusion), see e.g.~\cite{SCIU}.

Since we suppose that the density of cells in the tumor region remains constant, the newborn cells produced according to~\eqref{NEWBORN} will cause the tumor region itself to grow.
Namely, by the incompressibility assumption, the total number of cells~${\mathcal{N}}(t+\tau)$
in the tumor region~$B_{s(t+\tau)}\setminus B_\varepsilon$ equals the 
number of cells~${\mathcal{N}}(t)$
in the tumor region~$B_{s(t)}\setminus B_\varepsilon$, plus the cells produced
in the interval of time~$[t,t+\tau]$. This gives that
$$ {\mathcal{N}}(t+\tau)={\mathcal{N}}(t)+\int_t^{t+\tau}{\mathcal{P}}(\theta)\,d\theta$$
and therefore
$$ \dot{\mathcal{N}}(t)={\mathcal{P}}(t).$$
{F}rom this, recalling~\eqref{NEWBORN} and~\eqref{NUMEROC}, we conclude that 
\begin{equation}\label{spuntofinale}
s^2(t)\dot s(t)=\int_\varepsilon^{\max\{s(t),\varepsilon\}}S(u(r,t))r^2\,dr.
\end{equation}
We point out that formula~\eqref{spuntofinale} acts as a free boundary condition for problem~\eqref{mainfinal1}.
It will also play a crucial role in our analysis when establishing the appropriate upper and lower bounds for $s(t)$ and its derivative (see the forthcoming Lemma~\ref{lemmastimespunto}).

In our model we assume a linear dependence of the proliferation rate~$S(u)$ to the cell concentration~$u$, which corresponds to the inhibitor free model in~\cite{altrisigma}.
In particular, we assume that
\begin{equation}\label{PROPORZIONALE}
S(u):=\mu \big(u(r,t)-\widetilde u\big),
\end{equation}
where $\mu$ and $\widetilde u$ are positive constants.
Here, $\widetilde u>0$ represents a critical value for the cell density under which there is no tumor proliferation. Moreover, $\mu u(r, t)$ is the birth rate and $\mu \widetilde u$ denotes the death rate due to apoptosis (that is the process for which, when cells divide and grow, apoptosis functions as a mechanism to regulate this growth by removing some of the cells that may be malfunctioning, damaged, or excessive). With this choice, \eqref{spuntofinale} becomes\footnote{As we will show below \label{ATTEFOO223er} in~\eqref{stimespunto2}, the setting in~\eqref{freeboundary} simplifies, since we will prove that $${\mbox{$s(t) > \varepsilon$ for all~$t \in [0,+\infty)$.}}$$ Consequently, \eqref{freeboundary} reduces to
$$s^2(t)\dot s(t)=\mu\int_\varepsilon^{s(t)}\big(u(r,t)-\widetilde u\big)r^2\,dr. $$
This point is nevertheless delicate. On the one hand, a tumor model such as the one considered in this paper
is biologically meaningful only when $s(t) > \varepsilon$, because otherwise the region 
containing the proliferating tumor cells would vanish, rendering the tumor evolution trivial. 
On the other hand, we \emph{do not to assume} that~$s(t) > \varepsilon$ for all~$t \in [0,+\infty)$, as this would introduce an unnecessary structural hypothesis into 
the model: in fact, a skeptical reader might even argue that, under additional assumptions of this kind, 
the solution of the full system could fail to exist!

Instead, we view it as a significant byproduct of our approach that the model does not rely 
on any artificial assumptions. Rather, all constraints required for biological 
consistency, such as $s(t) > \varepsilon$ for all~$t \in [0,+\infty)$, are derived as 
consequences of the fundamental structural properties of the model and rigorous 
mathematical analysis.

In this sense, our model is self-consistent and robust: the biological admissibility of 
the solution is not imposed externally, but follows intrinsically from the structure of the system, and biological admissibility is not enforced by assumption, but emerges naturally from the internal dynamics of the model.}
\begin{equation}\label{freeboundary}
s^2(t)\dot s(t)=\mu\int_\varepsilon^{\max\{s(t),\varepsilon\}}\big(u(r,t)-\widetilde u\big)r^2\,dr.
\end{equation}
We stress that, since the necrotic core is made of dead cells, it gives no contribution on proliferation. Thus, the growth rate of the tumor is only influenced by the region $\Omega(t)\setminus B_\varepsilon$, as one can see from formula~\eqref{freeboundary}.

Throughout the paper we will employ the particular version of formula~\eqref{spuntofinale} stated in~\eqref{freeboundary}.
Moreover, we assume that the critical
value for cell proliferation
lies below the cell concentration on the boundary of the tumor region (this models a case in which
proliferation provides a significant contribution
to tumor growth, as posited in the angiogenesis framework). Namely, we suppose that\footnote{In particular, condition~\eqref{relationutilde} models the case in which the necrotic core has density that is
smaller than that the boundary of the active tumor region.
This is coherent with the biological scenario, since,
in most tumors, the necrotic core, where cells disintegrate and tissue structure breaks down, exhibits a lower cell density than the outer proliferative region, which benefits from better oxygenation and nutrient supply, see e.g. the section on ``MCTs growing and structure'' in~\cite{CHALLE}.}
\begin{equation}\label{relationutilde}
0<\widetilde u < \underline{u} < \overline{u}.
\end{equation}

As shown in the forthcoming Proposition \ref{propabcsPI}, the (unique and radial) stationary solution of problem~\eqref{mainfinal1} satisfies \eqref{freeboundary} if and only if $\widetilde u>\dfrac{\Gamma u_B}{\lambda}$.
Accordingly, our analysis is carried out under the assumption
\begin{equation}\label{relationutilde2}
0<\frac{\Gamma u_B}{\lambda}<\widetilde u < \underline{u} < \overline{u}.
\end{equation}

In particular, we assume that the cell concentration is higher at the outer boundary of the tumor than in the necrotic core. This assumption is physiologically justified by the fact that nutrients diffuse from the surrounding environment toward the interior,
favoring tumor growth at the interface. Thus, cells located in the outer proliferative rim have direct access to higher nutrient levels, while moving inward these nutrients are progressively consumed, eventually dropping below the threshold required for cell survival and leading to necrosis. This behavior is well documented in the literature (see e.g. \cite{XU2016}), and strong supports our assumption \eqref{relationutilde2}.

To ease the comparison between the mathematical formalism and the medical application,
we provide in here below a table of Symbol Description.
\smallskip

\begin{center}
\begin{tabularx}{\textwidth}{ >{\centering\arraybackslash}m{1.5cm} X }
\toprule
\textbf{Symbol} & \textbf{Description} \\
\midrule
$u$ & The concentration of the tumor cells. \\
\midrule
$\Omega(t)$ & 
The region occupied by the tumor at time~$t$
(corresponding to the region in which the cell density is above the threshold~$\widetilde{u}$).\\
\midrule
$B_\varepsilon$ & The necrotic core (the boundary of which is supposed
to have a constant cell concentration~$\overline{u}>\underline{u}$). \\
\midrule
$c$ & The concentration of the chemoattractant. \\
\midrule
$\chi(c)$ & The chemotactic sensitivity
(taken here to be of the form~$\frac2{c}$). \\
\midrule
$\kappa$ & The reciprocal diffusion coefficient. \\
\midrule
$\lambda $  & The effective death rate of tumor cells. \\
\bottomrule
\end{tabularx}
\end{center}
\medskip

\subsection{Main results}
Our first result ensures that, in the model under consideration, the tumor, once formed, survives through time. This is formalized in the following result:

\begin{theorem}\label{th:main1}
Assume~\eqref{u_0positiva},
\eqref{mainfinal1:DUE}, and~\eqref{relationutilde2}.
Let~$(u(r,t), s(t))$ be a solution of~\eqref{mainfinal1} and~\eqref{freeboundary}.

Then, there exists $\delta_*>0$ depending only on\footnote{An explicit formulation for $\delta_*$ is given by \eqref{delta0esplicita}, \eqref{delta1esplicita} and \eqref{deltastar}.}
 $\kappa$, $\lambda$, $\widetilde u$, $\underline u$, $\overline u$, $\mu$, $\varepsilon$, $\Gamma$, $u_B$, and~$s(0)$ such that
\begin{equation}\label{deltastarclaim}
s(t)- \varepsilon\ge\delta_* \quad \mbox{for any }t\geq 0.
\end{equation}
\end{theorem}

{ For the analysis that follows, we need to study the dependence of $\delta_*$ in Theorem~\ref{th:main1}
on the coefficient $\kappa$. To this end, we will treat $\delta_*$ as a function of $\kappa$ and we will use the notation $\delta_*(\kappa)$.


However, for our analysis it becomes important to retain some quantitative control on $\delta_*$ in dependence of
the reciprocal diffusion coefficient~$\kappa$. More precisely, we need to establish a positive lower bound, at least when $\kappa$ is restricted to compact subsets. The next result provides exactly this estimate:
\begin{proposition}\label{propCK>0}
Let the assumptions of Theorem \ref{th:main1} hold true
and let $\delta_*=\delta_*(\kappa)$ be as in Theorem \ref{th:main1}. 

Then, for any compact set $K\subset (0, +\infty)$, there exists a constant $C_K>0$ depending only on $K$, $\lambda$, $\widetilde u$, $\underline u$, $\overline u$, $\mu$, $\varepsilon$, $\Gamma$, $u_B$, and~$ s(0)$, such that
\[
s(t)-\varepsilon\ge\delta_*(\kappa)\ge C_K>0\quad\mbox{ for any } t\ge 0,
\]as long as~$\kappa\in K$.
\end{proposition}

We point out that Proposition \ref{propCK>0} will play a crucial role in several intermediate results (see e.g. Lemma \ref{lpèkoijhugyvftcdrxsedctfvygbhjnkm}) leading to the proof of our main results.  This constitutes a significant difference with respect to Friedman’s works (see e.g. \cite{MR1684873}), and it is due to the fact that our model is more complex and also features the presence of a necrotic core. 

Furthermore, we point out that the scaling law of $\delta_*(\kappa)$, as derived from the forthcoming formulas \eqref{delta0esplicita},  \eqref{delta1esplicita} and \eqref{deltastar}, suggests that the viable rim thickness $\delta_*(\kappa)$ can be interpreted as a diffusion-driven length scale, providing a natural framework for comparison with the experimental and numerical evidence reported below.

The three-layer structure introduced by \cite{greenspan} and already presented in Section \ref{INT1.1} arises from the fact that nutrient supply is diffusion-limited. In this context, one often focuses in particular on oxygen diffusion, since oxygen is typically the limiting nutrient for cell survival and proliferation. Accordingly, the cell diffusion coefficient in our model can be considered as possibly related to the effective diffusion coefficient of oxygen. 

Oxygen and other essential nutrients can reach only up to a finite distance from the outer boundary before they become depleted by cellular consumption. As it is known, oxygen diffusion typically occurs over distances of about $100–200\mu m$ in biological tissues, which determines the characteristic thickness of the viable zone near the tumor boundary. Experimental measurements in solid tumors report oxygen diffusion distances in the range of $107-193\mu m$ in mouse tumor ``cubes" prepared from solid tumors, depending on the tumor type, the oxygen consumption rate, and the external conditions, and corresponding to the thickness of the oxygenated rim surrounding hypoxic regions (see \cite{Peggyetal}). Consistently, studies on human breast cancer xenografts provide typical estimates within this range, showing that hypoxia and radioresistance arise when intercapillary distances exceed approximately $100-120\mu m$, and are further aggravated by factors such as reduced microvascular blood flow, decreased red blood cell flux, or anemia (see \cite{GroebeVaupel}).

As a consequence, the distance between the tumor boundary and the onset of necrosis is expected to be of this order, suggesting that\footnote{We recall that, for the sake of simplicity, here the radius of the necrotic core is assumed to be constant in time, so that $\varepsilon(t)=\varepsilon$.}
\[
s(t)-\varepsilon(t)\sim 100/200 \mu m.
\]
This conclusion is further supported by numerical computations. For instance, in \cite{SCIU}, it is observed that once a necrotic core is established, there is a shell of viable tumor cells surrounding the tumor, whose thickness remains approximately constant over time.  More precisely, in Section 5.1 there, a shell thickness of approximately $150 \mu m$ is reported based on numerical computations (see \cite[Figure 4(d)]{SCIU}, where the radii of the spheroid and of the necrotic core are displayed). It is also remarked there that this thickness depends on the cell line and on nutrient availability.}

We now provide our second main result. 

We prove that, under suitable assumptions on the cell distribution and the model parameters, the tumor radius is either uniformly bounded for all $t>0$ or may experience extremely rapid contractions. More precisely, we let $\kappa$ range over arbitrary compact subsets of $\left(0, \dfrac{\lambda}{\mu(\overline u-\widetilde u)}\right)$. Within this framework, we assume that, in a left neighborhood of $s(0)$,  the initial cell concentration profile $u_0(r)$ stays above the value attained at the boundary of the necrotic core, and that is of class $C^1 ([\varepsilon, s(0)])$, with $u'_0(r)$ bounded by a positive constant depending only on $\lambda, \overline u, \Gamma, u_B, C_K$. 
Under these conditions, we prove that either the speed of the tumor boundary can become arbitrarily negative, meaning that the boundary may experience episodes of fast contraction, or the tumor radius remains uniformly bounded for all $t>0$. From a biological perspective, the latter represents an expected outcome for a diffusion driven model, while the former describes an extreme regime in which boundary retreat occurs at unbounded rates.

The result that we obtain reads as follows:
\begin{theorem}\label{main3}
 Assume that~\eqref{mainfinal1:DUE} and \eqref{relationutilde2} hold true.

Let $K\subset \left(0, \dfrac{\lambda}{\mu(\overline u-\widetilde u)}\right)$ be compact and let $C_K>0$ be as in Proposition \ref{propCK>0}. Assume that~$\kappa\in K$, and that $u_0\in C^1([\varepsilon, s(0)])$ is such that
\begin{equation}\label{kolmjopjnhbgyugfcdrtesxzwea}
u_0(r)\ge\underline u\quad\mbox{ for any } r\in [\varepsilon, s(0)-C_K],
\end{equation}
and satisfies
\begin{equation}\label{okpijhugyftdruyjmlò,.àxqBHJKNGVCUhbjnaekfml}
\|u'_0\|_{L^\infty([s(0)-C_K, s(0)])}<\dfrac{\lambda\overline u-\Gamma u_B}{\lambda C_K}.
\end{equation}

Let $(u(r,t), s(t))$ be a solution of~\eqref{mainfinal1} and \eqref{freeboundary}.

Then,  either
\begin{equation}\label{either1}
\inf_{t>0} \dot{s}(t)=-\infty
\end{equation}
or
\begin{equation}\label{s(t)bounded}
\mbox{there exists a constant $C>0$ such that } \sup_{t>0}s(t)\le C.
\end{equation}
\end{theorem}

It would be interesting to investigate whether both of the scenarios in~\eqref{either1} and~\eqref{s(t)bounded} can actually occur. In any case, the alternative described in~\eqref{either1} and \eqref{s(t)bounded} is biologically interesting, because it rules out the possibility that the tumor mass grows monotonically and indefinitely.

We now turn to the next result, which describes the behavior of the tumor in the regime of large $\kappa$. We emphasize that, in order to establish this result, we cannot rely on the lower bound provided by Theorem \ref{th:main1}, since its dependence on $\kappa$ prevents a meaningful formulation in this setting.

However, by using a barrier argument, we are able to derive an alternative lower bound for the quantity $s(t)-\varepsilon$. 
More precisely, we prove that there exists a positive constant $\widetilde \eta$ such that, if $\eta<\min\{\widetilde \eta,s(0)-\varepsilon\}$, then $s(t)-\varepsilon\ge \eta$ holds true for any $t\ge 0$.
Therefore, roughly
speaking, while some arbitrariness is allowed in the choice of $\eta$ here above, the results obtained
will be stronger if one picks $\eta$ “as large as possible” but still strictly smaller than $\min\{\widetilde \eta,s(0)-\varepsilon\}$.

Our result reads as follows:
\begin{proposition}\label{propeta}
Assume that 
\begin{equation}\label{oiuygtfdoiuy2}
u_0(r)>\frac{\Gamma u_B}{\lambda} \quad \mbox{for any } r\in [\varepsilon, s(0)].
\end{equation}
Let $(u(r,t), s(t))$ be a solution of~\eqref{mainfinal1} and~\eqref{freeboundary}.

Then, there exists~$\eta\in(0,\,s(0)-\varepsilon]$ depending on~$\varepsilon$, $s(0)$, $u_0$, $\lambda$, $\underline u$, $\widetilde u$, $\Gamma$ and $u_B$ such that, for any $t\ge0$, 
\[
s(t)-\varepsilon\ge \eta.
\]
\end{proposition}

We assume that the initial concentration of cells has a positive weighted integral (as quantified in \eqref{ctfyivgoubhjklmò,2}) and remains below a suitable map, which provides a monotone interpolation between the inner and outer boundary concentrations, as required in \eqref{lpko,mijnhugyftcdrxctfvl,ò}.
Under these assumptions, we prove that there exists a threshold $\kappa_\star$ (given explicitly in~\eqref{kopfweoiojiwavnjbhin}) such that, whenever $\kappa$ exceeds this threshold, the tumor radius grows exponentially fast in time. 

Moreover, for the reader’s convenience,  we provide in Appendix \ref{apprdfctbhjkl–òlkjhbgvfdrdxcfvgbhnjmk,} some examples of functions $u_0(r)$ satisfying both assumption \eqref{ctfyivgoubhjklmò,2} and \eqref{lpko,mijnhugyftcdrxctfvl,ò}.

\begin{figure}[h]
\begin{center}
\includegraphics[scale=.48]{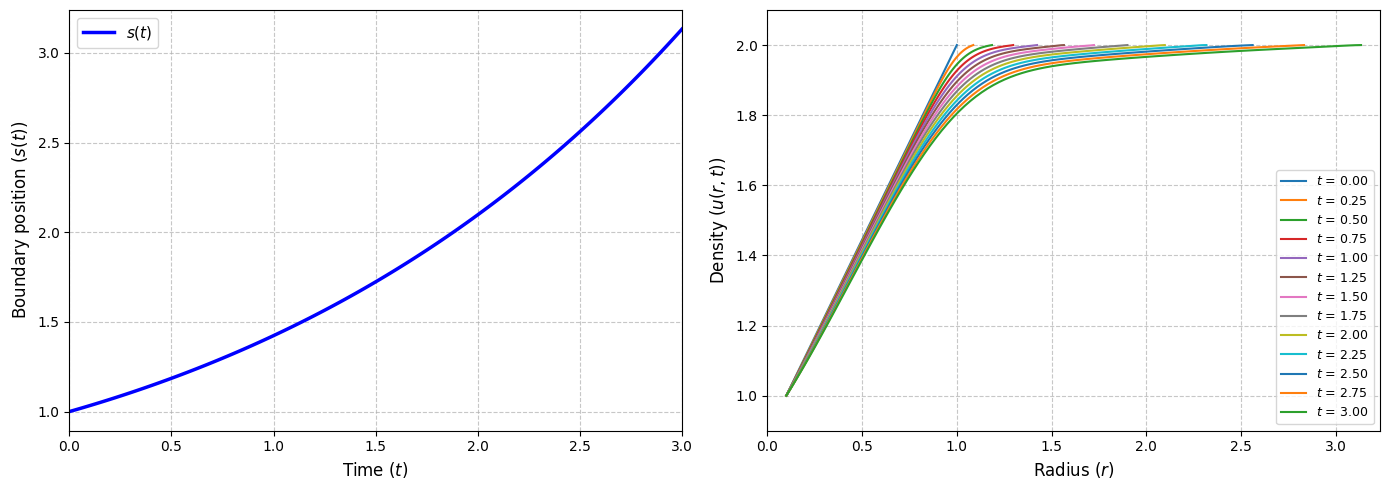}
\end{center}
\caption{Numerical plots of~$s(t)$ (left) and~$u(r,t)$ (right),
with~$u_0(r):= \underline u +\dfrac{\overline u -\underline u}{s(0)-\varepsilon}(r-\varepsilon)$ and the 
illustrative parameters~$
\Gamma = 1$, $u_B = 1$, $
\lambda = 2$, $\varepsilon = 0.1$,
$s(0) = 1$,
$\tilde u = 0.75$,
$\underline u = 1$, $\overline u = 2$,
$\mu = 1$, $\kappa = 100$.}
\label{PLOTSANDU}
\end{figure}

In particular, the following holds:
\begin{theorem}\label{main2} Assume~\eqref{mainfinal1:DUE},
\eqref{relationutilde2}, and~\eqref{oiuygtfdoiuy2}.

Suppose that $u_0\in C^1([\varepsilon, s(0)])$ is such that $u_0(r)>\frac{\Gamma u_B}{\lambda}$ for any $r\in(\varepsilon, s(0))$ and that it satisfies
\begin{equation}\label{ctfyivgoubhjklmò,2}
\int_\varepsilon^{s(0)} (2 u_0(r)-\overline{u} -\widetilde u) r^2\, dr>0
\end{equation}
and
\begin{equation}\label{lpko,mijnhugyftcdrxctfvl,ò}
 u_0(r) \le\underline u +\dfrac{\overline u-\underline u}{1-e^{-\sqrt\lambda (s(0)-\varepsilon)}}\left(1-e^{-\sqrt\lambda(r-\varepsilon)}\right) \quad\mbox{ for any } r\in [\varepsilon, s(0)].
\end{equation}

Let $(u(r,t), s(t))$ be a solution of~\eqref{mainfinal1} and~\eqref{freeboundary}. 

Then, there exists~$\kappa_\star>0$, depending on $\varepsilon$, $s(0)$,
$u_0$, $\lambda$, $\underline u$, $\overline u$,
$\widetilde u$, $\Gamma$, $\mu$ and $u_B$,
such that
\[
\limsup_{t\to +\infty} s(t) =+\infty,
\]as long as~$\kappa\ge\kappa_\star$.
\end{theorem}

\begin{remark} An explicit value for~$\kappa_\star$
in Theorem~\ref{main2} is
\begin{equation}\label{kopfweoiojiwavnjbhin}
\kappa\ge\kappa_\star:=\max\left\{\dfrac{2\lambda}{\mu(\overline u-\widetilde u)}, \frac{16}{\mu(\overline u-\widetilde u)^2}\left(\lambda\widetilde u +\dfrac{\sqrt\lambda(\overline u-\underline u)}{\eta(1-e^{-\sqrt\lambda \eta})}+\frac{2\overline u}{\varepsilon^2}+\dfrac{2(\overline u -\underline u)}{\varepsilon\eta}\right)\right\},
\end{equation}
where $\eta$ is as in Proposition \ref{propeta}.
\end{remark}

We refer the reader to Figure~\ref{PLOTSANDU} for a numerical sketch of a tumor invasion of the type described by Theorem~\ref{main2}.  The parameter values used 
for these graphs are merely illustrative,
but the plot in Figure~\ref{PLOTSANDU} is qualitative in good agreement with medical data available in the literature. Indeed, one can compare, for instance, the left panel in Figure~\ref{PLOTSANDU} with the experimental data plotted in Figure~1(A) in~\cite{SILVO},  and
in Figures~2 and~3 of~\cite{SCHIK}.

See also Appendix \ref{apprdfctbhjkl–òlkjhbgvfdrdxcfvgbhnjmk,} for examples of initial cell densities
satisfying both~\eqref{ctfyivgoubhjklmò,2} and~\eqref{lpko,mijnhugyftcdrxctfvl,ò}.

We stress that condition~\eqref{kopfweoiojiwavnjbhin}, though perhaps not completely transparent at a first sight, aligns with our physical understanding of tumor growth and diffusion.
Indeed,
if the growth of the tumor domain ${\Omega(t)}$ is sufficiently rapid (as formalized by~\eqref{kopfweoiojiwavnjbhin}), the pressure generated by cell proliferation must necessarily be large
(because this pressure is the driving force behind the expansion of the high-density region, which defines the tumor domain, recall~\eqref{spuntofinale}).

Given the monotonic relationship between the cell proliferation rate $S(u)$ and the cell density~$u$ (as quantified by~\eqref{PROPORZIONALE}), a large proliferation-induced pressure within the tumor domain~$\Omega(t)$ requires the local cell density~$u$ to be large.

Accordingly,
since diffusion acts to spread and thus decreases the concentration of tumor cells inside a given domain, this rapid growth of the tumor domain implies that the parabolic equation modeling the cell transport must have a small diffusion coefficient
(at least, sufficiently small to ensure that the cells do not tend to escape the region~$\Omega(t)$ ``too fast'').

As a consequence, because $\kappa$ is defined as the
reciprocal of the diffusion coefficient, fast tumor growth should correspond to sufficiently large values of $\kappa$ (coherently with~\eqref{kopfweoiojiwavnjbhin}).
As a counterpart of this argument, we see that diffusion limits tumor spheroid
growth and determines the maximum size and growth rate: this fact is confirmed by in vitro
experiments (see [\cite{MR2327053}, page 181] and the references therein).

A further remark concerns our choice to distinguish between the regimes of large and small $\kappa$.
We point out that, although the effective diffusion coefficient is relatively similar across human tissues, small variations may arise due to microstructural properties such as porosity, tortuosity, and extracellular matrix density. However, these differences are moderate, and oxygen availability is primarily governed by vascular architecture and cellular consumption rather than by substantial changes in diffusivity (see \cite{sykova08, vaupel07}).

In our model, we nevertheless distinguish between the regimes $\kappa$ ``large" and $\kappa$ ``small", in order to investigate the sensitivity of the viable rim thickness to diffusion efficiency. This should not be interpreted as assuming physiologically large variations in the true diffusion coefficient across tissues. Rather,  here $\kappa$ is used as an effective macroscopic parameter that incorporates modest changes induced by tissue microstructure (such as porosity and tortuosity) and other unresolved heterogeneities affecting transport. In this sense, the two regimes represent limiting cases of slightly more or less permissive diffusion environments, rather than distinct biological classes of tissues.

\subsection{Comparison with existing models}
We point out that the presence of both the drift term and the necrotic core marks a clear distinction from our model and the others present in the literature. 
However, we compare our results with those stated in the pioneering work~\cite{MR1684873}, even though it considers neither the presence of a drift term nor the presence of a necrotic core.
By comparing Theorem~\ref{th:main1} with~\cite[Theorem~4.1]{MR1684873}, one can notice that the presence of the drift term still ensures a lower bound for the tumor radius. This guarantees that the tumor persists once it has been formed. To the best of our knowledge, this is the first time that such a lower bound is given in a fully explicit form, expressed in terms of the initial data, including the radius of the necrotic core (see formulas \eqref{delta0esplicita}, \eqref{delta1esplicita}, \eqref{deltastar}). We believe this feature to be particularly relevant, as it makes possible, in principle, to directly compare the value predicted by the theoretical analysis with those obtained from experimental measurements.

In addition, by means of a barrier argument, we obtain an additional lower bound that is independent of $\kappa$ (see Proposition \ref{propeta}). To the best of our knowledge, neither this kind of estimate nor this type of argument has been employed in earlier works.

In Theorem~\ref{main3}, we show that either the radius $s(t)$ of the tumor remains uniformly bounded over time, or it may experience a fast contraction. This latter possibility is entirely new and does not arise in any of the previously considered models.
The results in Theorem~\ref{main3} are guaranteed by imposing a smallness condition on the parameter $\kappa$ and mild regularity assumptions on the initial datum~$u_0(r)$. However,  the proof of the result relies on several technical arguments related to the construction of appropriate barriers for the control of the radial derivative of~$u(r,t)$ at both the boundary of the necrotic core and the moving free boundary.
Since these constructions are related to the particular form of our model equation, which takes into account the presence of a necrotic core and a drift term, they do not arise in~\cite[Theorem 5.1]{MR1684873}.


\subsection{Organization of the paper}
The paper is organized as follows.
In Section~\ref{dertfyuiokpl} we compute the stationary solution to the problem given by~\eqref{mainfinal1} and~\eqref{freeboundary}. In Section~\ref{dtfygubhijnkml,ò.-aklmSJNBH} we provide some preliminary results needed to prove Theorem~\ref{th:main1}. In Section~\ref{mkljnbgvyftcdrxsedctfvygbhunijm}, we present the proofs of Theorems~\ref{th:main1} and Proposition \ref{propCK>0}.

In Section \ref{èlpkmojbhugvyufctydrstetdryftgyhu} we present the proof of Theorem \ref{main3}. In Section \ref{kpoijhuygtfrdetfyguhijk} we prove both Proposition \ref{propeta} and Theorem \ref{main2}.
  
Finally, in Appendix \ref{apprdfctbhjkl–òlkjhbgvfdrdxcfvgbhnjmk,}, we illustrate some examples of functions satisfying both assumption \eqref{ctfyivgoubhjklmò,2} and \eqref{lpko,mijnhugyftcdrxctfvl,ò}.

\section{Stationary solutions}\label{dertfyuiokpl}
As a first step towards the proof of Theorem~\ref{th:main1}, we need to write the (unique and radial) stationary solution of problem~\eqref{mainfinal1} and~\eqref{freeboundary} explicitly. To do this, we let~$R_0>\varepsilon$ and look for a solution of~\eqref{mainfinal1} not depending on time and with~$s(t)=R_0$.

In light of the radial symmetry of the solution, we can write
\begin{equation}\label{laplradial}
\Delta u(x) = u''(r)+\frac{2}{r}u'(r)\qquad {\mbox{and}}\qquad
x\cdot\nabla u(x) = r\, u'(r).
\end{equation}
Hence, a stationary solution of~\eqref{mainfinal1} is determined by the system
\begin{equation}\label{stationarysystem}
\begin{cases}
u''(r)+\displaystyle\frac{2-\chi_0}{r}u'(r) -\lambda u(r)+\Gamma u_B=0 &\quad\mbox{ if } \varepsilon<r<R_0,\\
u(r) = \overline{u} &\quad\mbox{ if } r=R_0, \\
u(r) = \underline{u}, &\quad\mbox{ if } r=\varepsilon.
\end{cases}
\end{equation}

We point out that
finding an exact solution in terms of elementary functions of this problem
is not always feasible and it depends on the specific value of the constant~$\chi_0>0$.

In the next result we prove that, by assuming~\eqref{mainfinal1:DUE}, it is possible to determine an exact solution in terms of elementary functions of~\eqref{stationarysystem}. For this, from now on we set~$\chi_0:=2$.

\begin{proposition}\label{propsoluzionestazionaria}
For any $R_0>\varepsilon$ the problem in~\eqref{stationarysystem} admits a solution given by
\begin{equation}\label{stationaryall}
u_\star(r):=
\frac{\left(\overline{u}-\frac{\Gamma u_B}{\lambda}\right)
\sinh\big(\sqrt{\lambda}(r-\varepsilon)\big) +\left(\underline{u}-\frac{\Gamma u_B}{\lambda}\right)
\sinh\big(\sqrt{\lambda}(R_0-r)\big)}{\sinh\big(\sqrt{\lambda}(R_0-\varepsilon)\big)}
+\frac{\Gamma u_B}{\lambda}.
\end{equation} 
\end{proposition}

\begin{proof}
Since $\chi_0=2$,
we have that~$u_\star$ satisfies the differential equation
\[
u''_\star(r)-\lambda u_\star(r)+\Gamma u_B=0 \quad\mbox{ in } \varepsilon<r<R_0,
\]
and therefore
\[
u_\star(r) = A e^{\sqrt{\lambda}r} + B e^{-\sqrt{\lambda}r} +\frac{\Gamma u_B}{\lambda},
\]
for some~$A$, $B\in\R$ to be determined. 

Since $u_\star(R_0)=\overline{u}$ and $u_\star(\varepsilon)=\underline{u}$, one obtains that
\begin{eqnarray*}
A&=&e^{-\sqrt{\lambda}\varepsilon}
\left(\frac{\overline{u} -\frac{\Gamma u_B}{\lambda}
-\left(\underline{u}-\frac{\Gamma u_B}{\lambda}\right) e^{-\sqrt{\lambda}(R_0-\varepsilon)}}{2\sinh\big(\sqrt{\lambda}(R_0-\varepsilon)\big)}\right)
\\{\mbox{and}} \qquad B&=& -e^{\sqrt{\lambda}\varepsilon}\left(
\frac{\overline{u} -\frac{\Gamma u_B}{\lambda}
-\left(\underline{u}-\frac{\Gamma u_B}{\lambda}\right)e^{\sqrt{\lambda}(R_0-\varepsilon)}}{2\sinh\big(\sqrt{\lambda}(R_0-\varepsilon)\big)}\right).
\end{eqnarray*}
Accordingly,  
\begin{eqnarray*}
u_\star(r)& =&  e^{\sqrt{\lambda}(r-\varepsilon)}
\left(\frac{\overline{u} -\frac{\Gamma u_B}{\lambda}
-\left(\underline{u}-\frac{\Gamma u_B}{\lambda}\right) e^{-\sqrt{\lambda}(R_0-\varepsilon)}}{2\sinh\big(\sqrt{\lambda}(R_0-\varepsilon)\big)}\right)
\\&&\quad-e^{-\sqrt{\lambda}(r-\varepsilon)}\left(
\frac{\overline{u} -\frac{\Gamma u_B}{\lambda}
-\left(\underline{u}-\frac{\Gamma u_B}{\lambda}\right)e^{\sqrt{\lambda}(R_0-\varepsilon)}}{2\sinh\big(\sqrt{\lambda}(R_0-\varepsilon)\big)}\right)
+\frac{\Gamma u_B}{\lambda}
\\&=&\frac{\left(\overline{u}-\frac{\Gamma u_B}{\lambda}\right)
\sinh\big(\sqrt{\lambda}(r-\varepsilon)\big)
+\left(\underline{u}-\frac{\Gamma u_B}{\lambda}\right)
\sinh\big(\sqrt{\lambda}(R_0-r)\big)}{\sinh\big(\sqrt{\lambda}(R_0-\varepsilon)\big)}
+\frac{\Gamma u_B}{\lambda},
\end{eqnarray*}
as desired.
\end{proof}

In the remaining part of this section, we prove that there exists $R_0>\varepsilon$ such that the solution found in Proposition~\ref{propsoluzionestazionaria} satisfies the free boundary condition stated in~\eqref{freeboundary} with $s(t)=R_0$ (and thus $\dot s(t)=0$), if and only if $\widetilde u$ is sufficiently large, with the precise lower bound specified in the condition~\eqref{kjinkmjnbhkmjbhkm,mjhukol8hbuijnklmò9} below. 

To this aim, we recall the following straightforward observation.

\begin{proposition}\label{Proplèpkmoò.,lmkbhul,ò.à-}
Let $a, b\in\R.$ Then,
\[
\sinh(a) +\sinh(b) = 2\sinh\left(\frac{a+b}{2}\right)\cosh\left(\frac{a-b}{2}\right).
\]
\end{proposition}

\begin{proof}
By recalling the definitions of $\sinh$ and $\cosh$, we infer that
\[
\begin{split}&
2\sinh\left(\frac{a+b}{2}\right)\cosh\left(\frac{a-b}{2}\right) =\frac12 (e^{\frac{a+b}{2}}-e^{-\frac{a+b}{2}})(e^{\frac{a-b}{2}}+e^{-\frac{a-b}{2}})=\frac12(e^a +e^b-e^{-b}-e^{-a})\\
&\qquad\qquad= \frac{e^a -e^{-a}}{2} +\frac{e^b -e^{-b}}{2} =\sinh(a) +\sinh(b),
\end{split}
\]
as desired.
\end{proof}

\begin{proposition}\label{propabcsPI}
Let $u$ be as in~\eqref{stationaryall} and assume that~\eqref{relationutilde} holds true. 

Then, there exists at least one $R_0>\varepsilon$ such that
\begin{equation}\label{<sercvgunjkop}
\int_\varepsilon^{R_0}\big(u(r)-\widetilde u\big)r^2\,dr=0
\end{equation}
if and only if 
\begin{equation}\label{kjinkmjnbhkmjbhkm,mjhukol8hbuijnklmò9}
\widetilde u >\frac{\Gamma u_B}{\lambda}.
\end{equation}

Moreover,
\begin{equation}\label{drftguhikl}
\inf\big\{R_0>\varepsilon \mbox{ such that $\eqref{<sercvgunjkop}$ holds}\big\}>\varepsilon.
\end{equation}
\end{proposition}

We point out that the statement in~\eqref{<sercvgunjkop} guarantees the
free boundary condition stated in~\eqref{freeboundary}.

\begin{proof}[Proof of Proposition~\ref{propabcsPI}]
Assume that \eqref{kjinkmjnbhkmjbhkm,mjhukol8hbuijnklmò9} holds true.
In order to prove~\eqref{<sercvgunjkop}, we treat $R_0$ as a variable and we define, for all~$R_0\ge\varepsilon$,
\[
m(R_0):=-\int_\varepsilon^{R_0}(u(r)-\widetilde u)r^2\,dr.
\]
We notice that, from \eqref{stationaryall}, the function $u$ also depends on $R_0$.

Now, we want to check that $m(R_0)=0$ for some $R_0>\varepsilon$. To start with, we observe that~$m(\varepsilon)=0$.

Moreover, recalling~\eqref{stationaryall}, we compute
\begin{eqnarray*}
&&m(R_0)\\&=&\left(\widetilde u-\frac{\Gamma u_B}\lambda\right)
\frac{{R_0}^3-\varepsilon^3}{3}
-\frac{\overline{u}-\frac{\Gamma u_B}{\lambda}}{\sinh\big(\sqrt{\lambda}(R_0-\varepsilon)\big)}\int_\varepsilon^{R_0}
\sinh\big(\sqrt{\lambda}(r-\varepsilon)\big) \,r^2\,dr\\&&
-
\frac{\underline{u}-\frac{\Gamma u_B}{\lambda}}{\sinh\big(\sqrt{\lambda}(R_0-\varepsilon)\big)}\int_\varepsilon^{R_0}
\sinh\big(\sqrt{\lambda}(R_0-r)\big)\,r^2\,dr\\
&=&
\left(\widetilde u-\frac{\Gamma u_B}\lambda\right)
\frac{{R_0}^3-\varepsilon^3}{3}
\\&&-\frac{\overline{u}-\frac{\Gamma u_B}{\lambda}}{\sinh\big(\sqrt{\lambda}(R_0-\varepsilon)\big)}
\left(
\frac{(2 + \lambda {R_0}^2) \cosh\big(\sqrt\lambda (R_0-\varepsilon )\big)}{\lambda^{3/2}} - \frac{2 R_0\sinh\big(\sqrt\lambda (R_0-\varepsilon)\big)}{\lambda}
-
\frac{(2 + \lambda \varepsilon^2) }{\lambda^{3/2}}\right)
\\&&
-
\frac{\underline{u}-\frac{\Gamma u_B}{\lambda}}{\sinh\big(\sqrt{\lambda}(R_0-\varepsilon)\big)}
\left( -\frac{(2 + \lambda {R_0}^2) }{\lambda^{3/2}} 
+\frac{(2 + \lambda \varepsilon^2) \cosh\big(\sqrt\lambda ( R_0-\varepsilon)\big)}{\lambda^{3/2}} 
+ \frac{2 \varepsilon \sinh\big(\sqrt\lambda (R_0-\varepsilon)\big)}{\lambda}
\right).
\end{eqnarray*}
Therefore, factorizing common terms,
\begin{eqnarray*}
m(R_0)&=&
\left(\widetilde u-\frac{\Gamma u_B}\lambda\right)
\frac{{R_0}^3-\varepsilon^3}{3}+\frac2{\lambda}\left(
R_0\left(\overline{u}-\frac{\Gamma u_B}{\lambda}\right)
-\varepsilon\left(\underline{u}-\frac{\Gamma u_B}{\lambda}
\right)\right)
\\&&
-\left(
\left(\overline{u}-\frac{\Gamma u_B}{\lambda}\right)
(2 + \lambda {R_0}^2) 
+\left(\underline{u}-\frac{\Gamma u_B}{\lambda}
\right)(2 + \lambda \varepsilon^2) \right)
\frac{\coth\big(\sqrt\lambda (R_0-\varepsilon )\big)}{\lambda^{3/2}}
\\&&+
\left(
\left(\overline{u}-\frac{\Gamma u_B}{\lambda}\right)(2 + \lambda \varepsilon^2) 
+
\left(\underline{u}-\frac{\Gamma u_B}{\lambda}\right)
(2 + \lambda {R_0}^2) \right)
\frac{\text{csch}\big(\sqrt{\lambda}(R_0-\varepsilon)\big)}{\lambda^{3/2}}
.
\end{eqnarray*}

We now show that the function $m$ is negative in a right neighbourhood of~$\varepsilon$. To do this, we let~$\delta\in(0,1)$ and we compute
\begin{eqnarray*}&&
m(\varepsilon+\delta)\\&=&
\left(\widetilde u-\frac{\Gamma u_B}\lambda\right)
\frac{(\varepsilon+\delta)^3-\varepsilon^3}{3}
+\frac2{\lambda}\left(
(\varepsilon+\delta)\left(\overline{u}-\frac{\Gamma u_B}{\lambda}\right)
-\varepsilon\left(\underline{u}-\frac{\Gamma u_B}{\lambda}
\right)\right)
\\&&
-\left(
\left(\overline{u}-\frac{\Gamma u_B}{\lambda}\right)
(2 + \lambda (\varepsilon+\delta)^2) 
+\left(\underline{u}-\frac{\Gamma u_B}{\lambda}
\right)(2 + \lambda \varepsilon^2) \right)
\frac{\coth\big(\sqrt\lambda \delta\big)}{\lambda^{3/2}}
\\&&+
\left(
\left(\overline{u}-\frac{\Gamma u_B}{\lambda}\right)(2 + \lambda \varepsilon^2) 
+
\left(\underline{u}-\frac{\Gamma u_B}{\lambda}\right)
(2 + \lambda (\varepsilon+\delta)^2) \right)
\frac{\text{csch}\big(\sqrt{\lambda}\delta\big)}{\lambda^{3/2}}\\
&=&\left(\widetilde u-\frac{\Gamma u_B}\lambda\right)
\frac{3\varepsilon^2\delta+3\varepsilon\delta^2+\delta^3}{3}
+\frac2{\lambda}\left(
\delta\left(\overline{u}-\frac{\Gamma u_B}{\lambda}\right)
+\varepsilon(\overline{u}-\underline{u})\right)
\\&&
-\left(\overline{u}+\underline{u}-\frac{2\Gamma u_B}{\lambda}\right)
(2 + \lambda \varepsilon^2)
\frac{\coth\big(\sqrt\lambda \delta\big)}{\lambda^{3/2}}
-\left(\overline{u}-\frac{\Gamma u_B}\lambda\right)(2\varepsilon\delta+\delta^2)
\frac{\coth\big(\sqrt\lambda \delta\big)}{\sqrt\lambda}
\\&&+
\left(\overline{u}+\underline{u}-\frac{2\Gamma u_B}{\lambda}\right)
(2 + \lambda \varepsilon^2) 
\frac{\text{csch}\big(\sqrt{\lambda}\delta\big)}{\lambda^{3/2}}
+
\left(\underline{u}-\frac{\Gamma u_B}\lambda\right)(2\varepsilon\delta+\delta^2)
\frac{\text{csch}\big(\sqrt{\lambda}\delta\big)}{\sqrt\lambda}
.
\end{eqnarray*}
As a result, using that
\begin{eqnarray*}
\coth\big(\sqrt\lambda \delta\big)&=&
\frac{1}{\sqrt{\lambda}\delta}+\frac{\sqrt{\lambda}\delta}{3}+o(\delta)
\\
{\mbox{and }}\qquad 
\text{csch}\big(\sqrt{\lambda}\delta\big)&=&
\frac{1}{\sqrt{\lambda}\delta}
-\frac{\sqrt{\lambda}\delta}{6}+o(\delta),\end{eqnarray*}
we find that
\begin{eqnarray*}
m(\varepsilon+\delta)&=&
\left(\widetilde u-\frac{\Gamma u_B}\lambda\right)
\varepsilon^2\delta
+\frac2{\lambda}\left(
\delta\left(\overline{u}-\frac{\Gamma u_B}{\lambda}\right)
+\varepsilon(\overline{u}-\underline{u})\right)
\\&&
-\frac{\left(\overline{u}+\underline{u}-\frac{2\Gamma u_B}{\lambda}\right)
(2 + \lambda \varepsilon^2)}{\lambda^{3/2}}
\left(\frac{1}{\sqrt{\lambda}\delta}+\frac{\sqrt{\lambda}\delta}{3}+o(\delta)\right)
-\frac{\left(\overline{u}-\frac{\Gamma u_B}\lambda\right)(2\varepsilon+\delta)}{\lambda}
\\&&+
\frac{\left(\overline{u}+\underline{u}-\frac{2\Gamma u_B}{\lambda}\right)
(2 + \lambda \varepsilon^2) }{\lambda^{3/2}}
\left(\frac{1}{\sqrt{\lambda}\delta}
-\frac{\sqrt{\lambda}\delta}{6}+o(\delta)\right)
+\frac{
\left(\underline{u}-\frac{\Gamma u_B}\lambda\right)(2\varepsilon+\delta)}{\lambda}+o(\delta)\\
&=&
\left(\left(\widetilde u-\frac{\Gamma u_B}\lambda\right)
\varepsilon^2
+\frac2{\lambda}\left(\overline{u}-\frac{\Gamma u_B}{\lambda}\right)
-
\frac{\left(\underline{u}+\overline{u}-\frac{2\Gamma u_B}{\lambda}\right)
(2 + \lambda \varepsilon^2) }{3\lambda}
-\frac{\overline{u}-\frac{\Gamma u_B}\lambda}{\lambda}
\right.\\&&\left.\qquad
-
\frac{\left(\underline{u}+\overline{u}-\frac{2\Gamma u_B}{\lambda}\right)
(2 + \lambda \varepsilon^2) }{6\lambda}
+\frac{
\underline{u}-\frac{\Gamma u_B}\lambda}{\lambda}
\right)\delta +o(\delta)\\&=&
\left(\widetilde u-\frac{\overline{u}+\underline{u}}2\right)
\varepsilon^2\delta +o(\delta)
.
\end{eqnarray*}
From this and~\eqref{relationutilde} we infer that
\begin{equation}\label{zsrdcfgtygvbhu}
m(\varepsilon+\delta)<0  \quad \mbox{ if $\delta$ is small enough}.
\end{equation}

Moreover, we have 
\begin{equation*}
\lim_{R_0\to +\infty}m(R_0)=+\infty.
\end{equation*}
Combining this and~\eqref{zsrdcfgtygvbhu} we deduce that there exists at least one $R_0>\varepsilon$ such that $m(R_0)=0$, which proves~\eqref{<sercvgunjkop}.

Moreover, from~\eqref{zsrdcfgtygvbhu} we infer that
\[
\inf\big\{R_0>\varepsilon \mbox{ such that $\eqref{<sercvgunjkop}$ holds}\big\}=\min\big\{R_0>\varepsilon \mbox{ such that $\eqref{<sercvgunjkop}$ holds}\big\},
\]
which gives \eqref{drftguhikl}.

To conclude the proof of Proposition~\ref{propabcsPI}, we prove that if \eqref{<sercvgunjkop} is in force, then \eqref{kjinkmjnbhkmjbhkm,mjhukol8hbuijnklmò9} holds true. Assume by contradiction that it is not, namely
\begin{equation}\label{lokijhnubgyvftcdvgybkm,}
\widetilde u \le\frac{\Gamma u_B}{\lambda}.
\end{equation}
From this and \eqref{stationaryall}, we have that
\[
\begin{split}
u(r)-\widetilde u&= \frac{(\overline{u}-\widetilde u)
\sinh\big(\sqrt{\lambda}(r-\varepsilon)\big)}{\sinh\big(\sqrt{\lambda}(R_0-\varepsilon)\big)} +\frac{(\underline{u}-\widetilde u)\sinh\big(\sqrt{\lambda}(R_0-r)\big)}{\sinh\big(\sqrt{\lambda}(R_0-\varepsilon)\big)}\\
&+\frac{\left(\frac{\Gamma u_B}{\lambda}-\widetilde u\right)}{\sinh\big(\sqrt{\lambda}(R_0-\varepsilon)\big)}\left(\sinh\big(\sqrt{\lambda}(R_0-\varepsilon)\big)- \sinh\big(\sqrt{\lambda}(r-\varepsilon)\big) -\sinh\big(\sqrt{\lambda}(R_0-r)\big)\right).
\end{split}
\]

We claim that
\begin{equation}\label{okmjnvioonjvhnjoimòwkN}
\sinh\big(\sqrt{\lambda}(R_0-\varepsilon)\big)- \sinh\big(\sqrt{\lambda}(r-\varepsilon)\big) -\sinh\big(\sqrt{\lambda}(R_0-r)\big)>0 \quad\mbox{ for any } r\in (\varepsilon, R_0).
\end{equation}
Indeed, by Proposition \ref{Proplèpkmoò.,lmkbhul,ò.à-} (used with~$a:=\sqrt\lambda(R_0-r)$ and~$b:=\sqrt\lambda(r-\varepsilon)$)
and the fact that the function $t\mapsto\cosh(t)$ is strictly increasing for $t\in (0, +\infty)$, we have that
\begin{eqnarray*} \sinh\big(\sqrt{\lambda}(r-\varepsilon)\big)+\sinh\big(\sqrt{\lambda}(R_0-r)\big)&= &
2\sinh\left(\frac{\sqrt\lambda (R_0-\varepsilon)}{2}\right) \cosh\left(\frac{\sqrt\lambda (R_0+\varepsilon-2r)}{2}\right)
\\&<&2\sinh\left(\frac{\sqrt\lambda (R_0-\varepsilon)}{2}\right)\cosh\left(\frac{\sqrt\lambda (R_0-\varepsilon)}{2}\right).\end{eqnarray*}
Now, using Proposition \ref{Proplèpkmoò.,lmkbhul,ò.à-} (here applied with~$a:=\sqrt\lambda(R_0-\varepsilon$ and~$b:=0$),
\[
\sinh\big(\sqrt{\lambda}(r-\varepsilon)\big)+\sinh\big(\sqrt{\lambda}(R_0-r)\big) < \sinh\big(\sqrt{\lambda}(R_0-\varepsilon)\big).
\]
This proves the claim in \eqref{okmjnvioonjvhnjoimòwkN}.

Thus, by \eqref{relationutilde}, \eqref{lokijhnubgyvftcdvgybkm,} and \eqref{okmjnvioonjvhnjoimòwkN}, we deduce that
\[
u(r)-\widetilde u>0 \quad\mbox{ for any } r\in (\varepsilon, R_0), 
\]
which contradicts \eqref{<sercvgunjkop}.  This concludes the proof.
\end{proof}

We point out that, from Proposition \ref{propabcsPI} and \eqref{relationutilde}, one obtains the relation in  \eqref{relationutilde2}.

\section{Preliminary results towards the proof of Theorem~\ref{th:main1}}\label{dtfygubhijnkml,ò.-aklmSJNBH}

In this section we give some technical results which will be helpful in proving Theorem~\ref{th:main1}.

To start with, we use the comparison principle to prove that the solutions of~\eqref{mainfinal1} and~\eqref{freeboundary} are positive and bounded above by~$\overline{u}$.

\begin{lemma}\label{lemma0<u<utilde}
Assume~\eqref{u_0positiva} and~\eqref{relationutilde2}.
Let $(u(r,t), s(t))$ be a solution of~\eqref{mainfinal1} and~\eqref{freeboundary}.

Then, for all $\varepsilon\le r\le s(t)$ and~$t\ge0$,
\begin{equation}\label{prima0}
u(r,t)\ge \frac{\Gamma u_B}{\lambda}>0
\end{equation}
and
\begin{equation}\label{prima4}
u(r,t)\leq\overline{u}.
\end{equation}
\end{lemma}

\begin{proof} 
We set
$$ Lu:=\kappa u_t-\Delta u
+\frac{\chi_0}{|x|^2}x\cdot \nabla u+\lambda u
-\Gamma u_B.$$

For any $\varepsilon\le r\le s(t)$ and~$t\ge0$, let $v(r,t):=\frac{\Gamma u_B}{\lambda}$. Then, recalling~\eqref{u_0positiva} and~\eqref{relationutilde2}, we have
\begin{alignat*}{2}
&L v=0=Lu
\qquad  &&\mbox{if } \varepsilon<r<s(t), \\
&v(\varepsilon,t)=\frac{\Gamma u_B}{\lambda}<\underline{u}=u(\varepsilon,t),\quad v(s(t),t)=\frac{\Gamma u_B}{\lambda}<\overline{u}=u(s(t),t) \qquad  &&\mbox{if } t>0, \\
\mbox{and} \quad &v(r,0)=\frac{\Gamma u_B}{\lambda}\le u_0(r)=u(r,0) \qquad  &&\mbox{if } \varepsilon<r<s(0).
\end{alignat*}
By using the comparison principle
(see e.g.~\cite[Corollary~2.5]{MR1465184})
we thus obtain~\eqref{prima0}.

Now, we consider the function~$z(r,t):=\overline{u}$ for any~$\varepsilon\le r\le s(t)$ and~$t\ge0$.
Thus, we have
\begin{alignat*}{2}
&Lz
=\lambda\overline{u}-\Gamma u_B> 0=Lu
\qquad  &&\mbox{if } \varepsilon<r<s(t), \\
& z(\varepsilon,t)=\overline u>u(\varepsilon,t),\quad z(s(t),t)=\overline u=u(s(t),t) \qquad  &&\mbox{if } t>0,
\\
\mbox{and} \quad & z(r,0)=\overline{u}\ge u_0(r)=u(r,0) \qquad  &&\mbox{if } \varepsilon<r<s(0).
\end{alignat*}As a consequence, employing the
comparison principle we obtain~\eqref{prima4}.
\end{proof}

We now use~\eqref{freeboundary} to derive useful upper and lower bounds for the radius $s(t)$ and its first derivative.

\begin{lemma}\label{lemmastimespunto}
Assume~\eqref{u_0positiva} and~\eqref{relationutilde2}.
Let $(u(r,t), s(t))$ be a solution of~\eqref{mainfinal1} and~\eqref{freeboundary}.

Then, for all $t>0$,
\begin{equation}\label{stimespunto}
-\mu\left(\widetilde u-\frac{\Gamma u_B}{\lambda}\right)(s(t)-\varepsilon)\leq \dot s(t)\leq \mu(\overline{u}-\widetilde u)(s(t)-\varepsilon)
\end{equation}
and
\begin{equation}\label{stimespunto2}
(s(0)-\varepsilon)e^{-\mu\left(\widetilde u-\frac{\Gamma u_B}{\lambda}\right)t} +\varepsilon\leq s(t)\leq (s(0)-\varepsilon)e^{\mu(\overline{u}-\widetilde u) t}+\varepsilon .
\end{equation}
\end{lemma}

\begin{proof} We set $$
T:=\sup\{ t\in[0,+\infty) {\mbox{ s.t. }} s(t)>\varepsilon \}$$
and we remark that~$T\in(0,+\infty]$, since~$s(0)>\varepsilon$.

Our strategy is to establish~\eqref{stimespunto}
and~\eqref{stimespunto2} for all~$t\in[0,T)$. If we accomplish this,
then the first inequality in~\eqref{stimespunto2} yields that~$T=+\infty$,
hence giving~\eqref{stimespunto}
and~\eqref{stimespunto2} for all~$t\in[0,+\infty)$.

This approach is helpful, because for all~$t\in[0,T)$
the free boundary condition~\eqref{freeboundary} boils down to
\begin{equation}\label{QUEST}s^2(t)\dot s(t)=\mu\int_\varepsilon^{s(t)}\big(u(r,t)-\widetilde u\big)r^2\,dr.\end{equation}

Now,
thanks to Lemma~\ref{lemma0<u<utilde} we know that~$u\ge\frac{\Gamma u_B}{\lambda}$. Thus,
from~\eqref{QUEST} we deduce that, for all~$t\in[0,T)$,
\begin{equation*}\begin{split}&
\dot{s}(t)
=\frac{\mu}{s^2(t)}\int_\varepsilon^{s(t)}\big(u(r,t)-\widetilde u\big)r^2\,dr
\geq-\frac{\mu\left(\widetilde u-\frac{\Gamma u_B}{\lambda}\right)}{s^2(t)}\int_\varepsilon^{s(t)}r^2\,dr=
-\frac{\mu\left(\widetilde u-\frac{\Gamma u_B}{\lambda}\right)(s^3(t)-\varepsilon^3)}{3s^2(t)}\\&\qquad
=
-\mu\left(\widetilde u-\frac{\Gamma u_B}{\lambda}\right)(s(t)-\varepsilon)\frac{s^2(t)+\varepsilon s(t)+\varepsilon^2}{3s^2(t)}.\end{split}
\end{equation*}
Since~$\varepsilon<s(t)$, we have that
$$ \frac{s^2(t)+\varepsilon s(t)+\varepsilon^2}{3s^2(t)}=
\frac13\left(1+\frac{\varepsilon}{s(t)}+\frac{\varepsilon^2}{s^2(t)}\right)\le 1
$$ and therefore, for all~$t\in[0,T)$,
\begin{equation}\label{zxdrtfvghyuhnj1} \dot{s}(t)
\ge -\mu\left(\widetilde u-\frac{\Gamma u_B}{\lambda}\right)(s(t)-\varepsilon). \end{equation}

We observe that Lemma~\ref{lemma0<u<utilde} also tells us that~$u\le \overline{u}$. As a result,
using this information into~\eqref{QUEST} we get that, for all~$t\in[0,T)$,
\begin{equation*}
\begin{split}
&\dot{s}(t)\leq \frac{\mu(\overline{u}-\widetilde u)}{s^2(t)}\int_\varepsilon^{s(t)}r^2\,dr
=\frac{\mu(\overline{u}-\widetilde u)(s^3(t)-\varepsilon^3)}{3s^2(t)}
\\&\qquad=\frac{\mu}{3}(\overline{u}-\widetilde u)(s(t)-\varepsilon)\frac{s^2(t)+\varepsilon s(t)+\varepsilon^2}{s^2(t)}
\le \mu(\overline{u}-\widetilde u)(s(t)-\varepsilon).
\end{split}
\end{equation*}
Combining this and~\eqref{zxdrtfvghyuhnj1}, we obtain the desired result in~\eqref{stimespunto}, for all~$t\in[0,T)$.

Furthermore, by integrating~\eqref{stimespunto},  we deduce that
\[
-\mu\left(\widetilde u-\frac{\Gamma u_B}{\lambda}\right) t\leq \log(s(t)-\varepsilon)-\log(s(0)-\varepsilon)\leq \mu(\overline{u}-\widetilde u) t,
\]
which entails~\eqref{stimespunto2}, for all~$t\in[0,T)$, as desired.
\end{proof}

The next result will be useful later on in establishing the existence of a suitable lower bound for the radius~$s(t)$:

\begin{lemma}\label{lemmaT1}
Let $T_0>0$ and assume that~\eqref{relationutilde2} holds true. Let
\begin{equation}\label{gammadefn}
\gamma(t):=\overline{u}\, e^{-\frac{\lambda(t-T_0)}{\kappa}} \quad\mbox{ for any } t\ge T_0.
\end{equation}

Then, there exists $T_1>T_0$ such that
\[
\underline{u}-\widetilde u -\gamma(t)>0 \quad\mbox{ for any } t> T_1.
\]
Explicitly, one has
\[
T_1:= T_0 -\frac{\kappa}{\lambda}\log\left(\frac{\underline{u} -\widetilde u}{\overline{u}}\right).
\]
\end{lemma}

\begin{proof}
In light of assumption~\eqref{relationutilde}, we see that
\[
0<\frac{\underline{u} -\widetilde u}{\overline{u}}<1.
\]
From this and~\eqref{gammadefn}, we deduce that $\underline{u}-\widetilde u -\gamma(t)>0$ if and only if
\[
t>T_0 -\frac{\kappa}{\lambda}\log\left(\frac{\underline{u} -\widetilde u}{\overline{u}}\right),
\] from which the desired result follows.
\end{proof}

The aim is now to construct a suitable barrier to be employed in the proof of Theorem~\ref{th:main1}. To accomplish this goal, for all~$M>0$ (that will be specified later on in formula~\eqref{vcew3495863tfhfaefojgews}) we define,
for all~$\varepsilon <r\le s(t)$ and~$t>0$,
\begin{equation}\label{vdefinition}
v(r, t):= \frac{\left(\overline{u}-\frac{\Gamma u_B}{\lambda}\right)\sinh(M(r-\varepsilon)) +\left(\underline{u}-\frac{\Gamma u_B}{\lambda}\right)\,\sinh(M(s(t)-r))}{\sinh(M(s(t)-\varepsilon))}+\frac{\Gamma u_B}{\lambda}.
\end{equation}
The function~$v(r, t)$ solves a suitable differential equation, according to the next result:

\begin{lemma}\label{lemmasoluzionev}
Let $M>0$. Let $v(r, t)$ be defined as in~\eqref{vdefinition}.

Then, for~$\varepsilon<r\le s(t)$ and~$t>0$,
\[
\frac{\partial^2v}{\partial r^2}(r, t) - M^2 v(r, t)+\frac{M^2\Gamma u_B}{\lambda}=0.
\]

Moreover,
\begin{eqnarray*}
v(\varepsilon, t)=\underline u
\qquad{\mbox{and}}\qquad v(s(t),t)=\overline u.
\end{eqnarray*}
\end{lemma}

\begin{proof}
By inspection, one has
\[
\begin{aligned}
\frac{\partial^2v}{\partial r^2}(r, t)&=\frac{M^2\left(\left(\overline{u}-\frac{\Gamma u_B}{\lambda}\right)\sinh(M(r-\varepsilon)) +\left(\underline{u}-\frac{\Gamma u_B}{\lambda}\right)\sinh(M(s(t)-r))\right)}{\sinh(M(s(t)-\varepsilon))}\\
& = M^2\left( v(r, t)-\frac{\Gamma u_B}{\lambda}\right).
\end{aligned}
\]
This provides the desired result.
\end{proof}

\section{Proof of Theorem~\ref{th:main1}}\label{mkljnbgvyftcdrxsedctfvygbhunijm}

The proof of Theorem~\ref{th:main1} unfolds in several technical steps which are outlined in detail here below.

We start by proving that, regardless of the initial chosen time, there always exists a later time at which the size of~$s(t)$ exceeds a specified threshold.
To give an estimate for this threshold, we denote by $C$ a constant such that, for any $\eta \in (0, 1]$,
\begin{equation}\label{vbuewi58221wsd4trg7uij}
\begin{split}
&|\sinh(\eta)-\eta| \leq C\eta^3,\\
&|\cosh(\eta)-1|\leq C\eta^2,\\
&|\mbox{csch}(\eta)-\frac1{\eta}|\leq C\eta\\
{\mbox{and }}\quad &|\coth(\eta)-\frac1{\eta}| \leq C \eta,
\end{split}
\end{equation}
and set
\begin{equation}\label{vcew3495863tfhfaefojgews}
M:=\sqrt{\dfrac{\lambda\underline{u} + 2\kappa\mu(\overline{u}-\underline{u})(\widetilde u-\frac{\Gamma u_B}{\lambda})}{\underline{u}-\frac{\Gamma u_B}{\lambda}}}.
\end{equation}
The precise statement goes as follows:

\begin{proposition}\label{propstep1}
Assume that \eqref{u_0positiva} and \eqref{relationutilde2} hold true. Let $(u(r, t),s(t))$ be a solution of~\eqref{mainfinal1} and~\eqref{freeboundary} and let~$R_0$
be as in~\eqref{drftguhikl}.

Then, there exists $\delta_0>0$ such that, for any $t>0$, there exists $t_1>t$ satisfying
\begin{equation}\label{step1}
s(t_1)>\varepsilon+\delta_0.
\end{equation}
Explicitly, one has
\begin{equation}\label{delta0esplicita}
\begin{aligned}
\delta_0&:=\min\Bigg\{ \frac{R_0-\varepsilon}{2},\frac{1}{M},
{
\frac{1}{2M}\left(\frac{(\underline u-\widetilde u)(1-e^{-1})}{C\left(\overline u+\underline u -\frac{\Gamma u_B}{\lambda}\right)}\right)^\frac{1}{2},} \\
&\frac{1}{2M}{\scriptstyle \left(\frac{\kappa\mu\left(\widetilde u-\frac{\Gamma u_B}{\lambda}\right)(\overline{u}-\underline{u})}{C\left[\kappa\mu(\overline u-\widetilde u)\left(2\overline u+3\underline u-5\frac{\Gamma u_B}{\lambda} \right)
+\kappa\mu\left(\widetilde u-\frac{\Gamma u_B}{\lambda}\right)(\overline{u}-\underline{u})
+\kappa C\mu(\overline u-\widetilde u)\left(2\overline u+3\underline u-5\frac{\Gamma u_B}{\lambda} \right)+2(M^2-\lambda)\left(\overline u+\underline u -\frac{\Gamma u_B}{\lambda}\right)
\right]}\right)^\frac{1}{2}}, \\
&\frac{1}{2M}{\scriptstyle \left(\frac{\kappa\mu\left(\widetilde u-\frac{\Gamma u_B}{\lambda}\right)(\overline{u}-\underline{u})}{\kappa C\mu\left(\widetilde u-\frac{\Gamma u_B}{\lambda}\right)\left[\left(3\overline u+2\underline u-5\frac{\Gamma u_B}{\lambda}\right) 
+ C\left(3\overline u+4\underline u-7\frac{\Gamma u_B}{\lambda} \right)+ C^2\left(\overline u+\underline u-2\frac{\Gamma u_B}{\lambda} \right)\right]+2C(M^2-\lambda)\left(\overline u+\underline u-\frac{\Gamma u_B}{\lambda} \right)}\right)^\frac{1}{2}}
\Bigg\},
\end{aligned}
\end{equation}
where $C$ is as in \eqref{vbuewi58221wsd4trg7uij} and $M$ is defined in \eqref{vcew3495863tfhfaefojgews}.
\end{proposition}

\begin{proof}
Assume by contradiction that this is not the case, namely that for  $\delta_0$ as in \eqref{delta0esplicita} there exists $t_0>0$ such that
\begin{equation}\label{contrstep1}
s(t)\le\varepsilon+\delta_0 \quad\mbox{ for any } t\ge t_0.
\end{equation}

We point out that, by~\eqref{relationutilde2} and \eqref{vcew3495863tfhfaefojgews},
\begin{equation}\label{98765ujhygfhueoirshyor}\begin{split}&
M^2-\lambda=\frac{\lambda\underline{u} + 2\kappa\mu(\overline{u}-\underline{u})(\widetilde u-\frac{\Gamma u_B}{\lambda})}{\underline{u}-\frac{\Gamma u_B}{\lambda}}-\lambda
\\&\qquad=\frac{  2\kappa\mu(\overline{u}-\underline{u})(\widetilde u-\frac{\Gamma u_B}{\lambda})+\Gamma u_B}{\underline{u}-\frac{\Gamma u_B}{\lambda}}
>0.
\end{split}
\end{equation}
hen, we consider the function~$v(r, t)$ as in~\eqref{vdefinition}. Thus, by Lemma~\ref{lemmasoluzionev}, we get
\begin{equation}\label{mnbvc09876iuytrewjhgf0}
\kappa v_t(r, t)- \partial^2_r v(r, t) +\lambda v(r, t) = \kappa v_t(r, t)-(M^2-\lambda) v(r, t)+\frac{M^2\Gamma u_B}{\lambda}.
\end{equation}

Now we observe that
\begin{equation}\label{vtspuntocambiasegno}
\begin{split}&
v_t(r, t)\\=\; &\frac{\left(\underline{u}-\frac{\Gamma u_B}{\lambda}\right)M\dot s(t)\cosh(M(s(t)-r))}{\sinh(M(s(t)-\varepsilon))} \\
& - \frac{\left[\left(\overline{u}-\frac{\Gamma u_B}{\lambda}\right)\sinh(M(r-\varepsilon)) + \left(\underline{u}-\frac{\Gamma u_B}{\lambda}\right)\,\sinh(M(s(t)-r))\right]M\dot s(t)\cosh(M(s(t)-\varepsilon))}{\sinh^2(M(s(t)-\varepsilon))}  \\
=\;&M\dot s(t)\,\mbox{csch}(M(s(t)-\varepsilon))\Bigg[\left(\underline{u}-\frac{\Gamma u_B}{\lambda}\right)\cosh(M(s(t)-r)) \\
&  -\left(\left(\overline{u}  -\frac{\Gamma u_B}{\lambda}\right)\sinh(M(r-\varepsilon))+\left(\underline{u}-\frac{\Gamma u_B}{\lambda}\right)\,\sinh(M(s(t)-r))\right)\coth(M(s(t)-\varepsilon))\Bigg] .
\end{split}
\end{equation}
We distinguish two cases, either $\dot s (t)\geq 0$ or $\dot s(t)<0$. We deal with the first case.
 Accordingly, by \eqref{delta0esplicita} and \eqref{vbuewi58221wsd4trg7uij}, from \eqref{vtspuntocambiasegno} we get 
\[
\begin{split}&
v_t(r, t)\\\leq\;&M\dot s(t)\mbox{csch}(M(s(t)-\varepsilon))\\&
\times\Bigg[\left(\underline{u}-\frac{\Gamma u_B}{\lambda}\right)\big(
1 + CM^2(s(t)-r)^2
\big) \\
&  \quad-\left(\left(\overline{u}  -\frac{\Gamma u_B}{\lambda}\right)
\big(M(r-\varepsilon) -CM^3(r-\varepsilon)^3
\big)+\left(\underline{u}-\frac{\Gamma u_B}{\lambda}\right)
\big(M(s(t)-r)-CM^3(s(t)-r)^3 )\big)\right)
\\
&\qquad\quad\times\left(
\frac1{M(s(t)-\varepsilon)} -CM(s(t)-\varepsilon)\right)
\Bigg]\\
=\;&M\dot s(t)\mbox{csch}(M(s(t)-\varepsilon))\\&
\times\Bigg[\underline{u}-\frac{\Gamma u_B}{\lambda}-
 \left(\overline{u}  -\frac{\Gamma u_B}{\lambda}\right)\frac{r-\varepsilon}{s(t)-\varepsilon}
-\left(\underline{u}-\frac{\Gamma u_B}{\lambda}\right)\frac{s(t)-r}{s(t)-\varepsilon} \\
&+CM^2 \left(\underline{u}  -\frac{\Gamma u_B}{\lambda}\right)\left((s(t)-r)^2+(s(t)-r)(s(t)-\varepsilon)+\frac{(s(t)-r)^3}{s(t)-\varepsilon}\right) \\
&+CM^2 \left(\overline{u}  -\frac{\Gamma u_B}{\lambda}\right)\left((r-\varepsilon)(s(t)-\varepsilon)+\frac{(r-\varepsilon)^3}{s(t)-\varepsilon}\right) \\
&-C^2M^4\left(\left(\overline{u}  -\frac{\Gamma u_B}{\lambda}\right) (r-\varepsilon)^3(s(t)-\varepsilon)+ \left(\underline{u}  -\frac{\Gamma u_B}{\lambda}\right)(s(t)-r)^3(s(t)-\varepsilon)\right)
\Bigg].
\end{split}
\]
Recalling that $\varepsilon\leq r \leq s(t)$, we have
\[
\begin{aligned}
v_t(r,t)&\leq M\dot s(t)\mbox{csch}(M(s(t)-\varepsilon))\Bigg[-\frac{(\overline u -\underline u)(r-\varepsilon)}{s(t)-\varepsilon} +CM^2(s(t)-\varepsilon)^2 \left(2\overline u+3\underline u-5\frac{\Gamma u_B}{\lambda} \right)
\Bigg]
\end{aligned}
\]
From Lemma~\ref{lemmastimespunto} and \eqref{vbuewi58221wsd4trg7uij},
 we can write
\begin{equation}\label{mnbvc09876iuytrewjhgf}
\begin{split}
v_t(r, t)&\leq M\mu \left(\widetilde u -\frac{\Gamma u_B}{\lambda} \right)(\overline u-\underline u)(s(t)-\varepsilon)\mbox{csch}(M(s(t)-\varepsilon)) \\
&\quad +CM^3\mu (\overline u -\widetilde u)\left(2\overline u+3\underline u-5\frac{\Gamma u_B}{\lambda} \right)(s(t)-\varepsilon)^3\mbox{csch}(M(s(t)-\varepsilon)) \\
&=\mu\left(\widetilde u -\frac{\Gamma u_B}{\lambda} \right)(\overline u-\underline u)\\
&\quad+CM^2\mu\left[(\overline u-\widetilde u)\left(2\overline u+3\underline u-5\frac{\Gamma u_B}{\lambda} \right)+ \left(\widetilde u -\frac{\Gamma u_B}{\lambda} \right)(\overline u-\underline u)\right](s(t)-\varepsilon)^2 \\
&\quad +C^2M^4\mu(\overline u-\widetilde u)\left(2\overline u+3\underline u-5\frac{\Gamma u_B}{\lambda} \right)(s(t)-\varepsilon)^4.
\end{split}
\end{equation}

Moreover,
\begin{eqnarray*}
&&  v(r, t)-\frac{\Gamma u_B}{\lambda}
=
\frac{\left(\overline{u}-\frac{\Gamma u_B}{\lambda}\right)\sinh(M(r-\varepsilon)) + \left(\underline{u}-\frac{\Gamma u_B}{\lambda}\right)\,\sinh(M(s(t)-r))}{\sinh(M(s(t)-\varepsilon))}.
\end{eqnarray*}
Thus, by~\eqref{vbuewi58221wsd4trg7uij},
\begin{eqnarray*}
&&  v(r, t)-\frac{\Gamma u_B}{\lambda}
\\&=&
\left[\left(\overline{u}-\frac{\Gamma u_B}{\lambda}\right) \big(M(r-\varepsilon)-CM^3(r-\varepsilon)^3)\big) 
+ \left(\underline{u}-\frac{\Gamma u_B}{\lambda}\right)\big(M(s(t)-r)-CM^3(s(t)-r)^3)
\big)\right]\\&&\qquad\times
\left(\frac1{M(s(t)-\varepsilon)}
-CM(s(t)-\varepsilon)\right)\\
&\geq & 
\left[\overline{u}M(r-\varepsilon) +\underline{u}M(s(t)-r)-
\frac{M\Gamma u_B}{\lambda}(s(t)-\varepsilon)-CM^3\left(\overline u+\underline u -2\frac{\Gamma u_B}{\lambda}\right)(s(t)-\varepsilon)^3
\right]\\&&\qquad\times
\left(\frac1{M(s(t)-\varepsilon)}
-CM(s(t)-\varepsilon)\right)\\&=&
\frac{\overline{u}M(r-\varepsilon) +\underline{u}M(s(t)-r)}{M(s(t)-\varepsilon)}
-\frac{\Gamma u_B}{\lambda} -CM^2(\overline u (r-\varepsilon)(s(t)-\varepsilon)+\underline u(s(t)-r)(s(t)-\varepsilon))\\
&&\quad -CM^2\left(\overline u+\underline u -2\frac{\Gamma u_B}{\lambda}\right)(s(t)-\varepsilon)^2+C^2M^4\left(\overline u+\underline u -2\frac{\Gamma u_B}{\lambda}\right)(s(t)-\varepsilon)^4 \\
&\ge&
\underline{u}
-\frac{\Gamma u_B}{\lambda} -2CM^2\left(\overline u+\underline u -\frac{\Gamma u_B}{\lambda}\right)(s(t)-\varepsilon)^2+C^2M^4\left(\overline u+\underline u -2\frac{\Gamma u_B}{\lambda}\right)(s(t)-\varepsilon)^4
.
\end{eqnarray*}
Therefore,
\begin{equation}\label{758gfkjhgewaqy5t38ufgggggg8765}
v(r, t)\ge \underline{u}-2CM^2\left(\overline u+\underline u -\frac{\Gamma u_B}{\lambda}\right)(s(t)-\varepsilon)^2+C^2M^4\left(\overline u+\underline u -2\frac{\Gamma u_B}{\lambda}\right)(s(t)-\varepsilon)^4.
\end{equation}

Plugging this and~\eqref{mnbvc09876iuytrewjhgf} into~\eqref{mnbvc09876iuytrewjhgf0}, and recalling~\eqref{98765ujhygfhueoirshyor} and \eqref{delta0esplicita},
we conclude that
\begin{eqnarray*}&&
\kappa v_t(r, t)- \partial^2_r v(r, t) +\lambda v(r, t) \\&\le& 
\kappa\mu\left(\widetilde u-\frac{\Gamma u_B}{\lambda} \right)(\overline{u}-\underline{u}) -\underline{u}(M^2-\lambda)+\frac{M^2\Gamma u_B}{\lambda}\\
&&\quad+\kappa CM^2\mu\left[(\overline u-\widetilde u)\left(2\overline u+3\underline u-5\frac{\Gamma u_B}{\lambda} \right)+\left(\widetilde u-\frac{\Gamma u_B}{\lambda} \right)(\overline{u}-\underline{u})\right](s(t)-\varepsilon)^2 \\
&&\quad +\kappa C^2M^4\mu(\overline u-\widetilde u)\left(2\overline u+3\underline u-5\frac{\Gamma u_B}{\lambda} \right)(s(t)-\varepsilon)^4
\\
&&\quad+(M^2-\lambda)\left(2CM^2\left(\overline u+\underline u -\frac{\Gamma u_B}{\lambda}\right)(s(t)-\varepsilon)^2-C^2M^4\left(\overline u+\underline u -2\frac{\Gamma u_B}{\lambda}\right)(s(t)-\varepsilon)^4 \right) \\
&\leq&
\kappa\mu\left(\widetilde u-\frac{\Gamma u_B}{\lambda} \right)(\overline{u}-\underline{u}) 
-   2\kappa\mu\left(\widetilde u-\frac{\Gamma u_B}{\lambda} \right)(\overline{u} -\widetilde u)  \\
&&\quad+\kappa CM^2\mu\left[(\overline u-\widetilde u)\left(2\overline u+3\underline u-5\frac{\Gamma u_B}{\lambda} \right)+\left(\widetilde u-\frac{\Gamma u_B}{\lambda} \right)(\overline{u}-\underline{u})\right](s(t)-\varepsilon)^2 \\
&&\quad +\kappa C^2M^4\mu(\overline u-\widetilde u)\left(2\overline u+3\underline u-5\frac{\Gamma u_B}{\lambda} \right)(s(t)-\varepsilon)^4 \\
&&\quad +2CM^2(M^2-\lambda)\left(\overline u+\underline u -\frac{\Gamma u_B}{\lambda}\right)(s(t)-\varepsilon)^2
\\&\leq&-\kappa\mu\left(\widetilde u-\frac{\Gamma u_B}{\lambda} \right)(\overline{u}-\underline{u}) +CM^2\Bigg[\kappa\mu(\overline u-\widetilde u)\left(2\overline u+3\underline u-5\frac{\Gamma u_B}{\lambda} \right)+\kappa \mu \left(\widetilde u-\frac{\Gamma u_B}{\lambda} \right)(\overline{u}-\underline{u}) \\
&&\quad +\kappa C\mu(\overline u-\widetilde u)\left(2\overline u+3
\underline u-5\frac{\Gamma u_B}{\lambda} \right)+2(M^2-\lambda)\left(\overline u+\underline u -\frac{\Gamma u_B}{\lambda}\right)
\Bigg](s(t)-\varepsilon)^2 .
\end{eqnarray*}
As a consequence, taking $\delta_0$ as in \eqref{delta0esplicita}, we find that
\begin{equation}\label{veqn<01}
\kappa v_t(r, t)- \partial^2_r v(r, t) +\lambda v(r, t) <0
\quad \mbox{for~$ \varepsilon<r<s(t)$ and~$t\ge t_0$ with }\dot s(t)\geq 0.
\end{equation}

We now deal with the case $\dot s(t)< 0$. From \eqref{vtspuntocambiasegno}, we now have
\[
\begin{aligned}
&v_t(r, t)\leq M(-\dot s(t))\mbox{csch}(M(s(t)-\varepsilon))\\
& \quad
\times\Bigg[\left(\underline{u}-\frac{\Gamma u_B}{\lambda}\right)\big(
-1 + CM^2(s(t)-r)^2
\big) \\
&  \quad+\left(
 \left(\overline{u}  -\frac{\Gamma u_B}{\lambda}\right)
\big(M(r-\varepsilon) +CM^3(r-\varepsilon)^3
\big)+\left(\underline{u}-\frac{\Gamma u_B}{\lambda}\right)
\big(M(s(t)-r)+CM^3(s(t)-r)^3 )\big)\right)
\\
&\quad\times\left(
\frac1{M(s(t)-\varepsilon)} +CM(s(t)-\varepsilon)\right)
\Bigg] \\ 
&\leq M(-\dot s(t))\mbox{csch}(M(s(t)-\varepsilon))\Bigg[\frac{(\overline u-\underline u)(r-\varepsilon)}{s(t)-\varepsilon}
+CM^2\left(2\overline u+3\underline u-5\frac{\Gamma u_B}{\lambda}\right)(s(t)-\varepsilon)^2 \\
&\quad +C^2M^4\left(\overline u+\underline u-2\frac{\Gamma u_B}{\lambda}\right)(s(t)-\varepsilon)^4
\Bigg] \\
&\leq  \mu\left(\widetilde u-\frac{\Gamma u_B}{\lambda} \right)(\overline u-\underline u)
+CM^2\mu\left(\widetilde u-\frac{\Gamma u_B}{\lambda} \right)\left(3\overline u+2\underline u-5\frac{\Gamma u_B}{\lambda} \right)(s(t)-\varepsilon)^2 \\
&\quad +C^2M^4\mu\left(\widetilde u-\frac{\Gamma u_B}{\lambda} \right)\left(3\overline u+4\underline u-7\frac{\Gamma u_B}{\lambda} \right)(s(t)-\varepsilon)^4 \\
&\quad+C^3M^6\mu\left(\widetilde u-\frac{\Gamma u_B}{\lambda} \right)\left(\overline u+\underline u-2\frac{\Gamma u_B}{\lambda} \right)(s(t)-\varepsilon)^6.
\end{aligned}
\]
Using this and \eqref{758gfkjhgewaqy5t38ufgggggg8765} in 
\eqref{mnbvc09876iuytrewjhgf0}, by \eqref{98765ujhygfhueoirshyor} and \eqref{delta0esplicita} we infer that
\[
\begin{aligned}
&\kappa v_t(r, t)- \partial^2_r v(r, t) +\lambda v(r, t) \\
&\leq \kappa \mu \left(\widetilde u-\frac{\Gamma u_B}{\lambda} \right)(\overline u-\underline u)-(M^2-\lambda)\underline{u}+\frac{M^2\Gamma u_B}{\lambda}\\
&\quad+\kappa CM^2\mu\left(\widetilde u-\frac{\Gamma u_B}{\lambda} \right)\left(3\overline u+2\underline u-5\frac{\Gamma u_B}{\lambda} \right)(s(t)-\varepsilon)^2 \\
&\quad +\kappa C^2M^4\mu\left(\widetilde u-\frac{\Gamma u_B}{\lambda} \right)\left(3\overline u+4\underline u-7\frac{\Gamma u_B}{\lambda} \right)(s(t)-\varepsilon)^4 \\
&\quad
+\kappa C^3M^6\mu\left(\widetilde u-\frac{\Gamma u_B}{\lambda} \right)\left(\overline u+\underline u-2\frac{\Gamma u_B}{\lambda} \right)(s(t)-\varepsilon)^6 \\
&\quad +2CM^2(M^2-\lambda)\left(\overline u+\underline u-\frac{\Gamma u_B}{\lambda} \right)(s(t)-\varepsilon)^2
-C^2M^4(M^2-\lambda)\left(\overline u+\underline u-\frac{\Gamma u_B}{\lambda} \right)(s(t)-\varepsilon)^4 \\
&\leq -\kappa \mu \left(\widetilde u-\frac{\Gamma u_B}{\lambda} \right)(\overline u-\underline u)+2CM^2(M^2-\lambda)\left(\overline u+\underline u-\frac{\Gamma u_B}{\lambda} \right)(s(t)-\varepsilon)^2\\
&\quad +\kappa CM^2 \mu\left(\widetilde u-\frac{\Gamma u_B}{\lambda} \right)\Bigg[3\overline u+4\underline u-7\frac{\Gamma u_B}{\lambda} +C\left(3\overline u+4\underline u-7\frac{\Gamma u_B}{\lambda} \right)+C^2\left(\overline u+\underline u-2\frac{\Gamma u_B}{\lambda} \right) \Bigg]\\
&\quad \times (s(t)-\varepsilon)^2.
\end{aligned}
\]
Then, taking $\delta_0$ as in \eqref{delta0esplicita}, we have that
\begin{equation}\label{veqn<02}
\kappa v_t(r, t)- \partial^2_r v(r, t) +\lambda v(r, t) <0
\quad \mbox{for~$ \varepsilon<r<s(t)$ and~$t\ge t_0$ with }\dot s(t)< 0.
\end{equation}
Combinig \eqref{veqn<01} and \eqref{veqn<02} we get that
\begin{equation}\label{veqn<0}
\kappa v_t(r, t)- \partial^2_r v(r, t) +\lambda v(r, t) <0
\quad \mbox{for~$ \varepsilon<r<s(t)$ and~$t\ge t_0$}.
\end{equation}

We now consider the function $\gamma$ given by~\eqref{gammadefn} (here with~$T_0:=t_0$) and, for~$\varepsilon<r<s(t)$ and~$ t\ge t_0$, we set
\begin{equation}\label{wdefn}
w(r, t):= u(r, t)-v(r, t) + \gamma(t) .
\end{equation}
By~\eqref{gammadefn}, we get
\[
\dot{\gamma}(t)=-\frac{\lambda\overline{u}}{\kappa}e^{-\frac{\lambda(t-t_0)}{\kappa}}
\] and therefore~$\kappa\dot\gamma(t)+\lambda\gamma(t)=0$.

Accordingly, recalling that $u(r, t)$ solves~\eqref{mainfinal1},
\[
\begin{split}
&\kappa w_t(r, t)-\partial^2_r w(r, t) +\lambda w(r, t)-\Gamma u_B \\=\;& \kappa u_t(r, t)-\partial^2_r u(r, t)+\lambda u(r, t)-\Gamma u_B
-\kappa v_t(r, t)+\partial^2_r v(r, t) -\lambda v(r, t) +\kappa\dot{\gamma}(t) +\lambda\gamma(t)\\
=\;&-\kappa v_t(r, t)+\partial^2_r v(r, t)   -\lambda v(r, t).
\end{split}
\]
From this and~\eqref{veqn<0}, we infer that, for~$\varepsilon<r<s(t)$ and~$t\ge t_0$,
\begin{equation}\label{weqn<0}
\kappa w_t(r, t)-\partial^2_r w(r, t)  +\lambda w(r, t)>0 .
\end{equation}

Furthermore, recalling Lemma~\ref{lemmasoluzionev}, if $r=s(t)$ and $t\ge t_0$, then
$$w(s(t), t)=u(s(t), t)-v(s(t), t) + \gamma(t)
=\overline u-\overline u +\gamma(t) = \gamma(t)>0$$
and, if $r=\varepsilon$ and $t\ge t_0$, then
$$ w(\varepsilon, t)= u(\varepsilon, t)-v(\varepsilon, t) +\gamma(t)= \underline u-\underline u+\gamma(t)=\gamma(t)>0.$$
Also, if $t=t_0$ and $\varepsilon<r<s(t_0)$, then, by Lemma~\ref{lemma0<u<utilde},
$$ w(r, t_0)= u(r, t_0)-v(r, t_0) +\gamma(t_0)
=u(r, t_0)-v(r, t_0) +
\overline{u} \ge
\frac{\Gamma u_B}\lambda -\overline{u}+\overline{u}>0.$$

Accordingly, thanks to~\eqref{weqn<0} and the boundary conditions above, we are in the position to apply 
the comparison principle
(see~\cite[Corollary 2.5]{MR1465184}) and deduce that $w(r, t)>0$ for~$\varepsilon<r<s(t)$ and~$t\ge t_0$. Namely,
in light of~\eqref{wdefn}, for~$\varepsilon<r<s(t)$ and~$t\ge t_0$,
\begin{equation*}
u(r, t)>v(r, t) - \gamma(t) .
\end{equation*}

We now use this fact into the free boundary condition~\eqref{freeboundary} (recall also footnote~\ref{ATTEFOO223er} on page~\pageref{ATTEFOO223er})
to obtain that, for any $t\ge t_0$,
$$ \frac{ s^2(t) \dot{s}(t)}{\mu}  =\int_{\varepsilon}^{s(t)} (u(r, t)-\widetilde u) r^2 \,dr > \int_{\varepsilon}^{s(t)} \big(v(r, t)-\gamma(t)-\widetilde u\big) r^2 \,dr .$$
Hence, recalling~\eqref{758gfkjhgewaqy5t38ufgggggg8765},
\begin{eqnarray*}\frac{s^2(t) \dot{s}(t)}{\mu} &>& \int_{\varepsilon}^{s(t)}
\left( \underline{u}-\widetilde u-\gamma(t)-2CM^2\left(\overline u+\underline u -\frac{\Gamma u_B}{\lambda}\right)(s(t)-\varepsilon)^2\right) r^2 \,dr\\&=&
\left( \underline{u}-\widetilde u-\gamma(t)-2CM^2\left(\overline u+\underline u -\frac{\Gamma u_B}{\lambda}\right)(s(t)-\varepsilon)^2\right) \frac{s^3(t)-\varepsilon^3}3.
\end{eqnarray*}
Now, we take $\delta_0$ as in \eqref{delta0esplicita} and  $t_1>t_0$ as in Lemma~\ref{lemmaT1}, given explicitly by
\[
t_1:= t_0 -\frac{\kappa}{\lambda}\log\left(\frac{\underline{u} -\widetilde u}{\overline{u}}\right).
\]
Then, for any\footnote{In fact, one can choose here~\(t\ge t_1+T_\star\) for any \(T_\star>0\) (e.g., \(t\ge t_1+1\)). We simply set \(T_\star := \frac{\kappa}{\lambda}\) since this ratio has the dimensions of time and serves as a natural, intrinsic time-like parameter. Naturally, if one wishes to optimize the estimates on~\(\delta _{0}\), it would be beneficial to make a more refined choice here, as well as a more bespoke selection of the Taylor expansion domains in \eqref{vbuewi58221wsd4trg7uij}.}
$t\ge t_1+{\frac{\kappa}{\lambda}}$,
\[
2CM^2\left(\overline u+\underline u -\frac{\Gamma u_B}{\lambda}\right)(s(t)-\varepsilon)^2\leq \frac{\underline{u}-\widetilde u-\gamma\left(t_1+{\frac{\kappa}{\lambda}}\right)}{2}\leq \frac{\underline{u}-\widetilde u-\gamma(t)}{2},
\]
where
\[
\gamma\left(t_1+{\frac{\kappa}{\lambda}}\right)=\overline u e^{-\frac{\lambda}{\kappa}\left(t_1+{\frac{\kappa}{\lambda}}-t_0\right)}
=(\underline u-\widetilde u)e^{-1}.
\]
Thus, we infer that 
\[
\frac{ s^2(t) \dot{s}(t)}{\mu}\geq \frac{(\underline{u}-\widetilde u-\gamma(t))(s^3(t)-\varepsilon^3)}{6}
>0\quad\mbox{ for any } t> t_1+{\frac{\kappa}{\lambda}}.
\]
From this, one has
\[
\dot{s}(t)>0\quad\mbox{ for any } t> t_1+{\frac{\kappa}{\lambda}},
\]
namely $s(t)$ is a monotone increasing function for any $t> t_1+{\frac{\kappa}{\lambda}}$.  

Accordingly, we can apply~\cite[Chapter~6, Theorem~2]{MR181836} and deduce that
\[
\lim_{t\to+\infty} u(r, t) = u_\star(r),
\]
being $u_\star(r)$ the stationary solution in~\eqref{stationaryall}.
By Proposition~\ref{propabcsPI}, we infer that
\[
\lim_{t\to+\infty} s(t)\ge R_0,
\]
which is in contradiction with~\eqref{contrstep1}. This concludes the proof.
\end{proof}

Next, we prove that the radius $s(t)$ cannot go below a certain threshold, which is strictly greater then the radius of the necrotic core. In order to do this, it is useful to define the constant
\begin{equation}\label{defMtilde}
\widetilde M:=\sqrt{2\kappa\mu \left(\widetilde u-\frac{\Gamma u_B}{\lambda} \right)\left(\frac{\lambda}{\Gamma u_B}(\overline u-\underline u)+\log\frac{\overline u \lambda}{\Gamma u_B}\right)}.
\end{equation}

\begin{proposition}\label{propstep2}
Assume~\eqref{u_0positiva} and~\eqref{relationutilde2}.
Let $(u(r,t), s(t))$ be a solution of~\eqref{mainfinal1} and~\eqref{freeboundary}.

Let~$\delta_0$ be as in Proposition~\ref{propstep1}, $C$ as in \eqref{vbuewi58221wsd4trg7uij}, and $\widetilde M$ as in \eqref{defMtilde}.

Let also
\begin{equation}\label{delta1esplicita}{\footnotesize
\begin{aligned}
\delta_1&:=\min\Bigg\{\frac{\delta_0}{4},\frac{s(0)-\varepsilon}{2},\frac{1}{\widetilde M},
\frac{1}{2\widetilde M}\left(\frac{\underline u-\widetilde u}{4\underline u}\right)^\frac{1}{2}, \\
&\frac{1}{2\widetilde M}\left(\frac{\kappa\mu\left(\widetilde u-\frac{\Gamma u_B}{\lambda} \right)(\overline u-\underline u)}{C\left[\kappa\mu\left(\widetilde u-\frac{\Gamma u_B}{\lambda} \right)(\overline u-\underline u)+\kappa\mu (\overline u-\widetilde u)(2\overline u +3\underline u)
+\kappa C \mu(\overline u-\widetilde u)(2\overline u+3\underline u)+\frac{4\Gamma u_B \widetilde M^2}{\lambda}
\right]}\right)^\frac{1}{2}, \\
&\frac{1}{2\widetilde M}\left(\frac{\kappa\mu\left(\widetilde u-\frac{\Gamma u_B}{\lambda} \right)(\overline u-\underline u)}{C\left[\kappa\mu\left(\widetilde u-\frac{\Gamma u_B}{\lambda} \right)\left(3\overline u+2\underline u+C(3\overline u+4\underline u)+C^2(\overline u+\underline u) \right)+\frac{4\Gamma u_B \widetilde M^2}{\lambda}
\right]}\right)^\frac{1}{2}
\Bigg\}.
\end{aligned}}
\end{equation}

Then,
\begin{equation}\label{eqstep2}
s(t)\geq \varepsilon+\delta_1 \quad \mbox{ for any }t>0.
\end{equation}
\end{proposition}

\begin{proof}
If~$s(t)> \varepsilon+2\delta_1$ for any $t>0$, then we are done. If instead~$s(t)< \varepsilon+2\delta_1<\varepsilon+\delta_0$ for any~$t>0$, then we get a contradiction with Proposition~\ref{propstep1}.
Accordingly, we can suppose that there exists~$t_1>0$ such that~$s(t_1)=\varepsilon+2\delta_1$ and $s(t)\geq \varepsilon+\delta_1$ for any $t\in [0,t_1]$.

In order to prove~\eqref{eqstep2}, we assume by contradiction that there exists~$t_2>t_1$ such that
\begin{equation}\label{zxdrtfcfgty}
s(t_2)=\varepsilon+\delta_1  \qquad \mbox{and}\qquad \dot s(t_2)< 0.
\end{equation}
Without loss of generality, we can also assume that
\begin{equation}\label{zxdrtfcfgty1}
\delta_1<s(t)-\varepsilon<2\delta_1 \qquad \mbox{if } t_1<t<t_2.
\end{equation}

We recall that, thanks to Lemma~\ref{lemma0<u<utilde},
\begin{equation}\label{xcftyvbhjmk}
u(r,t_1)\geq \frac{\Gamma u_B}{\lambda}>0
\quad \mbox{ if } \varepsilon<r<s(t_1).
\end{equation}

Now, we define the domain 
\[
D:=\big\{(r,t)\in \R^2 \mbox{ such that } \varepsilon<r<s(t),\,\, t_1<t<t_2\big\}
\]
and we look for a subsolution of~\eqref{mainfinal1} in~$D$. In order to do this, we set 
\[
v_1(r,t):=\frac{\sinh(\widetilde M(r-\varepsilon))}{\sinh( \widetilde M(s(t)-\varepsilon))}\qquad \mbox{ and }\qquad v_2(r,t):=\frac{\sinh(\widetilde M(s(t)-r))}{\sinh( \widetilde M(s(t)-\varepsilon))}
\]
where $\widetilde M$ is defined in \eqref{defMtilde}. We also let
\[
\underline{u}(t):=\frac{\Gamma u_B}{\lambda} e^{\underline{N} (t-t_1)} \qquad \mbox{ and }\qquad \overline{u}(t):=\frac{\Gamma u_B}{\lambda} e^{\overline{N} (t-t_1)},
\]
where
\[
\underline{N}:=\frac{1}{t_2-t_1}\log\frac{\underline{u}\lambda}{\Gamma u_B}\qquad \mbox{ and }\qquad \overline{N}:=\frac{1}{t_2-t_1}\log\frac{\overline{u}\lambda}{\Gamma u_B}.
\]
It is worth noting that, from these definitions and~\eqref{relationutilde2}, it follows that $\overline{N}\geq \underline{N}$ and $\overline{u}(t)\geq \underline{u}(t)$ for any $t\in (t_1,t_2)$. In addition, the function~$\overline{u}(t)- \underline{u}(t)$ is increasing in~$t$. Hence, for any~$t\in (t_1,t_2)$,
\begin{equation}\label{xdftygvbnj}
\overline{u}(t)- \underline{u}(t)\leq \overline{u}(t_2)- \underline{u}(t_2)=\overline{u}- \underline{u}.
\end{equation}
Moreover, we claim that
\begin{equation}\label{Nbounded}
\overline N\leq 2\mu\left(\widetilde u-\frac{\Gamma u_B}{\lambda} \right) \log\frac{\overline{u}\lambda}{\Gamma u_B},
\end{equation}
and to prove this we show that
\begin{equation}\label{claimt1t2}
t_2-t_1\geq \frac{1}{2\mu \widetilde u}.
\end{equation}
First, we define a function
\[
f(t):=s(t_1)+\inf_{\tau\in[t_1,t_2]}\dot s(\tau)(t-t_1), \quad t\in[t_1,t_2],
\]
and observe that $f(t)\leq s(t)$ for any $t\in[t_1,t_2]$.
From \eqref{stimespunto} and \eqref{zxdrtfcfgty1} we have
\[
\dot s(t)\geq -\mu\left(\widetilde u-\frac{\Gamma u_B}{\lambda} \right) (s(t)-\varepsilon)\geq -2\mu\left(\widetilde u-\frac{\Gamma u_B}{\lambda} \right) \delta_1, \quad \mbox{for any }t\in[t_1,t_2].
\]
Thus, we can consider the function
\[
\widetilde f(t):=s(t_1)-2\mu\left(\widetilde u-\frac{\Gamma u_B}{\lambda} \right) \delta_1(t-t_1), \quad t\in[t_1,t_2],
\]
amd obtain that $\widetilde f(t)\leq f(t)\leq s(t)$ for any $t\in[t_1,t_2]$. Then, there exists $\overline t \in [t_1,t_2]$ such that $\widetilde f(\overline t)=s(t_2)$, which implies that
\[
\varepsilon+\delta_1=\varepsilon+2\delta_1-2\mu\left(\widetilde u-\frac{\Gamma u_B}{\lambda} \right) \delta_1(\overline t-t_1).
\]
From this, we infer that
\[
t_2-t_1\geq \overline t-t_1=\frac{1}{2\mu\left(\widetilde u-\frac{\Gamma u_B}{\lambda} \right)},
\]
which proves \eqref{claimt1t2}, and thus the claim in \eqref{Nbounded}.

We now define 
\[
v(r,t):=\overline{u}(t)v_1(r,t)+\underline{u}(t)v_2(r,t)
\]
and we show that $v(r, t)$ is a subsolution of~\eqref{mainfinal1}.

For this, we compute
\begin{equation}\label{xdftygvbnji}
\begin{split}
\partial^2_r v(r,t)&= \overline{u}(t)\partial^2_r v_1(r,t)+\underline{u}(t)\partial^2_r v_2(r,t)\\&=
\overline{u}(t)
\frac{\widetilde M^2\sinh(\widetilde M(r-\varepsilon))}{\sinh(\widetilde M(s(t)-\varepsilon))}
+\underline{u}(t)\frac{\widetilde M^2\sinh(\widetilde M(s(t)-r))}{\sinh(\widetilde M(s(t)-\varepsilon))}\\&=\widetilde M^2v(r,t).
\end{split}
\end{equation}
Moreover,
\begin{equation}\label{xdftygvbnji1}
\begin{split}
v_t(r,t)&=\dot{\overline{u}}(t)v_1(r,t)+\dot{\underline{u}}(t)v_2(r,t)+\overline{u}(t)\partial_t v_{1}(r,t)+\underline{u}(t)\partial_t v_{2}(r,t) \\
&=\overline{N} \,\overline{u}(t)v_1(r,t)+\underline{N}\underline{u}(t)v_2(r,t)+\overline{u}(t)\partial_t v_{1}(r,t)+\underline{u}(t)\partial_t v_{2}(r,t) \\
&\leq \overline{N} v(r,t)+\overline{u}(t)\partial_t v_{1}(r,t)+\underline{u}(t)\partial_t v_{2}(r,t).
\end{split}
\end{equation}

Now, we distinguish two cases, either $\dot s(t)\geq 0$ or $\dot s(t)<0$. In the first case, from~\eqref{vbuewi58221wsd4trg7uij} and \eqref{delta1esplicita} we see that
\begin{eqnarray*}
&&\partial_t v_{1}(r,t)=-\frac{\widetilde M\dot{s}(t)\sinh(\widetilde M(r-\varepsilon))\cosh(\widetilde M(s(t)-\varepsilon))}{\sinh^2(\widetilde M(s(t)-\varepsilon))}\\
&=&-\widetilde M\dot s(t)\sinh(\widetilde M(r-\varepsilon))\coth(\widetilde M(s(t)-\varepsilon))\mbox{csch}(\widetilde M(s(t)-\varepsilon)) \\
&\leq &-\widetilde M\dot s(t)\mbox{csch}(\widetilde M(s(t)-\varepsilon))\big(\widetilde M(r-\varepsilon) - C\widetilde M^3(r-\varepsilon)^3)\big)\left(\frac1{\widetilde M(s(t)-\varepsilon)}- C\widetilde M(s(t)-\varepsilon)\right)
\\
&=&\widetilde M\dot s(t)\mbox{csch}(\widetilde M(s(t)-\varepsilon))\\
&&\quad\times\Bigg(-\frac{r-\varepsilon}{s(t)-\varepsilon}+C\widetilde M^2\left((r-\varepsilon)(s(t)-\varepsilon)-\frac{(r-\varepsilon)^3}{s(t)-\varepsilon}\right) 
-C^2\widetilde M^4(r-\varepsilon)^3(s(t)-\varepsilon)\Bigg)
\end{eqnarray*}
and
\begin{eqnarray*}&&
\partial_t v_{2}(r,t)=
\frac{\widetilde M\dot{s}(t)\cosh(\widetilde M(s(t)-r))}{\sinh(\widetilde M(s(t)-\varepsilon))}-
\frac{\widetilde M\dot{s}(t)\sinh(\widetilde M(s(t)-r))\cosh(\widetilde M(s(t)-\varepsilon))}{\sinh^2(\widetilde M(s(t)-\varepsilon))}\\&&=
\widetilde M\dot{s}(t)\mbox{csch}(\widetilde M(s(t)-\varepsilon))\Big(\cosh(\widetilde M(s(t)-r))-\sinh(\widetilde M(s(t)-r))\coth(\widetilde M(s(t)-\varepsilon))\Big)
\\&&\leq \widetilde M\dot{s}(t)\mbox{csch}(\widetilde M(s(t)-\varepsilon))
\Bigg( 1+C\widetilde M^2(s(t)-r)^2 \\
&&\quad -\big(\widetilde M(s(t)-r)-C\widetilde M^3(s(t)-r)^3)\big)\left(\frac1{\widetilde M(s(t)-\varepsilon)} -C\widetilde M(s(t)-\varepsilon) \right)\Bigg)
\\&&=\widetilde M\dot{s}(t)\mbox{csch}(\widetilde M(s(t)-\varepsilon))\Bigg(\frac{r-\varepsilon}{s(t)-\varepsilon} \\
&&\quad+C\widetilde M^2\left((s(t)-r)^2+(s(t)-r)(s(t)-\varepsilon)-\frac{(s(t)-r)^3}{s(t)-\varepsilon}\right) 
-C^2\widetilde M^4(s(t)-r)^3(s(t)-\varepsilon)\Bigg)
.
\end{eqnarray*}

By using these observations into~\eqref{xdftygvbnji1}, and recalling the estimate in~\eqref{stimespunto} of Lemma~\ref{lemmastimespunto}, we infer that
\begin{equation*}
\begin{split}
&v_t(r,t)\leq \overline{N} v(r,t)\\
&\quad+\widetilde M\dot{s}(t)\mbox{csch}(\widetilde M(s(t)-\varepsilon))
\Bigg[-(\overline{u}(t)- \underline{u}(t))\frac{r-\varepsilon}{s(t)-\varepsilon}+C\widetilde M^2(2\overline u(t)+3\underline u(t))(s(t)-\varepsilon)^2
\Bigg]
\\
&\leq 
\overline{N} v(r,t)+ \widetilde M \mu \left(\widetilde u-\frac{\Gamma u_B}{\lambda} \right)(s(t)-\varepsilon)\mbox{csch}(\widetilde M(s(t)-\varepsilon)) \\
&\quad +C\widetilde M^3\mu(\overline u-\widetilde u)(2\overline u(t)+3\underline u(t))\mbox{csch}(\widetilde M(s(t)-\varepsilon))(s(t)-\varepsilon)^2
\\
&\le \overline{N} v(r,t)+\mu\left(\widetilde u-\frac{\Gamma u_B}{\lambda} \right)(\overline{u}(t)- \underline{u}(t))+C\widetilde M^2 \mu\left(\widetilde u-\frac{\Gamma u_B}{\lambda} \right)(\overline u(t)-\underline u(t))(s(t)-\varepsilon)^2 \\
&\quad +C\widetilde M^2 \mu (\overline u(t)-\widetilde u)(2\overline u(t)+3\underline u(t))(s(t)-\varepsilon)^2
+C^2 \widetilde M^4\mu(\overline u(t)-\widetilde u)(2\overline u(t)+3\underline u(t))(s(t)-\varepsilon)^4.
\end{split}
\end{equation*}
Thus, in light of~\eqref{xdftygvbnj}, we conclude that
\begin{equation*}
\begin{aligned}
v_t(r,t)&\leq \overline{N} v(r,t)+\mu\left(\widetilde u-\frac{\Gamma u_B}{\lambda} \right)(\overline{u}- \underline{u})+C\widetilde M^2 \mu\left(\widetilde u-\frac{\Gamma u_B}{\lambda} \right)(\overline u-\underline u)(s(t)-\varepsilon)^2 \\
&\quad +C\widetilde M^2 \mu (\overline u-\widetilde u)(2\overline u+3\underline u)(s(t)-\varepsilon)^2
+C^2 \widetilde M^4\mu(\overline u-\widetilde u)(2\overline u+3\underline u)(s(t)-\varepsilon)^4.
\end{aligned}
\end{equation*}

Combining this and~\eqref{xdftygvbnji} we get
\begin{equation}\label{xdftygvbnji5}
\begin{split}&
\kappa v_t(r,t)-\partial^2_r v(r,t)+\lambda v(r,t)-\Gamma u_B \\
&\le  \kappa\mu\left(\widetilde u-\frac{\Gamma u_B}{\lambda} \right)(\overline{u}- \underline{u})
-(\widetilde M^2-\lambda-\kappa\overline{N})v(r,t)
-\Gamma u_B \\
&\quad+\kappa C\widetilde M^2 \mu\left(\widetilde u-\frac{\Gamma u_B}{\lambda} \right)(\overline u-\underline u)(s(t)-\varepsilon)^2
+\kappa C\widetilde M^2 \mu (\overline u-\widetilde u)(2\overline u+3\underline u)(s(t)-\varepsilon)^2 \\
&\quad +\kappa C^2 \widetilde M^4\mu(\overline u-\widetilde u)(2\overline u+3\underline u)(s(t)-\varepsilon)^4.
\end{split}
\end{equation}

Now, from \eqref{vbuewi58221wsd4trg7uij} we observe that
\begin{eqnarray*}
v_1(r,t)&\geq &\frac{r-\varepsilon}{s(t)-\varepsilon}-2C\widetilde M^2(s(t)-\varepsilon)^2\\
{\mbox{and }}\quad
v_2(r,t)&\geq &\frac{s(t)-r}{s(t)-\varepsilon}-2C\widetilde M^2(s(t)-\varepsilon)^2.
\end{eqnarray*}
Therefore, for any~$t\in(t_1,t_2)$,
\begin{equation}\label{8765jhgfdhgfd765436543} v(r,t)\geq \underline{u}(t)(v_1+v_2)=\underline{u}(t)\big(1-4C\widetilde M^2(s(t)-\varepsilon)^2)\big)
\geq \frac{\Gamma u_B}{\lambda}\big(1-4C\widetilde M^2(s(t)-\varepsilon)^2)\big).
\end{equation}

Thus, using~\eqref{8765jhgfdhgfd765436543} into~\eqref{xdftygvbnji5} and recalling \eqref{defMtilde} and \eqref{Nbounded}
 we infer that
\[
\begin{aligned}
\kappa v_t(r,t)&-\partial^2_r v(r,t)+\lambda v(r,t)-\Gamma u_B \\
&\leq -\kappa\mu\left(\widetilde u-\frac{\Gamma u_B}{\lambda} \right)(\overline{u}- \underline{u})+\frac{4\Gamma u_B C\widetilde M^4}{\lambda}(s(t)-\varepsilon)^2 \\
&\quad +\kappa C\widetilde M^2 \mu\left(\widetilde u-\frac{\Gamma u_B}{\lambda} \right)(\overline u-\underline u)(s(t)-\varepsilon)^2
+\kappa C\widetilde M^2 \mu (\overline u-\widetilde u)(2\overline u+3\underline u)(s(t)-\varepsilon)^2 \\
&\quad +\kappa C^2 \widetilde M^4\mu(\overline u-\widetilde u)(2\overline u+3\underline u)(s(t)-\varepsilon)^4.
\\
&\leq -\kappa\mu\left(\widetilde u-\frac{\Gamma u_B}{\lambda} \right)(\overline{u}- \underline{u})+C\widetilde M^2(s(t)-\varepsilon)^2 
\Bigg[\kappa\mu\left(\widetilde u-\frac{\Gamma u_B}{\lambda} \right)(\overline u-\underline u) \\
&\quad +\kappa  \mu(\overline u-\widetilde u)(2\overline u+3\underline u)+\kappa C\mu(\overline u-\widetilde u)(2\overline u+3\underline u)
+\frac{4\Gamma u_B \widetilde M^2}{\lambda}
\Bigg].
\end{aligned}
\]
Taking~$\delta_1$ as in \eqref{delta1esplicita}, we have that
\[
\kappa v_t(r,t)-\partial^2_r v(r,t)+\lambda v(r,t)-\Gamma u_B<0
\]
for $\varepsilon<r<s(t)$ and $t\in(t_1,t_2)$ with $\dot s(t)\geq 0$.

Similar computations show that 
\[
\kappa v_t(r,t)-\partial^2_r v(r,t)+\lambda v(r,t)-\Gamma u_B<0
\]
for $\varepsilon<r<s(t)$ and $t\in(t_1,t_2)$ with $\dot s(t)< 0$, provided that $\delta_1$ is taken as in \eqref{delta1esplicita}.

We now show that $v(r, t)\leq u(r, t)$ on the parabolic boundary of $D$, namely on 
\[
\partial D\setminus \{(r,t)\in \R^2 \mbox{ such that } \varepsilon<r<s(t_2),\,\,t=t_2\}.
\]
First, by the convexity of~$\sinh(x)$ for~$x>0$ and recalling~\eqref{xcftyvbhjmk}, we have that, if~$\varepsilon<r<s(t_1)$,
\begin{eqnarray*}
v(r,t_1)&=&
\frac{\overline{u}(t_1)\sinh(M(r-\varepsilon))+\underline{u}(t_1)\sinh(M(s(t_1)-r))}{\sinh(M(s(t_1)-\varepsilon))}
\\&=&
\frac{\Gamma u_B\big(\sinh(M(r-\varepsilon))+\sinh(M(s(t_1)-r))\big)}{\lambda\sinh(M(s(t_1)-\varepsilon))}\\
&\leq& \frac{\Gamma u_B}{\lambda}\\&\leq& u(r,t_1). 
\end{eqnarray*}
Moreover, if~$t_1<t<t_2$,
\begin{eqnarray*}
v(\varepsilon,t)&=&\underline{u}(t)\leq \underline{u}=u(\varepsilon,t)
\\ {\mbox{and }}\quad
v(s(t),t)&=&\overline{u}(t)\leq \overline{u}=u(s(t),t).
\end{eqnarray*}

Thus, we are in the position of employing the comparison principle (see e.g.~\cite[Corollary~2.5]{MR1465184}) and obtain that~$u\geq v$ in~$D$ and, in particular,
\begin{equation*}
u(r,t_2)\geq v(r,t_2) \qquad \mbox{if }\varepsilon<r<s(t_2).
\end{equation*}
From this and~\eqref{freeboundary} (recall also footnote~\ref{ATTEFOO223er} on page~\pageref{ATTEFOO223er}), we infer that
$$ \dot s(t_2)\geq \frac{\mu}{s^2(t_2)}\int_\varepsilon^{s(t_2)}(v(r,t_2)-\widetilde u)r^2\,dr.$$
Recalling~\eqref{8765jhgfdhgfd765436543}, we thus find that
\begin{eqnarray*} \dot s(t_2)&\ge&\frac{\mu}{s^2(t_2)} \big(\underline{u}(t_2)-4\underline{u}(t_2)C\widetilde M^2(s(t_2)-\varepsilon)^2-\widetilde u\big)\int_\varepsilon^{s(t_2)}r^2\,dr
\\&=&\frac{\mu}{s^2(t_2)} \big(\underline u-\widetilde u-4\underline u C\widetilde M^2\delta_1^2\big)\frac{s^3(t_2)-\varepsilon^3}3
.\end{eqnarray*}
Therefore, if~$\delta_1$ is taken as in \eqref{delta1esplicita}, we conclude that~$\dot s(t_2)>0$, in contradiction with~\eqref{zxdrtfcfgty}.
\end{proof}

We are now in the position of completing the proof of Theorem~\ref{th:main1}.

\begin{proof}[Proof of the Theorem~\ref{th:main1}]
From Proposition~\ref{propstep2}, it is enough to take
\begin{equation}\label{deltastar}
\delta_*:=\delta_1
\end{equation}
to obtain the desired inequality in~\eqref{deltastarclaim}.
\end{proof}

We are now in the position to prove Proposition \ref{propCK>0}:
\begin{proof}[Proof of Proposition \ref{propCK>0}]
Let $K\subset (0, +\infty)$ be a compact set.
By \eqref{delta0esplicita}, \eqref{delta1esplicita} and \eqref{deltastar}, we infer that $\delta_*(\kappa)$ is continuous in $\kappa$. Moreover,  $\delta_*(\kappa)>0$.
Therefore, by the extreme value theorem,  we have that
\[
\delta_*(\kappa)\ge \min_{\kappa\in K} \delta_*(\kappa)>0.
\]
Thus, the desired result follows by taking $C_K:=\min_{\kappa\in K} \delta_*(\kappa)$.
\end{proof}

\section{Proof of Theorem~\ref{main3}}\label{èlpkmojbhugvyufctydrstetdryftgyhu}
To start with, we provide a useful integral equation for the radius $s(t)$, which will play a key role in the proofs of Theorems \ref{main3} and \ref{main2}.

\begin{lemma}\label{lpèkoijhugyvftcdrxsedctfvygbhjnkm}
Let $(u(r,t), s(t))$ be a solution of~\eqref{mainfinal1} and~\eqref{freeboundary}. 

For any $t>0$,  let
\begin{equation}\label{hdefn}
h(t):=\frac13 (s^3(t)-\varepsilon^3).
\end{equation}

Then,
\begin{equation}\label{dctfuvygibuhonijpmko,}
\begin{split}
\frac{\kappa}{\mu}\dot h(t)&=\int_0^t\left(s^2(\tau)\frac{\partial u}{\partial r}(s(\tau),\tau)-\varepsilon^2\frac{\partial u}{\partial r}(\varepsilon,\tau)-2\overline{u}s(\tau)+2\varepsilon \underline{u} +2\int_\varepsilon^{s(\tau)}u(r,\tau)dr\right) \,d\tau \\
&\qquad +\left(\kappa(\overline{u}-\widetilde u)-\frac{\lambda}{\mu}\right)h(t)-(\lambda\widetilde u -\Gamma u_B) \int_0^t h(\tau)\,d\tau-\overline\gamma,
\end{split}
\end{equation}
being
\begin{equation}\label{overgammadefn}
\overline\gamma:=\left(\kappa \overline{u}-\frac{\lambda}{\mu}\right)h(0)-\kappa\int_\varepsilon^{s(0)}r^2u_0(r)\,dr.
\end{equation}
\end{lemma}

\begin{proof} Multiplying the equation in~\eqref{mainfinal1}
by $r^2$ and integrating, we see that
\begin{equation}\label{xcfgyucsfghujicvh}
\begin{split}
&\kappa\int_0^t\left(\int_\varepsilon^{s(\tau)}r^2\frac{\partial u}{\partial t}(r,\tau)\, dr\right)\,d\tau-\int_0^t\left(\int_\varepsilon^{s(\tau)}r^2\frac{\partial^2u}{\partial r^2}(r,\tau)\, dr\right)\,d\tau+\lambda\int_0^t\left(\int_\varepsilon^{s(\tau)}r^2 u(r,\tau)\, dr\right)\,d\tau \\&\qquad=\Gamma u_B\int_0^t\left(\int_\varepsilon^{s(\tau)}r^2\, dr\right)\,d\tau=\frac{\Gamma u_B}3\int_0^t (s^3(\tau)-\varepsilon^3) \,d\tau.
\end{split}
\end{equation}
Moreover, from~\eqref{freeboundary} (recall also footnote~\ref{ATTEFOO223er} on page~\pageref{ATTEFOO223er}),
we infer that
\begin{equation}\label{xcfgyucsfghujicvh1}
\int_\varepsilon^{s(t)}r^2 u(r,t)\,dr=\frac{1}{\mu}s^2(t)\dot s(t)+\frac{\widetilde u}{3}(s^3(t)-\varepsilon^3).\end{equation}
Integrating this identity over~$\tau\in[0,t]$, we have that
\begin{equation}\label{xcfgyucsfghujicvh4}
\begin{aligned}
\int_0^t\left(\int_\varepsilon^{s(\tau)}r^2 u(r,\tau)\,dr\right)\,d\tau&=
\frac{1}{\mu}\int_0^t s^2(\tau)\dot s(\tau)\,d\tau+\frac{\widetilde u}{3}\int_0^t(s^3(\tau)-\varepsilon^3)\,d\tau \\
&=\frac{1}{3\mu}(s^3(t)-s^3(0))+\frac{\widetilde u}{3}\int_0^t(s^3(\tau)-\varepsilon^3)\,d\tau.
\end{aligned}
\end{equation}

Now, we observe that
\begin{eqnarray*}&&
\frac{d}{d\tau}\left(\int_\varepsilon^{s(\tau)}r^2u(r,\tau)\,dr\right)
=s^2(\tau) u(s(\tau),\tau)\dot s(\tau)+
\int_\varepsilon^{s(\tau)}r^2\frac{\partial u}{\partial t}(r,\tau)\,dr.
\end{eqnarray*}
On this account, by integrating over~$\tau\in[0,t]$,
we find that
\begin{eqnarray*}&&
\int_\varepsilon^{s(t)}r^2u(r,t)\,dr
-\int_\varepsilon^{s(0)}r^2u(r,0)\,dr
=\int_0^t
s^2(\tau) u(s(\tau),\tau)\dot s(\tau)\,d\tau+\int_0^t\left(
\int_\varepsilon^{s(\tau)}r^2\frac{\partial u}{\partial t}(r,\tau)\,dr\right)\,d\tau.
\end{eqnarray*}

Thus, recalling the boundary conditions in~\eqref{mainfinal1} and~\eqref{xcfgyucsfghujicvh1}, we get that
\begin{equation}\label{xcfgyucsfghujicvh2}
\begin{aligned}&
\int_0^t\left(\int_\varepsilon^{s(\tau)}r^2\frac{\partial u}{\partial t}(r,\tau)\, dr\right)\,d\tau=\int_\varepsilon^{s(t)}r^2u(r,t)\,dr-\int_\varepsilon^{s(0)}r^2u(r,0)\,dr-\overline{u}\int_0^ts^2(\tau)\dot s(\tau)\,d\tau \\
&\qquad=\frac{1}{\mu}s^2(t)\dot s(t)+\frac{\widetilde u}{3}(s^3(t)-\varepsilon^3)-\int_\varepsilon^{s(0)}r^2u(r,0)\,dr-\frac{\overline{u}}{3}(s^3(t)-s^3(0)) \\
&\qquad=\frac{1}{\mu}s^2(t)\dot s(t)-\frac{\overline{u}-\widetilde u}{3}s^3(t)-\frac{\widetilde u \varepsilon^3}{3}+\frac{\overline{u}}{3}s^3(0)-\int_\varepsilon^{s(0)}r^2u(r,0)\,dr.
\end{aligned}
\end{equation}

Furthermore,
$$
\frac{\partial}{\partial r}\big( 2r u(r,\tau)+r^2 u(r,\tau)\big)
= 2u(r,\tau)+r^2 \frac{\partial^2u}{\partial r^2}(r,\tau)
+4r\frac{\partial u}{\partial r}(r,\tau).$$
Hence, integrating for~$r\in[\varepsilon,s(\tau)]$,
and recalling the boundary conditions in~\eqref{mainfinal1}, we conclude that
\begin{equation}\label{vfdedfghPASKJdmedf}
\begin{split}
&2\overline u s(\tau)+s^2(\tau) \frac{\partial u}{\partial r}(s(\tau),\tau)
-2\underline u\varepsilon -\varepsilon^2 \frac{\partial u}{\partial r}(\varepsilon,\tau)\\
&\qquad= 2s(\tau) u(s(\tau),\tau)+s^2(\tau) \frac{\partial u}{\partial r}(s(\tau),\tau)
-2\varepsilon u(\varepsilon,\tau)-\varepsilon^2 \frac{\partial u}{\partial r}(\varepsilon,\tau)\\&\qquad= 2\int_\varepsilon^{s(\tau)}u(r,\tau)\,dr+\int_\varepsilon^{s(\tau)}r^2 \frac{\partial^2u}{\partial r^2}(r,\tau)\,dr
+4\int_\varepsilon^{s(\tau)}r\frac{\partial u}{\partial r}(r,\tau)\,dr.
\end{split}\end{equation}

Moreover,
$$ \frac{\partial }{\partial r}\big(ru(r,\tau)\big)=
u(r,\tau)+r\frac{\partial u}{\partial r}(r,\tau)$$
and consequently
\begin{eqnarray*} 
&& \overline u s(\tau)-\underline u\varepsilon=
s(\tau)u(s(\tau),\tau)-\varepsilon u(\varepsilon,\tau)
=\int_\varepsilon^{s(\tau)}\frac{\partial }{\partial r}\big(ru(r,\tau)\big)\,dr=\int_\varepsilon^{s(\tau)}\Big(
u(r,\tau)+r\frac{\partial u}{\partial r}(r,\tau)\Big)\,dr.
\end{eqnarray*}
We rewrite this in the form
$$ \int_\varepsilon^{s(\tau)}r\frac{\partial u}{\partial r}(r,\tau)\,dr=
\overline u s(\tau)-\underline u\varepsilon
-\int_\varepsilon^{s(\tau)}
u(r,\tau)\,dr.$$

Combining this identity and~\eqref{vfdedfghPASKJdmedf} we arrive at
\begin{equation*}
\begin{split}
&-2\overline u s(\tau)+s^2(\tau) \frac{\partial u}{\partial r}(s(\tau),\tau)
+2\underline u\varepsilon -\varepsilon^2 \frac{\partial u}{\partial r}(\varepsilon,\tau)+ 2\int_\varepsilon^{s(\tau)}u(r,\tau)\,dr=\int_\varepsilon^{s(\tau)}r^2 \frac{\partial^2u}{\partial r^2}(r,\tau)\,dr.
\end{split}\end{equation*}
We now integrate over~$\tau\in[0,t]$ and gather that
\begin{equation}\label{xcfgyucsfghujicvh3}
\begin{aligned}&
\int_0^t\left(\int_\varepsilon^{s(\tau)}r^2\frac{\partial^2u}{\partial r^2}(r,\tau)\, dr\right)\,d\tau \\
&=\int_0^t\left(s^2(\tau)\frac{\partial u}{\partial r}(s(\tau),\tau)-\varepsilon^2\frac{\partial u}{\partial r}(\varepsilon,\tau)-2\overline{u} s(\tau)+2\varepsilon\underline{u}+2\int_\varepsilon^{s(\tau)}u(r,\tau)dr\right) \,d\tau.
\end{aligned}
\end{equation}

Plugging \eqref{xcfgyucsfghujicvh4},
\eqref{xcfgyucsfghujicvh2} and~\eqref{xcfgyucsfghujicvh3} into~\eqref{xcfgyucsfghujicvh} we arrive 
at
\begin{equation*}
\begin{aligned}
\frac{\kappa}{\mu}s^2(t)\dot s(t)&=\int_0^t\left(s^2(\tau)\frac{\partial u}{\partial r}(s(\tau),\tau)-\varepsilon^2\frac{\partial u}{\partial r}(\varepsilon,\tau)-2\overline{u}s(\tau)+2\varepsilon \underline{u} +2\int_\varepsilon^{s(\tau)}u(r,\tau)dr\right) \,d\tau \\
&\qquad +\frac{1}{3}\left(\kappa(\overline{u}-\widetilde u)-\frac{\lambda}{\mu}\right)s^3(t)-\frac{\lambda\widetilde u-\Gamma u_B}{3} \int_0^t (s^3(\tau)-\varepsilon^3)\,d\tau\\&\qquad-
\frac{1}{3}\left(\kappa \overline{u}-\frac{\lambda}{\mu}\right)s^3(0)+\frac{\kappa \widetilde u \varepsilon^3}{3}+\kappa\int_\varepsilon^{s(0)}r^2u_0(r)\,dr.
\end{aligned}
\end{equation*}
This and~\eqref{hdefn} yield the desired result in~\eqref{dctfuvygibuhonijpmko,}.
\end{proof}

We now present some preliminary results which will be useful in the proof of Theorem~\ref{main3}.
First, we construct suitable barrier functions in order to control the radial derivative of $u(r,t)$ at both the boundary of the necrotic core and the moving free boundary.
Although these barriers are primarily employed in the proof of Theorem \ref{main3}, they are also of independent interest, as they yield biologically meaningful bounds on the cell flux at the tumor interfaces. In particular, we prove that cell gradients at both
the tumor free boundary and the necrotic core remain uniformly bounded,
as quantified, respectively,
in~\eqref{PKA:SNFVBAol2dPjmsd} and~\eqref{pqkd-32i5mb96n-8mnbJHvm689mbS8} here below. These uniform bounds align with biological expectations, since diffusion and growth processes occur at finite rates, preventing sharp variations in cell concentration.

To begin with, we establish a uniform upper bound for the radial derivative of the cell concentration at the free boundary $r=s(t)$.
Assuming that the cell concentration near the tumor surface is initially bounded away from zero and that the boundary velocity admits a uniform lower bound, we show that the outward radial derivative of $r=s(t)$ remains bounded from above for any $t>0$. 
More precisely, we have the following:

\begin{lemma}\label{barrieradasopra}
Assume~\eqref{relationutilde2} and \eqref{okpijhugyftdruyjmlò,.àxqBHJKNGVCUhbjnaekfml}.
Let $K\subset (0, +\infty)$ be compact and let $C_K>0$ be as in Proposition \ref{propCK>0}.
Let~$\kappa\in K$.
Let $(u(r,t), s(t))$ be a solution of~\eqref{mainfinal1} and~\eqref{freeboundary}, with~$u_0\in C^1([s(0)-C_K, s(0)])$.

Suppose that
\begin{equation}\label{2fhojabkcsbhigfueabjckvgòu}
\mbox{there exists $S>0$ such that $\dot{s}(t)\ge - S$, for any $t>0$}.
\end{equation}
Then,
\begin{equation}\label{PKA:SNFVBAol2dPjmsd}
\dfrac{\partial u}{\partial r}(s(t), t) \le \beta\left(\overline u -\frac{\Gamma u_B}{\lambda}\right),\quad\mbox{ for any } t>0,
\end{equation}
being
\begin{equation}\label{kpmnojbixhvubijnokmòlcvgjklmò,.}
\beta:=\max\left\{\dfrac{\kappa S +\sqrt{(\kappa S)^2 +4\lambda}}{2}, \frac{1}{C_K}\log\left(\dfrac{\overline u-\frac{\Gamma u_B}{\lambda}}{\underline u-\frac{\Gamma u_B}{\lambda}}\right), \dfrac{\|u'_0\|_{L^\infty([s(0)-C_K, s(0)])}}{\overline u-\frac{\Gamma u_B}{\lambda}-C_K\|u'_0\|_{L^\infty([s(0)-C_K, s(0)])}}\right\}. 
\end{equation}
\end{lemma}
\begin{proof}
In light of \eqref{relationutilde2} and \eqref{okpijhugyftdruyjmlò,.àxqBHJKNGVCUhbjnaekfml}, we have that~$\beta\in(0,+\infty)$.

We use a barrier argument by considering the function
\begin{equation}\label{phidefbnjhbiugs}
\phi(r, t):= \frac{\Gamma u_B}{\lambda}+\left(\overline u -\frac{\Gamma u_B}{\lambda}\right)e^{-\beta(s(t)-r)}.
\end{equation}

By \eqref{relationutilde2}, \eqref{2fhojabkcsbhigfueabjckvgòu} and \eqref{kpmnojbixhvubijnokmòlcvgjklmò,.},  we infer that
\[
\begin{split}
&\kappa \phi_t (r,t)-\partial_{rr} \phi(r,t)+\lambda \phi(r,t)- \Gamma u_B \\&\;=-\kappa\beta \dot{s}(t)\left(\overline u -\frac{\Gamma u_B}{\lambda}\right) e^{-\beta(s(t)-r)}
-\beta^2\left(\overline u -\frac{\Gamma u_B}{\lambda}\right) e^{-\beta(s(t)-r)}+\lambda\left(\overline u -\frac{\Gamma u_B}{\lambda}\right) e^{-\beta(s(t)-r)}\\
&\;\le\left(\overline u -\frac{\Gamma u_B}{\lambda}\right) e^{-\beta(s(t)-r)}\left(\kappa S \beta-\beta^2 +\lambda \right)\\
&\;\le 0.
\end{split}
\]
Moreover, by \eqref{phidefbnjhbiugs} and the boundary condition in \eqref{mainfinal1}, and recalling Proposition~\ref{propCK>0}, we get
\[
\begin{split}
&\phi(\varepsilon, t) =\frac{\Gamma u_B}{\lambda}+\left(\overline u -\frac{\Gamma u_B}{\lambda}\right)e^{-\beta(s(t)-\varepsilon)}\le \frac{\Gamma u_B}{\lambda}+\left(\overline u -\frac{\Gamma u_B}{\lambda}\right)e^{-\beta C_K}\le\underline u = u(\varepsilon,  t),\\
\mbox{ and } \quad &\phi(s(t), t)=\overline u = u(s(t), t).
\end{split}
\]
Furthermore, by Proposition \ref{propCK>0}, \eqref{kpmnojbixhvubijnokmòlcvgjklmò,.} and \eqref{kolmjopjnhbgyugfcdrtesxzwea}, it follows that
\begin{equation}\label{qòwlwakpmofjnibhcduxjzkmox}
\begin{split}
\phi(r, 0)&=\frac{\Gamma u_B}{\lambda}+\left(\overline u -\frac{\Gamma u_B}{\lambda}\right)e^{-\beta(s(0)-r)}\\
&\le \frac{\Gamma u_B}{\lambda}+\left(\overline u -\frac{\Gamma u_B}{\lambda}\right)e^{-\beta C_K}\le \underline u\le u_0(r), \quad\mbox{ for any } r\in [\varepsilon, s(0)-C_K].
\end{split}
\end{equation}

We now claim that
\begin{equation}\label{pèlkmn .kjopkjnbhuygfchtdfxseswaz<erdtfcgyuhij}
u_0(r)\ge \frac{\Gamma u_B}{\lambda}+\left(\overline u -\frac{\Gamma u_B}{\lambda}\right)e^{-\beta(s(0)-r)}\quad\mbox{ for any } r\in [s(0)- C_K, s(0)].
\end{equation}
Indeed, since $u_0\in C^1([s(0)-C_K, s(0)])$,  and $u_0(s(0))=\overline u$, we have that
\begin{equation}\label{èq+DOPWEFT67CDYU8KOL}
\overline u - u_0(r)=\int_r^{s(0)} u'_0(s) ds \le \|u'_0\|_{L^\infty([s(0)-C_K, s(0)])} (s(0)-r),  \quad r\in [s(0)-C_K, s(0)].
\end{equation}
Moreover, by \eqref{okpijhugyftdruyjmlò,.àxqBHJKNGVCUhbjnaekfml} and \eqref{kpmnojbixhvubijnokmòlcvgjklmò,.},  we infer that, for any $r\in [s(0)- C_K, s(0)]$, 
\[
\begin{split}
\|u'_0\|_{L^\infty([s(0)- C_K, s(0)])}&\le\dfrac{\beta\left(\overline u-\frac{\Gamma u_B}{\lambda}\right)}{1+\beta C_K }\le\dfrac{\beta\left(\overline u-\frac{\Gamma u_B}{\lambda}\right)}{1+\beta (s(0)-r)}\le\left(\overline u-\frac{\Gamma u_B}{\lambda}\right)\dfrac{1-e^{-\beta(s(0)-r)}}{s(0)-r},
\end{split}
\]
namely
\[
\|u'_0\|_{L^\infty([s(0)-C_K, s(0)])}(s(0)-r)\le\left(\overline u-\frac{\Gamma u_B}{\lambda}\right) (1-e^{-\beta(s(0)-r)}).
\]
Accordingly, by \eqref{èq+DOPWEFT67CDYU8KOL}, we infer that
\[
\overline u - u_0(r)\le \overline u-\frac{\Gamma u_B}{\lambda} -\left(\overline u-\frac{\Gamma u_B}{\lambda}\right)e^{-\beta(s(0)-r)}, 
\]
which proves the claim in \eqref{pèlkmn .kjopkjnbhuygfchtdfxseswaz<erdtfcgyuhij}.

Thus,  \eqref{qòwlwakpmofjnibhcduxjzkmox} and \eqref{pèlkmn .kjopkjnbhuygfchtdfxseswaz<erdtfcgyuhij} yield
\[
\phi(r, 0)\le u_0(r) \quad\mbox{ for any } r\in[\varepsilon, s(0)].
\]
As a result, by \cite[Corollary~2.5]{MR1465184},
we conclude that~$\phi(r, t)\le u(r,t)$ and accordingly
\[
\begin{split}
\frac{\partial u}{\partial r}(s(t),t)&=\lim_{h\searrow0}\frac{u(s(t),t)-u(s(t)-h,t)}{h}=\lim_{h\searrow0}\frac{\phi(s(t),t)-u(s(t)-h,t)}{h}\\
&\le\lim_{h\searrow0}\frac{\phi(s(t),t)-\phi(s(t)-h,t)}{h}=
\frac{\partial \phi}{\partial r}(s(t),t)= \beta\left(\overline u -\frac{\Gamma u_B}{\lambda}\right),
\end{split}
\]
as desired.
\end{proof}

In the next result we show that the radial derivative of the cell concentration evaluated at the interface between the necrotic core and the living tumor region remains uniformly bounded from below for all times.
More precisely,  we prove the following:

\begin{lemma}\label{barrieradasotto}
Assume~\eqref{relationutilde2}.
Let $(u(r,t), s(t))$ be a solution of~\eqref{mainfinal1} and~\eqref{freeboundary}. 

Then, 
\begin{equation}\label{pqkd-32i5mb96n-8mnbJHvm689mbS8}
\dfrac{\partial u}{\partial r}(\varepsilon, t) \ge -\sqrt{\lambda}\left(\underline u-\frac{\Gamma u_B}{\lambda} \right)\quad\mbox{ for any } t>0.
\end{equation}
\end{lemma}

\begin{proof}
We use a barrier argument by considering the function $\psi(r)$ defined as
\begin{equation}\label{òpèloiuytdreszawxrdcftvbhjnkm}
\psi(r):=\dfrac{\Gamma u_B}{\lambda}+\left(\underline u-\dfrac{\Gamma u_B}{\lambda}\right) e^{-\sqrt{\lambda}(r-\varepsilon)}\quad\mbox{ for any } r\in [\varepsilon, s(t)].
\end{equation}
Hence, by \eqref{relationutilde2}, we infer that
\[
\begin{split}
&\kappa\psi_t(r)-\psi''(r) +\lambda\psi(r)-\Gamma u_B \\
&\qquad=-\lambda\left(\underline u-\dfrac{\Gamma u_B}{\lambda}\right)e^{-\sqrt{\lambda}(r-\varepsilon)} +\Gamma u_B +\lambda\left(\underline u-\dfrac{\Gamma u_B}{\lambda}\right)e^{-\sqrt{\lambda}(r-\varepsilon)} -\Gamma u_B\\
&\qquad= 0.
\end{split}
\]
Furthermore, by \eqref{relationutilde2} and the boundary condition in \eqref{mainfinal1},  we get
\[
\begin{split}
&\psi(\varepsilon) =\underline u= u(\varepsilon, t),\\[7pt]
\mbox{ and }\quad &\psi(s(t))\le\underline u\le\overline u =u(s(t), t).
\end{split}
\]
In addition,  we notice that
\[
\psi(r)=\dfrac{\Gamma u_B}{\lambda}+\left(\underline u-\dfrac{\Gamma u_B}{\lambda}\right) e^{-\sqrt{\lambda}(r-\varepsilon)} \le\underline u\le u_0(r)\quad\mbox{ for any } r\in [\varepsilon, s(0)].
\]
As a result, by \cite[Corollary~2.5]{MR1465184},
we conclude that~$\psi(r)\le u(r,t)$ and accordingly
\[
\begin{split}
\frac{\partial u}{\partial r}(\varepsilon,t)&=\lim_{h\searrow0}\frac{u(\varepsilon +h, t)-u(\varepsilon, t)}{h}=\lim_{h\searrow0}\frac{u(\varepsilon +h, t)-\psi(\varepsilon)}{h}\\
&\ge\lim_{h\searrow0}\frac{\psi(\varepsilon +h)-\psi(\varepsilon)}{h}= -\sqrt{\lambda}\left(\underline u-\frac{\Gamma u_B}{\lambda} \right),
\end{split}
\]
as desired.
\end{proof}

We are now in the position to complete the proof of Theorem~\ref{main3}.

\begin{proof}[Proof of Theorem~\ref{main3}]
 Let 
\begin{equation}\label{kxjiquhftypijhugyftdrseawsdrtfygfcdxsexdrcftg}
K\subset \left(0, \dfrac{\lambda}{\mu(\overline u-\widetilde u)}\right)
\end{equation}
be compact and take $\kappa\in K$. Let $C_K>0$ be as in Proposition \ref{propCK>0}.
We start showing that if \eqref{either1} is not verified, then \eqref{s(t)bounded} holds true. Assume that \eqref{either1} is not satisfied, that is
\begin{equation}\label{okpuyftcdvklmòà.-}
\mbox{there exists $S>0$ such that $\dot{s}(t)\ge - S$ for any $t>0$}.
\end{equation}

Now, assume by contradiction that~\eqref{s(t)bounded} is not true. Then, for any $M>0$ there exists an instant of time $T>0$ such that $s(T)>M$.  In particular, for any arbitrarily large $M$, one can find $T=t_2$ such that
\begin{equation}\label{drfytguyhijokp,lM2}
h(t_2)=M^2.
\end{equation}
We consider the smallest value of $t_2>0$ that satisfies~\eqref{drfytguyhijokp,lM2} and we define $0<t_1<t_2$ in such a way that
\begin{alignat}{1}\label{gBCFhòfiòOFJEHò1}
&M< h(t) < M^2\quad\mbox{ if } t\in (t_1, t_2)\\ \label{ctfvygbhunjmk}
\mbox{ and }\quad& h(t_1)=M.
\end{alignat}
Let $\beta$ be as in \eqref{kpmnojbixhvubijnokmòlcvgjklmò,.}. We set
\begin{equation}\label{overlineMdefn}
\overline M:=\max\left\{\frac13\left(\dfrac{3(\beta +\sqrt{\lambda})\left(\overline u-\frac{\Gamma u_B}{\lambda}\right)}{\lambda\widetilde u-\Gamma u_B} +\varepsilon\right)^3 -\frac{\varepsilon^3}{3}, 1+\dfrac{3\kappa(\overline{u} -\widetilde u)}{\left|\kappa(\overline{u}-\widetilde u)-\frac{\lambda}{\mu}\right|}\right\}
\end{equation}
and we show that if $M>\overline M$, then one gets a contradiction. This will establish the validity of inequality~\eqref{s(t)bounded}.

We evaluate identity~\eqref{dctfuvygibuhonijpmko,} at $t=t_1$ and at~$t=t_2$,
obtaining that
\begin{equation}\label{vgbhnjmk,òl.à-}
\begin{split}
&\frac{\kappa}{\mu}\left(\dot h(t_2)-\dot h(t_1)\right)\\
&\quad=\int_{t_1}^{t_2}\left(s^2(\tau)\frac{\partial u}{\partial r}(s(\tau),\tau)-\varepsilon^2\frac{\partial u}{\partial r}(\varepsilon,\tau)-2\overline{u}s(\tau)+2\varepsilon \underline{u} +2\int_\varepsilon^{s(\tau)}u(r,\tau)dr\right) \,d\tau \\
&\qquad +\left(\kappa(\overline{u}-\widetilde u)-\frac{\lambda}{\mu}\right)(h(t_2)-h(t_1))-(\lambda\widetilde u -\Gamma u_B) \int_{t_1}^{t_2} h(\tau)\,d\tau.
\end{split}
\end{equation}
Now, by~\eqref{prima4}, we infer that
\[
\begin{split}
&\int_{t_1}^{t_2} \left(2\varepsilon \underline{u} +2\int_\varepsilon^{s(\tau)}u(r,\tau)dr\right) \,d\tau\le \int_{t_1}^{t_2} \left(2\varepsilon \underline{u} +2\overline{u} s(\tau) -2\varepsilon\overline{u}\right) \,d\tau\\
&\quad= 2\overline{u} \int_{t_1}^{t_2} s(\tau) \,d\tau -2\varepsilon(\overline u-\underline u)(t_2-t_1)\le 2\overline{u} \int_{t_1}^{t_2} s(\tau) \,d\tau.
\end{split}
\]
From this and~\eqref{vgbhnjmk,òl.à-}, we get
\begin{equation}\label{ctvtybuhinjmokp,l.ò}
\begin{split}
\frac{\kappa}{\mu}\left(\dot h(t_2)-\dot h(t_1)\right)&\le\int_{t_1}^{t_2}\left(s^2(\tau)\frac{\partial u}{\partial r}(s(\tau),\tau)-\varepsilon^2\frac{\partial u}{\partial r}(\varepsilon,\tau)\right) \,d\tau \\
&\quad +\left(\kappa(\overline{u}-\widetilde u)-\frac{\lambda}{\mu}\right)(h(t_2)-h(t_1))-(\lambda\widetilde u -\Gamma u_B) \int_{t_1}^{t_2} h(\tau)\,d\tau.
\end{split}
\end{equation}

We notice that, by~\eqref{hdefn} and~\eqref{gBCFhòfiòOFJEHò1}, it follows that
\begin{equation}\label{lokpijplèòlkoiopèòà+}
(3M +\varepsilon^3)^{\frac13}-\varepsilon<s(t)-\varepsilon<(3M^2 +\varepsilon^3)^{\frac13}-\varepsilon.
\end{equation}
From this, Lemma \ref{barrieradasopra}, Lemma \ref{barrieradasotto},  \eqref{hdefn} and recalling \eqref{kpmnojbixhvubijnokmòlcvgjklmò,.}, we deduce that
\begin{equation}\label{asrdtfyguhijmkolp1}
\begin{split}
&\int_{t_1}^{t_2} s^2(\tau)\frac{\partial u}{\partial r}(s(\tau),\tau) \,d\tau\le \beta\left(\overline u -\frac{\Gamma u_B}{\lambda}\right) \int_{t_1}^{t_2} s^2(\tau) \,d\tau\\
&\qquad\le\dfrac{\beta\left(\overline u -\frac{\Gamma u_B}{\lambda}\right)}{(3M +\varepsilon^3)^{\frac13}-\varepsilon} \int_{t_1}^{t_2} s^2(\tau) (s(\tau)-\varepsilon) \,d\tau\\
&\qquad\le\dfrac{\beta\left(\overline u -\frac{\Gamma u_B}{\lambda}\right)}{(3M +\varepsilon^3)^{\frac13}-\varepsilon} \int_{t_1}^{t_2} (s^3(\tau)-\varepsilon^3) \,d\tau = \dfrac{3\beta\left(\overline u -\frac{\Gamma u_B}{\lambda}\right)}{(3M +\varepsilon^3)^{\frac13}-\varepsilon}\int_{t_1}^{t_2} h(\tau) \,d\tau,
\end{split}
\end{equation}
and
\begin{equation}\label{asrdtfyguhijmkolp2}
-\int_{t_1}^{t_2}\varepsilon^2\frac{\partial u}{\partial r}(\varepsilon,\tau) \,d\tau\le \dfrac{\sqrt{\lambda}\left(\underline u -\frac{\Gamma u_B}{\lambda}\right)}{(3M +\varepsilon^3)^{\frac13}-\varepsilon} \int_{t_1}^{t_2}\varepsilon^2\,(s(\tau)-\varepsilon) \,d\tau =  \dfrac{3\sqrt{\lambda}\left(\underline u -\frac{\Gamma u_B}{\lambda}\right)}{(3M +\varepsilon^3)^{\frac13}-\varepsilon} \int_{t_1}^{t_2} h(\tau) \,d\tau.
\end{equation}
Thus, combining~\eqref{ctvtybuhinjmokp,l.ò}, \eqref{asrdtfyguhijmkolp1}, \eqref{asrdtfyguhijmkolp2} and \eqref{relationutilde2}, we obtain that
\begin{equation}\label{sdrftgyhunijmko,l}
\begin{split}
&\frac{\kappa}{\mu}\left(\dot h(t_2)-\dot h(t_1)\right)\le 3\left(\frac{(\beta +\sqrt{\lambda})\left(\overline u-\frac{\Gamma u_B}{\lambda}\right)}{(3M +\varepsilon^3)^{\frac13}-\varepsilon}-\frac{\lambda\widetilde u -\Gamma u_B}{3}\right)\int_{t_1}^{t_2} h(\tau) \,d\tau\\
&\quad+\left(\kappa(\overline{u}-\widetilde u)-\frac{\lambda}{\mu}\right)(h(t_2)-h(t_1)).
\end{split}
\end{equation}

Since we take $M>\overline M$, being $\overline M$ as in definition~\eqref{overlineMdefn}, we find that
\[
\frac{(\beta +\sqrt{\lambda})\left(\overline u-\frac{\Gamma u_B}{\lambda}\right)}{(3M +\varepsilon^3)^{\frac13}-\varepsilon}-\frac{\lambda\widetilde u -\Gamma u_B}{3}<0.
\]
Accordingly, by~\eqref{sdrftgyhunijmko,l} and~\eqref{hdefn}, we infer that
\begin{equation}\label{ytguhinjomkpl,ò.CDHBNJM}
\frac{\kappa}{\mu}\left(\dot h(t_2)-\dot h(t_1)\right)<\left(\kappa(\overline{u}-\widetilde u)-\frac{\lambda}{\mu}\right)(h(t_2)-h(t_1)).
\end{equation}
On the one hand,  given that $\kappa\in K$,  by \eqref{kxjiquhftypijhugyftdrseawsdrtfygfcdxsexdrcftg} we infer that
$\kappa<\frac{\lambda}{\mu(\overline u-\widetilde u)}$. In addition, since \eqref{drfytguyhijokp,lM2} and~\eqref{ctfvygbhunjmk} hold true, we have
\begin{equation}\label{lkoftnjbhgvcfdx}
\left(\kappa(\overline{u}-\widetilde u)-\frac{\lambda}{\mu}\right)(h(t_2)-h(t_1)) = -\left|\kappa(\overline{u}-\widetilde u)-\frac{\lambda}{\mu}\right| (M^2 - M).
\end{equation}
On the other hand, in virtue of~\eqref{drfytguyhijokp,lM2}
and~\eqref{gBCFhòfiòOFJEHò1}, we know that~$\dot h(t_2)\ge0$.
On this account,
by~\eqref{stimespunto}, \eqref{hdefn} and~\eqref{ctfvygbhunjmk}, it follows that, for any $t>0$,
\begin{equation}\label{gtuhijomkp,l}
\begin{split}
\frac{\kappa}{\mu}\left(\dot h(t_2)-\dot h(t_1)\right)&\ge -\frac{\kappa}{\mu} \dot h(t_1) = -\frac{\kappa}{\mu} s^2(t_1)\,\dot{s}(t_1)\ge -\kappa(\overline{u} -\widetilde u) s^2(t_1) (s(t_1)-\varepsilon)\\
&\ge -\kappa(\overline{u} -\widetilde u)(s^3(t_1)-\varepsilon^3)= -3\kappa(\overline{u} -\widetilde u) h(t_1) =-3\kappa(\overline{u} -\widetilde u) M.
\end{split}
\end{equation}
Thus, from~\eqref{ytguhinjomkpl,ò.CDHBNJM}, \eqref{lkoftnjbhgvcfdx} and~\eqref{gtuhijomkp,l}, we get that
\[
-3\kappa(\overline{u} -\widetilde u) < -\left|\kappa(\overline{u}-\widetilde u)-\frac{\lambda}{\mu}\right| (M - 1), 
\]
namely
\[
M<1+\dfrac{3\kappa(\overline{u} -\widetilde u)}{\left|\kappa(\overline{u}-\widetilde u)-\frac{\lambda}{\mu}\right|}, 
\]
which is in contradiction with the definition in~\eqref{overlineMdefn}. This shows that \eqref{s(t)bounded} holds true.

It remains to prove that if \eqref{either1} holds true, then \eqref{s(t)bounded} is not satisfied.
 Let \eqref{either1} be in force and assume by contradiction that \eqref{s(t)bounded} holds true. Thus, by \eqref{stimespunto}, we infer that for any $t>0$,
\[
\dot s(t)\ge-\mu\left(\widetilde u-\frac{\Gamma u_B}{\lambda}\right)(s(t)-\varepsilon)\ge -\mu\left(\widetilde u-\frac{\Gamma u_B}{\lambda}\right)\widetilde u(C-\varepsilon)=:-S,
\]
which is in contradiction with the assumption \eqref{either1}. This shows that if \eqref{either1} is satisfied, then \eqref{s(t)bounded} cannot hold, as desired.
\end{proof}

\section{Proof of Proposition~\ref{propeta} and Theorem~\ref{main2}}\label{kpoijhuygtfrdetfyguhijk}
In this section we provide the proof of Theorem~\ref{main2}
by showing that if the
ratio~$\kappa$ between the spreading rate of tumor cells and the growth rate of the tumor mass is large enough, then the radius $s(t)$
of the tumor mass diverges exponentially as $t$ goes to infinity.

To begin with, we establish Proposition \ref{propeta}.

\begin{proof}[Proof of Proposition \ref{propeta}]
We set
\begin{equation}\label{oiuygtfdoiuy1}
\psi_\alpha(r):=\frac{\Gamma u_B}{\lambda}+\left(\underline u- \frac{\Gamma u_B}{\lambda}\right)e^{-\sqrt{\lambda+\alpha}(r-\varepsilon)}
\end{equation}
for some $\alpha>0$ (in what follows, we will take~$\alpha$ conveniently large, possibly depending on $u_0$).

We claim that, if $\alpha$ is large enough, then 
\begin{equation}\label{xcfghuicvghj}
u(r,t)\geq \psi_\alpha(r)\quad \mbox{for any $r\in [\varepsilon, s(t)]$ and $t>0$}.
\end{equation}

To check this, we start observing that, for any $r\in [\varepsilon, s(t)]$ and $t>0$,
\begin{equation}\label{MAPLNSOJHNED:001}
\kappa (\psi_\alpha)_t(r)-\psi_\alpha''(r)+\lambda\psi_\alpha(r)-\Gamma u_B =-\alpha\left(\underline u- \frac{\Gamma u_B}{\lambda}\right)e^{-\sqrt{\lambda+\alpha}(r-\varepsilon)}<0.
\end{equation}

Moreover, 
\begin{equation}\label{MAPLNSOJHNED:002}
\psi_\alpha(\varepsilon)=\underline u=u(\varepsilon,t)
\end{equation}
and
\begin{equation}\label{MAPLNSOJHNED:003}
\psi_\alpha(s(t))=\frac{\Gamma u_B}{\lambda}+\left(\underline u- \frac{\Gamma u_B}{\lambda}\right)e^{-\sqrt{\lambda+\alpha}(s(t)-\varepsilon)}\leq \underline u \leq \overline u =u(s(t),t).
\end{equation}

Now, we prove that there exists some $\alpha_0\ge1$ such that, if $\alpha\geq \alpha_0$, then 
\begin{equation}\label{xcghjcvghuvb2poj4e.P.kf0w}
{\mbox{$\psi_\alpha(r)\leq u_0(r)$ for any $r\in [\varepsilon, s(0)]$.}}\end{equation}
To check this, we observe that,
by \eqref{oiuygtfdoiuy2} and since $u_0\in C^1([\varepsilon,s(0)])$, we can define
\[
m:=\min_{r\in [\varepsilon, s(0)]}u_0(r)>\frac{\Gamma u_B}{\lambda}
\quad \mbox{and} \quad m':=\min_{r\in [\varepsilon, s(0)]}u_0'(r).
\]

Now, if $m\geq \underline u$, then we have
\[
\psi_\alpha(r)\leq \underline u \leq u_0(r) \quad \mbox{for any }r\in [\varepsilon, s(0)],
\]
as desired. Thus, we can focus on the case $m<\underline u$.
We observe that in this case we also have $m'<0$. 

Hence, for all~$r\in[\varepsilon,s(0)]$,
\begin{equation}\label{vf489s-edj0rflK} \underline u-u_0(r)=u_0(\varepsilon)-u_0(r)\le \max_{r\in[\varepsilon,s(0)]} |u_0'(r)|\,(r-\varepsilon)=-m'(r-\varepsilon).\end{equation}

We now define $\widetilde \psi(r):=\underline u+m'(r-\varepsilon)$ and deduce from~\eqref{vf489s-edj0rflK} that
\begin{equation}\label{xcghjcvghuvb1}
\widetilde \psi(r)\leq u_0(r)\quad \mbox{for any }r\in [\varepsilon, s(0)].
\end{equation}

We also set $\widetilde r:= \varepsilon+\frac{m-\underline u}{m'}$ and remark that~$\widetilde \psi(\widetilde r)=m$. 

We claim that
\begin{equation}\label{0pkjd0-32u5j1t93kfgbad75}
\widetilde r\in(\varepsilon,s(0)).\end{equation}
To check this, we point out that, on the one hand, \begin{equation*}
\frac{m-\underline u}{m'}>0,
\end{equation*}since both the numerator and the denominator are negative, and this gives that~$\widetilde r>\varepsilon$.

On the other hand, if~$r_\star\in[\varepsilon,s(0)]$ is such that~$m=u_0(r_\star)$, we have that~$u_0(r_\star)=m<\underline u<\overline u(s(0))$ and therefore~$r_\star\ne s(0)$. As a result, in view of~\eqref{vf489s-edj0rflK},
$$ \underline u-m=\underline u-u_0(r_\star)\le -m'(r-\varepsilon)<-m'(s(0)-\varepsilon).$$
This observation proves that~$\widetilde r<s(0)$, thus establishing~\eqref{0pkjd0-32u5j1t93kfgbad75}.

We now observe that if 
\[
\alpha_\star:=-\lambda+\left[\frac{1}{\widetilde r-\varepsilon}\log\left(\frac{\lambda \underline u-\Gamma u_B}{\lambda m-\Gamma u_B}\right)\right]^2,
\]
then $\psi_{\alpha_\star}(\widetilde r)=m$. 

Consequently, since~$\psi_{\alpha_\star}$ is convex and~$\widetilde\psi$ is linear, with~$\psi_{\alpha_\star}(\varepsilon)=\widetilde\psi(\varepsilon)$ and~$\psi_{\alpha_\star}(\widetilde r)=\widetilde\psi(\widetilde r)$, we conclude that
\begin{equation}\label{xcghjcvghuvb2}
\psi_{\alpha_\star}(r)\leq \widetilde \psi (r) \quad \mbox{for any }r\in [\varepsilon, \widetilde r].
\end{equation}

We also observe that, for all~$r\ge0$ and~$\alpha_1\ge\alpha_2\ge0$, we have that\begin{equation}\label{ojs0u0325jymj.0}
\psi_{\alpha_1}(r)\leq\psi_{\alpha_2}(r).\end{equation}  
Thus, setting
\[
\alpha_0:=\max\left\{1, \alpha_\star \right\},
\]
we infer from~\eqref{xcghjcvghuvb2} that, for any $\alpha\ge \alpha_0$,
$$ \psi_{\alpha}(r)\leq \widetilde \psi (r) \quad \mbox{for any }r\in [\varepsilon, \widetilde r].$$

{F}rom this and~\eqref{xcghjcvghuvb1}, we conclude that
\begin{equation}\label{xcghjcvghuvb2poj4e.P} \psi_{\alpha}(r)\leq \widetilde u_0 (r) \quad \mbox{for any }r\in [\varepsilon, \widetilde r].\end{equation}

Additionally, using that~$\psi_\alpha$ is decreasing in~$r$ and recalling~\eqref{ojs0u0325jymj.0},
for all~$\alpha\ge\alpha_0$ we have that
$$\psi_\alpha(r)\le\psi_{\alpha_\star}(r)\le\psi_{\alpha_\star}(\widetilde r)
= m \leq u_0(r) \quad \mbox{for any }r\in ( \widetilde r, s(0)].$$
This, in tandem with~\eqref{xcghjcvghuvb2poj4e.P}, completes the proof of~\eqref{xcghjcvghuvb2poj4e.P.kf0w}.

Then, by~\eqref{MAPLNSOJHNED:001}, \eqref{MAPLNSOJHNED:002}, \eqref{MAPLNSOJHNED:003},
\eqref{xcghjcvghuvb2poj4e.P.kf0w} and \cite[Corollary~2.5]{MR1465184}, we have the desired claim in \eqref{xcfghuicvghj}.

Now, for any $x>0$ we consider the function 
\[
f(x):=\int_\varepsilon^{\varepsilon+x}(\psi_\alpha(r)-\widetilde u)r^2dr.
\]
Since $f'(x)=(\psi_\alpha(\varepsilon+x)-\widetilde u)(\varepsilon+x)^2$, we have that
\[
f'(x)>0 \quad \mbox{ if and only if } \quad x\in \left(0, \frac{1}{\sqrt{\lambda+\alpha}}\log\left(\frac{\lambda \underline u-\Gamma u_B}{\lambda \widetilde u-\Gamma u_B}\right)\right).
\]
Moreover, we observe that
\[
\lim_{x\to 0^+}f(x)=0.
\] 
Gathering these pieces of information, we infer that there exists a unique $\widetilde \eta>0$ such that $f(\widetilde \eta)=0$ and $f(x)>0$ for any $x\in (0,\widetilde \eta)$.

We now let $\eta \in (0,\min\{\widetilde \eta,s(0)-\varepsilon\})$, and we assume by contradiction that there exists $t_0>0$ such that $s(t_0)=\varepsilon+\eta$ and $\dot s(t_0)\leq 0$.
Then, from \eqref{freeboundary} and \eqref{xcfghuicvghj} we have 
\[
0\ge s^2(t_0)\dot s(t_0)=\mu \int_\varepsilon^{s(t_0)}(u(r,t_0)-\widetilde u)r^2dr\ge \mu\int_\varepsilon^{\varepsilon+\eta}(\psi_\alpha(r)-\widetilde u)r^2dr= \mu f(\eta)>0,
\]
which is a contradiction.
\end{proof}

Our next objective is to bound the derivative of the solution at the free boundary. Specifically,
we use a barrier argument to establish a uniform lower bound for the radial derivative of the solution $u(r, t)$ evaluated at the free boundary $r=s(t)$. More precisely we prove the follwing:

\begin{proposition}\label{lemmabarriera1}
Assume~\eqref{relationutilde2}, \eqref{oiuygtfdoiuy2}, and~\eqref{lpko,mijnhugyftcdrxctfvl,ò}.
Let $(u(r,t), s(t))$ be a solution of~\eqref{mainfinal1} and~\eqref{freeboundary}.  

Then, there exists $C_0>0$ such that
\begin{equation}\label{derivuniflim}
\frac{\partial u}{\partial r}(\varepsilon,t)\le C_0 \quad\mbox{ for any } t>0.
\end{equation}
Explicitly, one has
\begin{equation}\label{09iu8ygvhbdjwenAFLò,RSMKNJHGUY}
C_0:= \dfrac{\sqrt\lambda(\overline u-\underline u)}{1-e^{-\sqrt\lambda \eta}},
\end{equation}
where $\eta>0$ is as in Proposition \ref{propeta}.
\end{proposition}

\begin{proof} 
We use a barrier argument by considering the function $\phi(r)$ defined as
\[
\phi(r):=\underline u +\dfrac{\overline u-\underline u}{1-e^{-\sqrt\lambda \eta}}\left(1-e^{-\sqrt\lambda(r-\varepsilon)}\right) \quad \mbox{for any } r\in [\varepsilon, s(t)].
\]
By \eqref{relationutilde2},  we infer that
\[
\begin{split}
&\kappa\phi_t(r)-\phi''(r) +\lambda\phi(r)-\Gamma u_B\\&\qquad = \dfrac{\lambda(\overline u-\underline u)}{1-e^{-\sqrt\lambda \eta}} e^{-\sqrt\lambda (r-\varepsilon)}+\lambda\underline u +\dfrac{\lambda(\overline u-\underline u)}{1-e^{-\sqrt\lambda \eta}}\left(1-e^{-\sqrt\lambda(r-\varepsilon)}\right)-\Gamma u_B\\
&\qquad=\dfrac{\lambda(\overline u-\underline u)}{1-e^{-\sqrt\lambda \eta}} +\lambda\underline u-\Gamma u_B\ge 0.
\end{split}
\]
Furthermore, by \eqref{relationutilde2}, the boundary condition in \eqref{mainfinal1},  and Proposition \ref{propeta},
\[
\begin{split}
&\phi(\varepsilon) =\underline u= u(\varepsilon, t),\\[7pt]
\mbox{ and }\quad &\phi(s(t))\ge\phi(\varepsilon+\eta)=\overline u =u(s(t), t).
\end{split}
\]
In addition, recalling \eqref{lpko,mijnhugyftcdrxctfvl,ò}, we notice that
\[
\phi(r)\ge \underline u +\dfrac{\overline u-\underline u}{1-e^{-\sqrt\lambda (s(0)-\varepsilon)}}\left(1-e^{-\sqrt\lambda(r-\varepsilon)}\right)\ge u_0(r)\quad\mbox{ for any } r\in [\varepsilon, s(0)].
\]
As a result, by \cite[Corollary~2.5]{MR1465184},
we conclude that~$\phi(r)\ge u(r,t)$ and accordingly
\begin{eqnarray*}&&
\frac{\partial u}{\partial r}(\varepsilon,t)=\lim_{h\searrow0}\frac{u(\varepsilon+h,t)-u(\varepsilon,t)}{h}=\lim_{h\searrow0}\frac{u(\varepsilon+h,t)-\phi(\varepsilon)}{h}\\&&\qquad\le\lim_{h\searrow0}\frac{\phi(\varepsilon+h)-\phi(\varepsilon)}{h} 
= \dfrac{\sqrt\lambda(\overline u-\underline u)}{1-e^{-\sqrt\lambda \eta}},
\end{eqnarray*}
as desired.
\end{proof}

Moreover, we note the following:
\begin{lemma}\label{kmcaHVIBEWòonvcàwkoplèò}
Assume that \eqref{relationutilde2} holds true. Let $\eta>0$ be as in Proposition \ref{propeta}, $C_0$ be 
as in~\eqref{09iu8ygvhbdjwenAFLò,RSMKNJHGUY} and~$\kappa_\star$ be as in~\eqref{kopfweoiojiwavnjbhin}.

Then, for any $\kappa>\kappa_\star$,
\[
\left(\kappa(\overline{u}-\widetilde u)-\frac{\lambda}{\mu}\right)^2 >\frac{4\kappa}{\mu}\left(\lambda\widetilde u +\dfrac{C_0}{\eta}+\frac{1}{\varepsilon^2}\left(2\overline u +\dfrac{2\varepsilon(\overline u -\underline u)}{\eta} \right)\right).
\]
\end{lemma}

\begin{proof}
Let $\kappa>\kappa_*$.
In light of \eqref{kopfweoiojiwavnjbhin}, we have 
\[
\kappa\ge\dfrac{2\lambda}{\mu(\overline u-\widetilde u)}.
\]
Accordingly, we have that
\[
\left(\kappa(\overline{u}-\widetilde u)-\frac{\lambda}{\mu}\right)^2\ge\dfrac{\kappa^2(\overline{u}-\widetilde u)^2}{4}.
\]
Thus, it suffices to show that
\[
\dfrac{\kappa(\overline{u}-\widetilde u)^2}{4}>\frac{4}{\mu}\left(\lambda\widetilde u +\dfrac{C_0}{\eta}+\frac{1}{\varepsilon^2}\left(2\overline u +\dfrac{2\varepsilon(\overline u -\underline u)}{\eta} \right)\right),
\]
which is ensured by \eqref{kopfweoiojiwavnjbhin}.
\end{proof}

We are now in the position to complete the proof of Theorem~\ref{main2}.

\begin{proof}[Proof of Theorem~\ref{main2}]
For any $t>0$, let $h(t)$ be as in~\eqref{hdefn}. We notice that $h(t)>0$ for any $t\ge 0$, owing to~\eqref{stimespunto2}.

We claim that
\begin{equation}\label{deriveps}
\frac{\partial u}{\partial r}(s(t),t)\ge 0\quad\mbox{ for any } t>0.
\end{equation}
Indeed,  since $u(s(t), t)= \overline{u}$, recalling~\eqref{prima4} we deduce that
\[
\frac{\partial u}{\partial r}(s(t),t) = \lim_{\zeta\searrow0}\dfrac{u(s(t), t)-u(s(t)-\zeta, t)}{\zeta}\ge\lim_{\zeta\searrow0}\dfrac{\overline{u} -\overline{u}}{\zeta} =0
\]
for any $t>0$. This proves~\eqref{deriveps}.

Now, recalling~\eqref{dctfuvygibuhonijpmko,}, \eqref{relationutilde2}, \eqref{deriveps} and Propositions~\ref{propeta}
and~\ref{lemmabarriera1},
\begin{equation}\label{OjslqwdmedfX3qwoeirf30eprvf}
\begin{split}
\frac{\kappa}{\mu}\dot h(t)&\ge -\frac{C_0\,\varepsilon^2}{\eta}\int_0^t (s(t)-\varepsilon)\,d\tau -\left(2\overline u +\dfrac{2\varepsilon(\overline u -\underline u)}{\eta} \right) \int_0^t (s(\tau)-\varepsilon) \,d\tau\\
&\qquad+\left(\kappa(\overline{u}-\widetilde u)-\frac{\lambda}{\mu}\right)h(t)-\lambda\widetilde u \int_0^t h(\tau)\,d\tau-\overline\gamma.
\end{split}
\end{equation}

Moreover,  we have
\[
\begin{split}
\int_0^t (s(\tau)-\varepsilon) \,d\tau\le\int_0^t \dfrac{(s^3(\tau)-\varepsilon^3)}{3\varepsilon^2} \,d\tau = \dfrac{1}{\varepsilon^2}\int_0^t h(\tau) d\tau.
\end{split}
\]
This and~\eqref{OjslqwdmedfX3qwoeirf30eprvf} yield that
\begin{equation}\label{vycgiuobhIJNPCKL}
\begin{split}
\frac{\kappa}{\mu}\dot h(t)&\ge -\frac{C_0}{\eta}\int_0^t h(\tau) \,d\tau -\frac{1}{\varepsilon^2}\left(2\overline u +\dfrac{2\varepsilon(\overline u -\underline u)}{\eta} \right) \int_0^t h(\tau) \,d\tau\\
&\qquad+\left(\kappa(\overline{u}-\widetilde u)-\frac{\lambda}{\mu}\right)h(t)-\lambda\widetilde u \int_0^t h(\tau)\,d\tau-\overline\gamma\\&=
\left(\kappa(\overline{u}-\widetilde u)-\frac{\lambda}{\mu}\right)h(t) -\left(\lambda\widetilde u +\dfrac{C_0}{\eta}+\frac{1}{\varepsilon^2}\left(2\overline u +\dfrac{2\varepsilon(\overline u -\underline u)}{\eta} \right)\right)\int_0^t h(\tau)\,d\tau -\overline\gamma.
\end{split}
\end{equation}

We now claim that, for any $ t\ge 0$,
\begin{equation}\label{claim2uvlghbiks}
\left(\kappa(\overline{u}-\widetilde u)-\frac{\lambda}{\mu}\right)h(t)>2\left(\lambda\widetilde u +\dfrac{C_0}{\eta}+\frac{1}{\varepsilon^2}\left(2\overline u +\dfrac{2\varepsilon(\overline u -\underline u)}{\eta} \right)\right)\int_0^t h(\tau)\,d\tau +2\overline\gamma .
\end{equation}
To check this, we first establish~\eqref{claim2uvlghbiks} when~$t=0$.
For this, we use the definitions in~\eqref{hdefn} and~\eqref{overgammadefn} to find that
\begin{eqnarray*}
&&\left(\kappa(\overline{u}-\widetilde u)-\frac{\lambda}{\mu}\right)h(0)-2\overline\gamma =\left(\kappa(\overline{u}-\widetilde u)-\frac{\lambda}{\mu}\right)h(0)-2\left(\kappa \overline{u}-\frac{\lambda}{\mu}\right)h(0)+2\kappa\int_\varepsilon^{s(0)}r^2u_0(r)\,dr\\
&&\qquad=\left(\frac{\lambda}{\mu}-\kappa(\overline{u}+\widetilde u)
\right)h(0)+2\kappa\int_\varepsilon^{s(0)}r^2u_0(r)\,dr\\&&\qquad=\left(\frac{\lambda}{\mu}-
\kappa(\overline{u}+\widetilde u)
\right)\frac{s^3(0)-\varepsilon^3}3+2\kappa\int_\varepsilon^{s(0)}r^2u_0(r)\,dr\\&&\qquad=\left(\frac{\lambda}{\mu}-
\kappa(\overline{u}+\widetilde u)
\right)\int_\varepsilon^{s(0)}r^2\,dr+2\kappa\int_\varepsilon^{s(0)}r^2u_0(r)\,dr\\&&\qquad=\int_\varepsilon^{s(0)}r^2
\left( \kappa(2 u_0(r)-\overline{u} -\widetilde u)+\frac\lambda\mu\right)
\,dr\\
&&\qquad>0,
\end{eqnarray*}
thanks to assumption~\eqref{ctfyivgoubhjklmò,2}.
This shows that~\eqref{claim2uvlghbiks} is verified when~$t=0$. 

Now, let $t>0$ and assume by contradiction that~\eqref{claim2uvlghbiks} does not hold. Then, there exists~$t_0>0$ such that
\begin{equation}\label{ucgbodjkwòlm}\begin{split}
& \left(\kappa(\overline{u}-\widetilde u)-\frac{\lambda}{\mu}\right)h(t)\\
&\qquad>2\left(\lambda\widetilde u +\dfrac{C_0}{\eta}+\frac{1}{\varepsilon^2}\left(2\overline u +\dfrac{2\varepsilon(\overline u -\underline u)}{\eta} \right)\right)\int_0^t h(\tau)\,d\tau +2\overline\gamma\quad\mbox{ for any } t\in (0, t_0),\\
\mbox{ and } \quad
&\left(\kappa(\overline{u}-\widetilde u)-\frac{\lambda}{\mu}\right)h(t_0)=2\left(\lambda\widetilde u +\dfrac{C_0}{\eta}+\frac{1}{\varepsilon^2}\left(2\overline u +\dfrac{2\varepsilon(\overline u -\underline u)}{\eta} \right)\right)\int_0^{t_0} h(\tau)\,d\tau +2\overline\gamma .
\end{split}\end{equation}
As a consequence, we have
\begin{equation}\label{cfyivgygobhiuajnvòk}
\left(\kappa(\overline{u}-\widetilde u)-\frac{\lambda}{\mu}\right)\dot h(t_0)\le 2\left(\lambda\widetilde u +\dfrac{C_0}{\eta}+\frac{1}{\varepsilon^2}\left(2\overline u +\dfrac{2\varepsilon(\overline u -\underline u)}{\eta} \right)\right) h(t_0).
\end{equation}

We notice that,  as a byproduct of~\eqref{kopfweoiojiwavnjbhin}, 
\begin{equation}\label{bdwipKMEJFà-}
\kappa(\overline{u}-\widetilde u)-\frac{\lambda}{\mu}>0.
\end{equation}
Accordingly, by using~\eqref{vycgiuobhIJNPCKL}, \eqref{ucgbodjkwòlm},  and Lemma \ref{kmcaHVIBEWòonvcàwkoplèò}, we deduce that
\[
\begin{split}
&\frac{\kappa}{\mu}\left(\kappa(\overline{u}-\widetilde u)-\frac{\lambda}{\mu}\right)\dot h(t_0)\\
&\quad\ge \left(\kappa(\overline{u}-\widetilde u)-\frac{\lambda}{\mu}\right)^2 h(t_0)\\
&\qquad-\left(\kappa(\overline{u}-\widetilde u)-\frac{\lambda}{\mu}\right)\left(\left(\lambda\widetilde u +\dfrac{C_0}{\eta}+\frac{1}{\varepsilon^2}\left(2\overline u +\dfrac{2\varepsilon(\overline u -\underline u)}{\eta} \right)\right)\int_0^{t_0} h(\tau)\,d\tau+\overline\gamma\right)\\
&\quad=\left(\kappa(\overline{u}-\widetilde u)-\frac{\lambda}{\mu}\right)^2 h(t_0)-\frac12\left(\kappa(\overline{u}-\widetilde u)-\frac{\lambda}{\mu}\right)^2 h(t_0)\\
&\quad=\frac12\left(\kappa(\overline{u}-\widetilde u)-\frac{\lambda}{\mu}\right)^2 h(t_0)\\
&\quad >\frac{2\kappa}{\mu}\left(\lambda\widetilde u +\dfrac{C_0}{\eta}+\frac{1}{\varepsilon^2}\left(2\overline u +\dfrac{2\varepsilon(\overline u -\underline u)}{\eta} \right)\right) h(t_0),
\end{split}
\]
which is in contradiction with~\eqref{cfyivgygobhiuajnvòk}. This establishes the desired claim in~\eqref{claim2uvlghbiks}.

Now, by using~\eqref{claim2uvlghbiks} in~\eqref{vycgiuobhIJNPCKL}, we conclude that
\[
\frac{\kappa}{\mu}\dot h(t)\ge\frac12\left(\kappa(\overline{u}-\widetilde u)-\frac{\lambda}{\mu}\right) h(t)\quad\mbox{ for any } t> 0,
\]
namely, by recalling the definition in~\eqref{hdefn},
\[
s^3(t)\ge \varepsilon^3 + (s^3(0)-\varepsilon^3) e^{\frac{\mu}{2\kappa}\left(\kappa(\overline{u}-\widetilde u)-\frac{\lambda}{\mu}\right)t}\quad\mbox{ for any } t>0.
\]
Recalling that \eqref{bdwipKMEJFà-} is in force, by taking the limit as $t\to +\infty$, we obtain the desired result.
\end{proof}

\begin{appendix}

\section{Examples of initial cell distributions satisfying~\eqref{oiuygtfdoiuy2}, \eqref{ctfyivgoubhjklmò,2} and~\eqref{lpko,mijnhugyftcdrxctfvl,ò}}\label{apprdfctbhjkl–òlkjhbgvfdrdxcfvgbhnjmk,}

Here we present some examples of functions~$u_0(r)$ verifying all the assumptions requested in Theorem~\ref{main2}, namely~\eqref{oiuygtfdoiuy2}, \eqref{ctfyivgoubhjklmò,2} and~\eqref{lpko,mijnhugyftcdrxctfvl,ò}.

\begin{example}\label{Exc1}
We consider the linear initial datum
\[
u_0(r):= \underline u +\dfrac{\overline u -\underline u}{s(0)-\varepsilon}(r-\varepsilon),\quad {\mbox{for any }}r\in [\varepsilon, s(0)],
\]
and we notice that
\[
u_0(\varepsilon)=\underline u \qquad\mbox{ and }\qquad u_0(s(0))=\overline u.
\]
Also, we have that such~$u_0$ satisfies~\eqref{oiuygtfdoiuy2}, since~$u_0\ge \underline u>\frac{\Gamma u_B}{\lambda}$,
thanks to~\eqref{relationutilde2}.

We now prove that \eqref{ctfyivgoubhjklmò,2} is satisfied.
Indeed,
\[
\int_{\varepsilon}^{s(0)} u_0(r) dr = \underline u (s(0)-\varepsilon) +\dfrac{\overline u-\underline u}{s(0)-\varepsilon}\int_{\varepsilon}^{s(0)} (r-\varepsilon) dr=\dfrac{\overline u +\underline u}{2}(s(0)-\varepsilon).
\]
From this and the Chebyshev integral inequality, we deduce that
\[
\int_{\varepsilon}^{s(0)} u_0(r) r^2 dr\ge\dfrac{1}{s(0)-\varepsilon}\left(\int_{\varepsilon}^{s(0)} u_0(r) dr\right)\left(\int_{\varepsilon}^{s(0)} r^2 dr\right)\ge \dfrac{\overline u +\underline u}{2} \int_{\varepsilon}^{s(0)} r^2 dr.
\]
Accordingly,  by recalling \eqref{relationutilde2}, we conclude that
\[
\begin{split}
&\int_\varepsilon^{s(0)} (2 u_0(r)-\overline{u} -\widetilde u) r^2\, dr = \int_\varepsilon^{s(0)} 2u_0(r) r^2\, dr -(\overline u+\widetilde u)\int_\varepsilon^{s(0)} r^2 \, dr\\
&\quad\ge (\overline u+\underline u)\int_\varepsilon^{s(0)} r^2 \, dr - (\overline u+\widetilde u)\int_\varepsilon^{s(0)} r^2 \, dr\\
&\quad = (\underline u-\widetilde u)\int_\varepsilon^{s(0)} r^2 \, dr>0, 
\end{split}
\]
as desired.

We now show that~\eqref{lpko,mijnhugyftcdrxctfvl,ò} is also fulfilled.
For this, we observe that the function
$$w(r):=\dfrac{1-e^{-\sqrt\lambda(r-\varepsilon)}}{1-e^{-\sqrt\lambda (s(0)-\varepsilon)}}$$ is strictly concave in~$[\varepsilon, s(0)]$ and satisfies~$ w(\varepsilon)=0$ and~$w(s(0))=1$.
Therefore,
$$ \dfrac{r-\varepsilon}{s(0)-\varepsilon}\le w(r), \quad {\mbox{for any }}r\in [\varepsilon, s(0)]$$
and this entails~\eqref{lpko,mijnhugyftcdrxctfvl,ò}.
\end{example}

\begin{example}
Let 
\begin{equation}\label{xfgyucvghjcvhj1}
\alpha:=\min\left\{\frac{s(0)}{2},\frac{1}{2},\frac{(\underline u-\widetilde u)(s^3(0)-\varepsilon^3)}{(3\varepsilon^2+3\varepsilon+1)(\overline u+\underline u-2\frac{\Gamma u_B}{\lambda})}\right\}
\end{equation} and~$u_0\in C^\infty([\varepsilon,s(0)])$ be such that 
\[
\begin{aligned}
&u_0(\varepsilon)=\underline u,
\qquad \frac{\Gamma u_B}{\lambda}<\min_{r\in(\varepsilon,\varepsilon+\alpha)}u_0(r)<\underline u, \\
\mbox{and }\quad & u_0(r):=\underline u +\dfrac{\overline u -\underline u}{s(0)-\varepsilon-\alpha}(r-\varepsilon-\alpha),\quad {\mbox{for any }}
r\in [\varepsilon+\alpha, s(0)].
\end{aligned}
\]
By construction, and recalling~\eqref{relationutilde2}, we have that~$u_0>\frac{\Gamma u_B}{\lambda}$, namely~\eqref{oiuygtfdoiuy2} is satisfied.

We now establish \eqref{ctfyivgoubhjklmò,2}. For this, we observe that, by~\eqref{xfgyucvghjcvhj1},
\begin{equation}\label{xfgyucvghjcvhj2}
\begin{aligned}
\int_\varepsilon^{\varepsilon+\alpha}(2u_0(r)-\overline u -\widetilde u)r^2dr&\geq -\frac{1}{3}\left(\overline u+\widetilde u -\frac{\Gamma u_B}{\lambda}\right)[(\varepsilon+\alpha)^3-\varepsilon^3] \\
&>-\frac{\alpha}{3}\left(\overline u+\widetilde u -\frac{\Gamma u_B}{\lambda}\right)(3\varepsilon^2+3\varepsilon+1).
\end{aligned}
\end{equation}
Moreover, we have
\[
\begin{aligned}
\int_{\varepsilon+\alpha}^{s(0)} u_0(r)\,dr &=\underline u (s(0)-\varepsilon-\alpha)+\frac{\overline u-\underline u}{s(0)-\varepsilon-\alpha}\int_{\varepsilon+\alpha}^{s(0)} (r-\varepsilon-\alpha)\,dr \\
&=\frac{(\overline u+\underline u)(s(0)-\varepsilon-\alpha)}{2}.
\end{aligned}
\]
Hence
\[
\int_{\varepsilon+\alpha}^{s(0)} (2u_0(r)-\overline u -\widetilde u)\,dr = (\underline u-\widetilde u)(s(0)-\varepsilon-\alpha).
\]
From this, the Chebyshev integral inequality and \eqref{xfgyucvghjcvhj1} we infer that
\begin{equation}\label{xfgyucvghjcvhj3}
\begin{aligned}
\int_{\varepsilon+\alpha}^{s(0)} (2u_0(r)-\overline u -\widetilde u)r^2dr &\ge\dfrac{1}{s(0)-\varepsilon-\alpha}\left(\,\int_{\varepsilon+\alpha}^{s(0)} (2 u_0(r)-\overline u-\widetilde u) dr\right)\left(\,\int_{\varepsilon+\alpha}^{s(0)} r^2 dr\right)\\
&\ge \frac{(\underline u -\widetilde u)(s^3(0)-(\varepsilon+\alpha)^3)}{3} \\
&>\frac{(\underline u -\widetilde u)(s^3(0)-\varepsilon^3-\alpha(3\varepsilon^2+3\varepsilon+1))}{3}.
\end{aligned}
\end{equation}
Combining \eqref{xfgyucvghjcvhj2} and \eqref{xfgyucvghjcvhj3} we get
\[
\int_{\varepsilon}^{s(0)} (2u_0(r)-\overline u -\widetilde u)r^2dr >  \frac{1}{3}\left[(\underline u -\widetilde u)(s^3(0)-\varepsilon^3)-\alpha(3\varepsilon^2+3\varepsilon+1)\left(\overline u+\underline u -\frac{\Gamma u_B}{\lambda}\right)\right].
\]
This together with \eqref{xfgyucvghjcvhj1} implies \eqref{ctfyivgoubhjklmò,2}, as desired.

Furthermore, since $u_0$ here is smaller than the one in Example \ref{Exc1}, we infer that \eqref{lpko,mijnhugyftcdrxctfvl,ò} holds true as well.
\end{example}

\end{appendix}

\section*{Acknowledgements} 
SD, CS and EV are members of the Australian Mathematical Society (AustMS). CS and EPL are members of the INdAM--GNAMPA.

CS acknowledges the support of the Juan de la Cierva Fellowship (grant number JDC2023-050365-I).

This work has been supported by the Australian Laureate Fellowship FL190100081 and by the Australian Future Fellowship FT230100333.

\vfill

\end{document}